\documentclass[onefignum,onetabnum]{siamart190516}
\usepackage{lipsum}
\usepackage{amsfonts}
\usepackage{graphicx}
\usepackage{epstopdf}
\usepackage{enumitem}
\usepackage{algorithmic}
\ifpdf
  \DeclareGraphicsExtensions{.eps,.pdf,.png,.jpg}
\else
  \DeclareGraphicsExtensions{.eps}
\fi

\usepackage{graphicx}
\usepackage{color}
\usepackage{subcaption}

\usepackage{amsmath,amsfonts,amssymb,stmaryrd} % Math packages
\usepackage{booktabs}
\usepackage{xspace}
\usepackage[normalem]{ulem}
\usepackage[misc]{ifsym}

\newcommand{\hilbert}{\mathcal{H}}
\newcommand{\Q}{\mathcal{Q}}
\newcommand{\M}{\mathcal{M}}

\newcommand{\Li}{\mathcal{L}}
\newcommand{\C}{\mathcal{C}}

\newcommand{\reals}{\mathbb{R}}
\newcommand{\inner}[2]{{\langle #1, #2\rangle}}
\newcommand{\normsq}[1]{{\lVert #1\rVert ^2}}
\newcommand{\norm}[1]{{\lVert #1\rVert}}
\DeclareMathOperator*{\minimize}{minimize}
\DeclareMathOperator*{\maximize}{maximize}
\DeclareMathOperator*{\tr}{Tr}

\newtheorem{rem}{Remark}
\newtheorem{prop}{Proposition}
\newtheorem{fact}{Fact}

\headers{Operator Splitting Performance Estimation}{Ryu, Taylor, Bergeling, and Giselsson}

\title{Operator Splitting Performance Estimation: Tight contraction factors and optimal parameter selection}

\author{Ernest K. Ryu\thanks{Department of Mathematical Sciences, Seoul National University, Seoul, Korea \email{ernestryu@snu.ac.kr}}
\and Adrien B. Taylor\thanks{INRIA, D\'epartement d'informatique de l'ENS, \'Ecole normale sup\'erieure, CNRS, PSL Research University, Paris, France \email{adrien.taylor@inria.fr}}
\and Carolina Bergeling\thanks{Department of Automatic Control, Lund University, Lund, Sweden \email{carolina.bergeling@control.lth.se}, \email{pontusg@control.lth.se}}
\and  Pontus Giselsson\footnotemark[4]
}

\usepackage{amsopn}

\ifpdf
\hypersetup{
  pdftitle={Operator Splitting Performance Estimation Problem},
  pdfauthor={E. K. Ryu, A. B. Taylor, C. Bergeling, and P. Giselsson}
}
\fi

\begin{document}

\maketitle

% REQUIRED
\begin{abstract}
We propose a methodology for studying the performance of common splitting methods through semidefinite programming. We prove tightness of the methodology and demonstrate its value by presenting two applications of it. First, we use the methodology as a tool for computer-assisted proofs to prove tight analytic contraction factors for Douglas--Rachford splitting that are likely too complicated for a human to find bare-handed. Second, we use the methodology as an algorithmic tool to computationally select the optimal splitting method parameters by solving a series of semidefinite programs.
\end{abstract}
%
%% REQUIRED
%\begin{keywords}
%  Computer-aided analyses, First-order methods, Rates of convergence, Monotone operators, Splitting methods
%\end{keywords}
%
%% REQUIRED
%\begin{AMS}
%47H05, 47H09, 68Q25, 90C22, 90C25, 90C60
%\end{AMS}

\section{Introduction}
Consider the fixed-point iteration in a real Hilbert space $\hilbert$ $$z^{k+1}=Tz^k,$$ where $T\colon\hilbert\rightarrow\hilbert$.
%, where $T$ is a fixed-point mapping of a splitting method.
We say $\rho<1$ is a contraction factor of $T$ if $$\|Tx-Ty\|\le \rho \|x-y\|$$ for all $x,y\in \hilbert$.
% Since a contraction factor (obtained through a proof) establishes a linear rate of convergence of the iteration,
We ask the question: given a set of assumptions, what is the best (tight) contraction factor one can prove?
In this work, we present the operator splitting performance estimation problem (OSPEP), a methodology for studying contraction factors of forward-backward splitting (FBS), Douglas--Rachford splitting (DRS), and Davis--Yin splitting (DYS).

% Following the technique of Drori and Teboulle \cite{drori2014performance}, we pose the problem of finding the best contraction factor as a semidefinite program (SDP).

First, we present the OSPEP problem, the infinite-dimensional non-convex optimization problem of finding the best (smallest) contraction factor given a set of assumptions on the operators.
Following the technique of Drori and Teboulle \cite{drori2014performance}, we reformulate the problem into a finite-dimensional convex semidefinite program (SDP).
We then establish tightness (exactness) of this reformulation with interpolation conditions.

%Since no relaxation is performed, the optimal value of the SDP is the best contraction factor.
%that is independent of the dimension of $\hilbert$.

Next, we demonstrate the value of OSPEP through two uses.
First, we use OSPEP as a tool for computer-assisted proofs to prove tight analytic contraction factors for DRS.
The results are tight in that they have exact matching lower bounds.
The proofs are computer-assisted in that their discoveries were assisted by a computer, but their verifications do not require a computer.
%These results are likely too complicated for a human to find bare-handed.
Second, we use OSPEP as an algorithmic tool to automatically select the optimal splitting method parameters.

The tightness guarantee and flexibility make OSPEP a powerful tool.
Due to tightness, OSPEP can provide both positive and negative results.
The flexibility allows users to pick and choose assumptions from a set of standard assumptions.

% proof of Sections~\ref{s:main-proof} and \ref{ss:verification} 

\subsection{Organization and contribution}
Section~\ref{s:operator_interp} presents operator interpolation, later used in Section~\ref{sec:PEP} to establish tightness.
Section~\ref{sec:PEP} presents the OSPEP methodology, an exact transformation of the problem of finding the best contraction factor into a convex SDP, and provides tightness guarantees.
Section~\ref{sec:tight_DRS} presents tight analytic contraction factors for DRS under assumptions considered in \cite{giselsson2017tight,moursi2018douglas} using OSPEP as a tool for computer-assisted proofs.
Section~\ref{sec:parameter_selection} presents an automatic parameter selection method using OSPEP as an algorithmic tool.
Section~\ref{sec:conclusion} concludes the paper.

The main contribution of this work is twofold.
The first is analyzing the performance of monotone splitting methods using SDPs \textbf{with tightness guarantees}.
The overall formulation generally follows from the technique of Drori and Teboulle \cite{drori2014performance} and the prior work discussed in Section~\ref{ss:prior_work}.
The tightness, established with the operator interpolation results of Sections~\ref{s:operator_interp}, is a novel theoretical contribution.
The second contribution is the techniques of Sections~\ref{sec:tight_DRS} and \ref{sec:parameter_selection}, an illustration of how to use the proposed methodology.
Although we do consider the results of Sections~\ref{sec:tight_DRS} and \ref{sec:parameter_selection} to be interesting and valuable, we view the technique, rather than the result, to be the second major contribution.

The major and minor contributions of this work are, to the best of our knowledge, novel in the following sense.
The tightness of Section~\ref{sec:PEP} is new.
% Analyzing the performance of monotone operator splitting methods, rather than convex optimization algorithms, with any form of computer-assisted proof is new.
% Past work in the performance estimation problem literature analyzed convex optimization algorithms, and, to the best of our knowledge, 
The technique of Section~\ref{sec:tight_DRS} is the first use of computer-assisted proofs to obtain provably tight rates for monotone operator splitting methods.
The tight results of Section~\ref{sec:tight_DRS} improve upon the prior results of \cite{giselsson2017tight,moursi2018douglas}.
The technique of Section~\ref{sec:parameter_selection} is the first use of automatic parameter selection that is optimal with respect to the algorithm and assumptions.

\subsection{Prior work}
\label{ss:prior_work}
FBS was first stated in the operator theoretic language in \cite{bruck1977,passty1979}.
The projected gradient method presented in \cite{goldstein1964,levitin1966} served as a precursor to FBS.
Peaceman-Rachford spitting (PRS) was first presented in \cite{peaceman1955,kellogg1969,lions1979},
and DRS was first presented in \cite{douglas1956,lions1979}.
DYS was first presented in \cite{Davis2017}.
Forward-Douglas--Rachford splitting of Raguet, Fadili, Peyr\'e, and Brine\~no-Arias \cite{raguetGFBS,BAFDRS,Raguet2018} served as a precursor to DYS.

What we call interpolation in this work is also called extension.
% Extension results for operators have been studied before.
The maximal monotone extension theorem, which we later state as Fact~\ref{prop:mm_ext}, is well known, and it follows from a standard application of Zorn's lemma.
Reich \cite{REICH1977378}, Bauschke \cite{bauschke2007}, Reich and Simons \cite{reich2005}, Bauschke, Wang, and Yao \cite{kvextension2010,kvextension2_2010,Wang2013}, and Crouzeix and Anaya \cite{cruozeix2007,CROUZEIX2010,CROUZEIX20102}
have studied more concrete and constructive extension theorems for maximal monotone, nonexpansive, and firmly-nonexpansive operators using tools from monotone operator theory.

% Bauschke maximal FNE exntension theorem via monotone operator theory
%  more explicit maximal monotone extension theory
% Reich and Simons alternative constructions to KV theorem
%  specific form of the maximal monotone extension theorem.

Contraction factors and linear convergence for first-order methods have been a subject of intense study.
Surprisingly, many of the published contraction factors are not tight. 
For FBS,
Mercier,
\cite[p.\ 25]{mercier1980inequations},
Tseng \cite{tseng1991},
Chen and Rockafellar \cite{chen1997}, and
Bauschke and Combettes \cite[Section 26.5]{BauschkeCombettes2017_convex} proved linear rates of convergence, but did not provide exact matching lower bounds.
Taylor, Hendrickx, and Glineur showed tight contraction factors and provided exact matching lower bounds \cite{taylor2018pgm}.
For DRS,
Lions and Mercier \cite{lions1979} and Davis and Yin \cite{davis_rates}
proved linear rates of convergence, but did not provide exact matching lower bounds.
Giselsson and Boyd \cite{giselsson2014diagonal,giselsson2017linear}, Giselsson \cite{giselsson2015tight,giselsson2017tight}, and  Moursi and Vandenberghe \cite{moursi2018douglas}
proved linear rates of convergence and provided exact matching lower bounds for certain cases.
ADMM is a splitting method closely related to DRS.
Deng and Yin \cite{Deng2016}, Giselsson and Boyd \cite{giselsson2014diagonal,giselsson2017linear}, Nishihara et al.\ \cite{nishihara15},
Fran{\c{c}}a and Bento~\cite{francca2016explicit},
Hong and Luo \cite{Hong2017},
Han, Sun, and Zhang \cite{d_sun_admm1}, and
Chen et al.\ \cite{d_sun_admm2}
proved linear rates of convergence for ADMM. Matching lower bounds are provided only in \cite{giselsson2017linear}. Further, \cite{giselsson2015tight} provides matching lower bounds to the rates in \cite{giselsson2014diagonal}.
Ghadimi et al.\ \cite{GHADIMI2012139,Ghadimi2015} and 
Teixeira et al.\ \cite{johansson_admm,johansson_admm2}
proved linear rates of convergence and provided matching lower bounds for ADMM applied to quadratic problems.
For DYS, Davis and Yin \cite{Davis2017}, Yan \cite{Yan2018},  Pedregosa and Gidel \cite{pedregosa2018}, and 
Pedregosa, Fatras, and Casotto \cite{pedegrosa_3op_var}
proved linear rates of convergence, but did not provide exact matching lower bounds.
Pedegrosa \cite{pedregosa2016} analyzed sublinear convergence, but not contraction factors.

Analyzing convex optimization algorithms by formulating the analysis as an SDP has been a rapidly growing area of research in the past 5 years.
Past work analyzed convex optimization algorithms, and, to the best of our knowledge, analyzing the performance of monotone operator splitting methods with SDPs or any form of computer-assisted proof is new.
% \AT{mb add that since the first version of this work, a few others followed? or mb in the cover letter?}.
(After the initial version of this paper was made public on arXiv, several papers citing our work followed up on our results and used SDPs to analyze other monotone operator splitting methods \cite{gu_yang_3,gu_yang_1,gu_yang_2,8736360,preciado2019,Ryu2019}.)
Drori and Teboulle \cite{drori2014performance} and Taylor, Hendrickx, and Glineur \cite{taylor2017exact,taylor2017smooth} presented the performance estimation problem (PEP) methodology.
%which formulates the problem of finding the worst-case instance of a given optimization method as an SDP.
Our work generally follows the techniques presented by Drori and Teboulle \cite{drori2014performance} while contributing by establishing tightness.
Lieder \cite{lieder_halpern} applied the PEP approach to analyze the Halpern iteration without an a priori guarantee of tightness.
Lessard, Recht, and Packard \cite{lessard2016analysis} leveraged techniques from control theory and used integral quadratic constraints (IQC) for finding Lyapunov functions for analyzing convex optimization algorithms.
The IQC and PEP approaches were recently linked by Taylor, Van Scoy, and Lessard \cite{taylor2018lyapunov}. Finally, Nishihara et al.\ \cite{nishihara15} and Fran{\c{c}}a and Bento \cite{francca2016explicit} used IQC to the analyze ADMM.

Finally, both IQC and PEP approaches allowed designing new methods for particular problem settings. For example, the optimized gradient method by Kim and Fessler~\cite{kim2016optimized,Kim2017,Kim2018,doi:10.1137/16M108940X,doi:10.1137/17M112124X,kim_arxiv}
 (first numerical version by Drori and Teboulle~\cite{drori2014performance}) was developed using PEPs and enjoys the best possible worst-case guarantee on the final objective function accuracy after a fixed number of iteration, as showed by Drori~\cite{drori2017exact}.  On the other hand, the IQC framework was used by Van Scoy et al.~\cite{van2018fastest} for developing the triple momentum method, the first-order method with the fastest known convergence rate for minimizing a smooth strongly convex function.

\subsection{Preliminaries}
We now quickly review standard results and set up the notation.
We follow standard notation \cite{ryu2016,BauschkeCombettes2017_convex}.
Write $\hilbert$ for a real Hilbert space equipped with a (symmetric) inner product $\langle \cdot,\cdot\rangle$.
% We say a symmetric matrix $M\in \reals^{n\times n}$ is positive semidefinite if $x^T Mx\ge 0$ for all $x\in\reals^n$.
Write $\mathbb{S}_+^{n}$ for the set of $n\times n$ symmetric positive semidefinite matrices.
Write $M\succeq 0$ if and only if $M\in \mathbb{S}_+^{n}$.

We say $A$ is an operator on $\hilbert$ and write $A\colon\hilbert\rightrightarrows\hilbert$ if $A$ maps a point in $\hilbert$ to a subset of $\hilbert$.
So $A(x)\subset\hilbert$ for all $x\in \hilbert$.
For simplicity, we also write $Ax=A(x)$.
% If $A$ always maps a point to a singleton, we say $A$ is single-valued and write $Ax=y$ instead of $Ax=\{y\}$.
Write $I\colon\hilbert\rightarrow\hilbert$ for the identity operator.
We say $A\colon\hilbert\rightrightarrows\hilbert$ is monotone if
\[
\langle Ax-Ay,x-y\rangle \ge 0
\]
for all $x,y\in \hilbert$.
To clarify, the inequality means $\langle u-v,x-y\rangle\ge 0$ for all $u\in Ax$ and $v\in Ay$.
We say $A\colon\hilbert\rightrightarrows\hilbert$ is $\mu$-strongly monotone if
\[
\langle Ax-Ay,x-y\rangle \geq \mu\|x-y\|^2,
\]
where $\mu\in (0,\infty)$.
We say a single-valued operator $A\colon\hilbert\to\hilbert$ is $\beta$-cocoercive if
\[
\langle Ax-Ay,x-y\rangle\ge \beta\|Ax-Ay\|^2,
\]
where $\beta\in(0,\infty)$.
We say a single-valued operator $A\colon\hilbert\to\hilbert$ is $L$-Lipschitz if
\[
\|Ax-Ay\|\le L\|x-y\|
\]
where $L\in (0,\infty)$.
A monotone operator is maximal if it cannot be properly extended to another monotone operator.
%, i.e., there is no other monotone operator whose graph properly contains its graph.
The resolvent of an operator $A$ is $J_{\alpha A}=(I+\alpha A)^{-1}$, where $\alpha>0$.
% The resolvent of an operator $A$ is the single-valued operator $J_{\alpha A}=(I+\alpha A)^{-1}$, defined as
% \[
% (I+\alpha A)^{-1}(u)=v\quad\Leftrightarrow\quad u\in v+\alpha Av,
% \]
% where $\alpha>0$.
% We say a single-valued operator $T\colon\hilbert\rightarrow\hilbert$ is nonexpansive if
% \[
% \|Tx-Ty\|\le \|x-y\|
% \]
% for all $x,y\in \hilbert$.
 We say a single-valued operator $T\colon\hilbert\rightarrow\hilbert$ is contractive if it is $\rho$-Lipschitz with $\rho<1$.
% \[
% \|Tx-Ty\|\le\rho \|x-y\|
% \]
% for all $x,y\in \hilbert$.
We say $x^\star$ is a fixed point of $T$ if $x^\star=Tx^\star$.
% If $x^\star$ is a fixed point and $T$ is contractive, then we have
% \[
% \|Tx-x^\star\|\le \rho\|x-x^\star\|
% \]
% for all $x\in \hilbert$.

Davis--Yin splitting (DYS) encodes solutions to
\[
\begin{array}{ll}
\underset{x\in \hilbert}{\mbox{find}}&0\in (A+B+C)x
\end{array}
\]
where $A$, $B$, and $C$ are maximal monotone and $C$ is single-valued, as fixed points of
\begin{equation}
T(z;A,B,C,\alpha,\theta)=z-\theta J_{\alpha B}z+\theta J_{\alpha A}(2J_{\alpha B}-I-\alpha CJ_{\alpha B})z
\label{eq:dys}
\end{equation}
where $\alpha>0$ and $\theta\ne 0$.
FBS and DRS are special cases of DYS; when $C = 0$ DYS reduces to DRS, and when $B = 0$ DYS reduces to FBS.
Therefore, our analysis on DYS directly applies to FBS and DRS.

\section{Operator interpolation}
\label{s:operator_interp}
Let $\Q$ be a class of operators, and let $\mathcal{I}$ be an arbitrary index set. We say a set of duplets
$\{(x_i , q_i)\}_{i\in \mathcal{I}}$, where $x_i,q_i\in\hilbert$ for all $i \in \mathcal{I}$, is $\Q$-{\emph{interpolable}} if there is an operator $Q \in \Q$
such that $q_i \in Qx_i$ for all $i \in \mathcal{I}$. In this case, we call $Q$ an {\emph{interpolation}} of $\{(x_i , q_i )\}_{i\in \mathcal{I}}$. In
this section, we present conditions that characterize when a set of duplets is interpolable
with respect to the class of operators listed in Table~\ref{tab:classes_summary} and their intersections.

\begin{table}[!h]
	\begin{center}
		{\renewcommand{\arraystretch}{1.3}
			\begin{tabular}{@{}ll@{}}
				\specialrule{2pt}{1pt}{1pt}
				\bfseries Class & \bfseries Description\\
				\hline
				$\M$ & maximal monotone operators \\ 
				$\M_{\mu}$ & $\mu$-strongly monotone maximal monotone operators\\
    $\Li_L$ & $L$-Lipschitz operators\\
    $\C_{\beta}$ & $\beta$-cocoercive operators\\
				\specialrule{2pt}{1pt}{1pt}
			\end{tabular}
		}
\caption{Operator classes for which we analyze interpolation. The parameters $\mu$, $L$, and
$\beta$  are in $(0, \infty)$. Note that $\M_{\mu} \subset \M$ for any $\mu > 0$, $\C_{\beta}  \subset \M$ for any $\beta  > 0$, but $\Li_L \not\subset \M$
for any $L > 0$.}
\label{tab:classes_summary}
\end{center}
\end{table}

\subsection{Interpolation with one class}
\label{sec:one_class_interp}
We now present interpolation results for the classes $\M$, $\M_{\mu}$, $\Li_L$, and $C_{\beta}$.
\begin{fact}[Maximal monotone extension theorem {\cite[Theorem 20.21]{BauschkeCombettes2017_convex}}]
$\{(x_i , q_i )\}_{i\in \mathcal{I}}$ is $\M$-interpolable if
and only if
\begin{align*}
\langle q_i - q_j, x_i - x_j  \rangle \geq 0 \qquad \forall i, j \in  \mathcal{I}.
\end{align*}
\label{prop:mm_ext}
\end{fact}

\begin{prop} Let $\mu \in  (0, \infty )$. Then $\{(x_i,q_i)\}_{i\in \mathcal{I}}$ is $M_{\mu}$-interpolable if and only if
\begin{align*}
\langle q_i - q_j,x_i - x_j \rangle \geq \mu \|x_i - x_j\|^2\qquad \forall i, j \in \mathcal{I}.
\end{align*}
\label{prop:msm_ext}
\end{prop}
\begin{proof}
With Fact~\ref{prop:mm_ext}, the proof follows from a sequence of equivalences:
\begin{align*}
\forall i,j \in  \mathcal{I}, \,\langle q_i - q_j,x_i - x_j \rangle &\geq \mu\|x_i - x_j\|^2\\
&\quad\Leftrightarrow\quad \forall i,j\in  \mathcal{I}, \langle (q_i - \mu x_i) - (q_j - \mu x_j),x_i - x_j\rangle \geq  0\\
&\quad\Leftrightarrow\quad  \exists R \in  \M, \forall i \in \mathcal{I}, (q_i - \mu x_i) \in  Rx_i\\
&\quad\Leftrightarrow\quad  \exists Q \in  \M_{\mu}, Q = R + \mu I, \,\forall i \in  \mathcal{I}, \,q_i \in  Qx_i.
\end{align*}
\end{proof}
\begin{prop}
Let $\beta  \in  (0, \infty)$. Then $\{(x_i,q_i)\}_{i\in \mathcal{I}}$ is $\C_{\beta}$-interpolable if and only if
\begin{align*}
\langle q_i - q_j ,x_i - x_j\rangle \geq  \beta \|q_i - q_j\|^2 \qquad \forall  i, j \in  \mathcal{I}.
\end{align*}
\label{prop:coco_ext}
\end{prop}
\begin{proof}
With Proposition~\ref{prop:msm_ext}, the proof follows from a sequence of equivalences:
\begin{align*}
\forall i,j \in  \mathcal{I},\, \langle q_i - q_j,x_i - x_j\rangle& \geq  \beta \|q_i - q_j\|^2\\
&\quad\Leftrightarrow\quad  \exists R \in  \M_\beta, \forall i \in  \mathcal{I},\, x_i \in  Rq_i\\
&\quad\Leftrightarrow\quad  \exists Q \in  \C_{\beta}, Q = R^{-1} , \,\forall i \in  \mathcal{I},\, q_i \in  Qx_i.
\end{align*}
\end{proof}

\begin{fact}[Kirszbraun--Valentine Theorem] 
\label{fact:kvtheorem}
Let $L \in (0, \infty)$.
Then $\{(x_i,q_i)\}_{i\in \mathcal{I}}$ is $\Li_L$-interpolable if and only if
\begin{align*}
\|q_i - q_j\|^2 \le L^2 \|x_i - x_j\|^2 \qquad \forall  i, j \in  \mathcal{I}.
\end{align*}
\label{prop:Lip_ext}
\end{fact}
Fact~\ref{fact:kvtheorem} is a special case of the Kirszbraun--Valentine theorem \cite{Kirszbraun1934,valentine1943,valentine1945}.
A direct proof follows from similar arguments.

\subsection{Failure of interpolation with intersection of classes}
\label{sec:interp_fail}
% So far, we have discussed interpolation conditions for the four individual classes, but
When considering interpolation with intersections of classes such as $\M \cap \Li_L$, one
might naively expect results as simple as those of Section~\ref{sec:one_class_interp}. Contrary to this
expectation, interpolation can fail.
\begin{prop}
$\{(x_i,q_i)\}_{i\in \mathcal{I}}$ may not be $(\M \cap  \Li_L)$-interpolable for $L \in  (0, \infty)$ even if
\begin{align*}
\|q_i - q_j\|^2\le L^2\|x_i - x_j\|^2 ,\qquad
\langle q_i - q_j, x_i - x_j\rangle \geq 0\qquad
\forall  i, j \in  \mathcal{I}.
\end{align*}
\label{prop:interp_fail}
\end{prop}
\begin{proof}
Consider the following example in $\reals^2$:
\begin{align*}
S = \left\{\left(\begin{bmatrix}0\\0\end{bmatrix},\begin{bmatrix}0\\0\end{bmatrix}\right),\left(\begin{bmatrix}1\\0\end{bmatrix},\begin{bmatrix}0\\0\end{bmatrix}\right),\left(\begin{bmatrix}1/2\\0\end{bmatrix},\begin{bmatrix}0\\L/2\end{bmatrix}\right)\right\}.
\end{align*}
These points satisfy the inequalities. %Proposition~\ref{prop:interp_fail}.
% By Fact~\ref{prop:mm_ext}, these points can be interpolated with a maximal monotone operator. By Fact~\ref{prop:Lip_ext}, these points can be interpolated with an $L$-Lipschitz operator.
However, there is no Lipschitz and maximal monotone operator interpolating these
points. Assume for contradiction that $Q \in (\M \cap \Li_L)$ is an interpolation of these points.
Since $Q$ is Lipschitz, it is single-valued. Since $Q$ is maximal monotone, the set $\{x \,|\, Qx = 0\}$
is convex \cite[Proposition 23.39]{BauschkeCombettes2017_convex}. This implies $Q(1/2, 0) = (0,0)$, which is a contradiction.
\end{proof}

The subtlety is that the counterexample has two separate interpolations in $\M$ and $\Li_L$
but does not have an interpolation in $\M \cap \Li_L$.
Interpolation with respect to $\M_{\mu}\cap \Li_L$, $\C_{\beta} \cap \Li_L$, and $\M_{\mu} \cap \C_{\beta}$ can fail in a similar manner.

% It is possible to show that interpolation with respect to $\M_{\mu}\cap \Li_L$, $\C_{\beta} \cap \Li_L$, and $\M_{\mu} \cap \C_{\beta}$ can fail in a similar manner.

\subsection{Two-point interpolation}
\label{sec:two_point_interp}
We now present conditions for two-point interpolation, i.e., interpolation when $|\mathcal{I}| = 2$. In
this case, interpolation conditions become simple, and the difficulty discussed in Section~\ref{sec:interp_fail} disappears. Although the setup $|\mathcal{I}| = 2$ may seem restrictive, it is sufficient for what we need in later sections.
\begin{prop}
Assume $0 < \mu$, $\mu \leq L < \infty$, and $\mu  \leq  1/\beta  < \infty$. Then $\{(x_1, q_1),(x_2,q_2)\}$ is $(\M_{\mu} \cap \C_{\beta} \cap  \Li_L)$-interpolable if and only if
\begin{align}
\nonumber\langle q_1 - q_2,x_1 - x_2\rangle \geq  \mu \|x_1 - x_2\|^2\\
\label{eq:two_point_MCL}\langle q_1 - q_2,x_1 - x_2\rangle \geq  \beta \|q_1 - q_2\|^2\\
\nonumber
\|q_1 - q_2\|^2\le  L^2 \|x_1 - x_2\|^2 .
\end{align}
\label{prop:two_point_MCL}
\end{prop}
\begin{proof}
If the points are $(\M_{\mu} \cap  \Li_L  \cap  \C_{\beta})$-interpolable, then \eqref{eq:two_point_MCL} holds by definition. Assume \eqref{eq:two_point_MCL} holds. When $\dim\hilbert = 1$ the result is trivial, so we assume, without loss of generality, $\dim\hilbert \geq 2$.

Define
$q = q_1 - q_2$ and $x = x_1 - x_2$.
If $x = 0$, then $\beta>0$ or $L>0$ implies $q=0$, 
and the operator $Q : \hilbert \to  \hilbert$ defined as
\begin{align*}
Q(y) = \mu (y - x_1 ) + q_1
\end{align*}
interpolates $\{(x_1,q_1),(x_2,q_2)\}$
and $Q\in \M_{\mu}  \cap  \Li_L  \cap \C_{\beta}$.
 Assume $x \neq 0$. If $q = \gamma x$ for some $\gamma \in  \reals$, then the operator $Q : \hilbert \to  \hilbert$ defined as
\begin{align*}
Q(y) = \gamma (y - x_1 ) + q_1
\end{align*}
interpolates $\{(x_1,q_1),(x_2,q_2)\}$
and $Q\in \M_{\mu}  \cap  \Li_L  \cap \C_{\beta}$.
Assume $q$ is linearly independent from $x$.
Define the orthonormal vectors,
\begin{align*}
e_1 =\frac{1}{\|x\|}x,\qquad
e_2 = \frac{1}{\sqrt{\|q\|^2-(\langle e_1,q\rangle)^2}}(q - \langle e_1,q\rangle e_1),
\end{align*}
along with an associated bounded linear operator $A : \hilbert \to \hilbert$ such that
\begin{align*}
A|_{\{e_1,e_2\}^\perp} = \mu I,
\end{align*}
where $\{e_1,e_2\}^\perp \subset \hilbert$ is the subspace orthogonal to $e_1$ and $e_2$ and $I$ is the identity mapping
on $\{e_1,e_2\}^\perp$. On ${\rm{span}}\{e_1,e_2\}$, define
\begin{align*}
Ae_1 &=\frac{\langle q,e_1\rangle}{\|x\|}e_1+\frac{\sqrt{\|q\|^2 - (\langle e_1,q\rangle)^2}}{\|x\|}e_2,\\
Ae_2 &= -\frac{\sqrt{\|q\|^2 - (\langle e_1,q\rangle)^2}}{\|x\|}e_1+\frac{\langle q,e_1\rangle}{\|x\|}e_2.
\end{align*}
Note that this definition satisfies $Ax = q$. Finally, define $M$ to be a $2 \times 2$ matrix isomorphic to $A|_{{\rm{span}}\{e_1,e_2\}}$, i.e.,
\begin{align*}
A|_{{\rm{span}}\{e_1,e_2\}} \cong\underbrace{\frac{1}{\|x\|}\begin{bmatrix}
\langle q,e_1\rangle & -\sqrt{\|q\|^2 - (\langle e_1,q\rangle)^2}\\
\sqrt{\|q\|^2 - (\langle e_1,q\rangle)^2}& \langle q,e_1\rangle
\end{bmatrix}}_{=M}\in\reals^{2\times 2}.
\end{align*}
With direct computations, we can verify that $M$ satisfies
\begin{align*}
L^2 &\geq  \lambda_{\max} (M^TM) = \frac{\|q\|^2}{\|x\|^2},\\
\mu  &\leq  \lambda_{\min} ((1/2)(M + M^T)) = \frac{\langle q, x\rangle}{\|x\|^2},\\
\beta &\leq  \lambda_{\min} ((1/2)(M^{-1} + M^{-T} )) =\frac{\langle q, x\rangle}{\|q\|^2}.
\end{align*}
This implies $A : \hilbert \to  \hilbert$ is $L$-Lipschitz, $\mu$-strongly monotone, and $\beta$-cocoercive. Finally,
the affine operator $Q : \hilbert \to  \hilbert$ defined as
\begin{align*}
Q(y) = A(y - x_1 ) + q_1
\end{align*}
interpolates $\{(x_1,q_1),(x_2,q_2)\}$
and $Q\in \M_{\mu}  \cap  \Li_L  \cap \C_{\beta}$.
\end{proof}

Proposition~\ref{prop:two_point_MCL} presents conditions for interpolation with 3 classes. Interpolation conditions with 2 of these classes,
such as
$(\C_{\beta} \cap \Li_L)$,
$(\M_{\mu} \cap \C_{\beta})$,
$(\M_{\mu} \cap \Li_L)$,
$(\M \cap \Li_L)$, and 
are of the same form and follow from the a very similar (identical) proof.

\section{Operator splitting performance estimation problems}
\label{sec:PEP}
% \AT{at this point I am not sure that the reader is convinced that solving (4) actually provides the best possible value for the contraction factor? (see suggestion of adding an overview section in the introduction) --- update: I simply mean that we should explicitly add the link between $\rho$ and (4), and not only in the text (it is in the second paragraph of this section)}
Consider the \emph{operator splitting performance estimation problem} (OSPEP)
\begin{align}
\begin{tabular}{ll}
 $\displaystyle \maximize$ & $\displaystyle\frac{\|T(z;A,B,C,\alpha,\theta)- T(z';A,B,C,\alpha,\theta)\|^2}{\|z-z'\|^2}$\\
 subject to & $A\in\Q_1$, $B\in\Q_2$, $C\in\Q_3$\\
&$z,z'\in\hilbert$, $z\neq z'$
\end{tabular}
\label{eq:ospep_main}
\end{align}
where  $z$, $z'$, $A$, $B$, and $C$ are the optimization variables.
$T$ is the DYS operator defined in \eqref{eq:dys}.
The scalars $\alpha>0$ and $\theta>0$ and the classes $\Q_1$, $\Q_2$, and $\Q_3$ are problem data.
Assume each class $\Q_1$, $\Q_2$, and $\Q_3$ is a single operator class of Table~\ref{tab:classes_summary} or is an intersection of classes of Table~\ref{tab:classes_summary}. (So the reader can freely pick the assumptions; the minimal assumptions are that $\Q_1$, $\Q_2$, and $\Q_3$ are monotone).

%\ER{I tried to add an ``in-math'' connection to $\rho$.}
By definition, $\rho$ is a valid contraction factor if and only if
\[
\rho^2\ge
\sup_{
\substack{A\in \Q_1,B\in \Q_2,C\in \Q_3,\\z,z'\in \hilbert,\,z\ne z'}}
\displaystyle\frac{\|T(z;A,B,C,\alpha,\theta)- T(z';A,B,C,\alpha,\theta)\|^2}{\|z-z'\|^2}.
\]
Therefore, the OSPEP, by definition, computes the square of the best contraction factor of $T$ given the assumptions on $A$, $B$, and $C$, encoded as the classes $\Q_1$, $\Q_2$, and $\Q_3$.
In fact, we say a contraction factor (established through a proof) is \emph{tight} if it is equal to the square root of the optimal value of \eqref{eq:ospep_main}.
A contraction factor that is not tight can be improved with a better proof without any further assumptions.

At first sight, \eqref{eq:ospep_main} seems difficult to solve, as it is posed as an infinite-dimensional non-convex optimization problem.
In this section, we present a reformulation of \eqref{eq:ospep_main} into a (finite-dimensional convex) SDP.
This reformulation is exact; it performs no relaxations or approximations, and the optimal value of the SDP coincides with that of \eqref{eq:ospep_main}.

\subsection{Convex formulation of OSPEP}
\label{ss:conv-ospep-transformation}
We now formulate \eqref{eq:ospep_main} into a (finite-dimensional) convex SDP through a series of equivalent transformations.
% We say two optimization problems are \emph{equivalent} if they have the same optimal values and a solution of one problem can be transformed into a solution of another.
% Our first step of simplification is uses homogeneity of the operator classes.
% The second step uses the two-point operator interpolation of Section~\ref{s:operator_interp}.
% The third step uses the Grammian representation.
First, we write \eqref{eq:ospep_main} more explicitly as
\begin{align}
\begin{tabular}{ll}
 $\displaystyle \maximize$ & $\displaystyle\frac{\|z-\theta(z_B-z_A)-z'+\theta(z_B'-z_A')\|^2}{\|z-z'\|^2}$\\
 subject to & $A\in\Q_1$, $B\in\Q_2$, $C\in\Q_3$\\
 &$z_B=J_{\alpha B}z$\\
 &$z_C=\alpha Cz_B$\\
 &$z_A=J_{\alpha A}(2z_B-z- z_C)$\\
 &$z_B'=J_{\alpha B}z'$\\
 &$z_C'=\alpha Cz_B'$\\
 &$z_A'=J_{\alpha A}(2z_B'-z'- z_C')$\\
& $z,z'\in\hilbert$, $z\neq z'$
\end{tabular}
\label{eq:ospep_main2}
\end{align}
where $z,z'\in \hilbert$, $A$, $B$, and $C$ are the optimization variables.

\subsubsection{Homogeneity}
We say a class of operators $\Q$ is \emph{homogeneous} if 
\[
A\in \Q\quad\Leftrightarrow\quad (\gamma ^{-1}I)A(\gamma I)\in \Q
\]
for all $\gamma>0$.
All operator classes of Table~\ref{tab:classes_summary} are homogeneous.
% On the other hand, optimal point of \eqref{eq:ospep_main2} can be transformed into a feasible point of \eqref{eq:ospep_main3}
Since $\Q_1$, $\Q_2$, and $\Q_3$ are homogeneous,
we can use the change of variables
$z\mapsto \gamma^{-1}z$,
$z'\mapsto \gamma^{-1}z'$,
$A\mapsto (\gamma ^{-1}I)A(\gamma I)$, 
$B\mapsto (\gamma ^{-1}I)B(\gamma I)$, and
$C\mapsto (\gamma ^{-1}I)C(\gamma I)$ where $\gamma=\|z-z'\|$
to equivalently reformulate \eqref{eq:ospep_main2} into 
%\begin{align}
%\begin{tabular}{ll}
%$\displaystyle \maximize$ & $\displaystyle\|z-\theta(z_B-z_A)-z'+\theta(z_B'-z_A')\|^2$\\
%subject to & $A\in\Q_1$, $B\in\Q_2$, $C\in\Q_3$\\
%&$z_B=J_{\alpha B}z$\\
%&$z_C=Cz_B$\\
%&$z_A=J_{\alpha A}(2z_B-z-\alpha z_C)$\\
%&$z_B'=J_{\alpha B}z'$\\
%&$z_C'=Cz_B'$\\
%&$z_A'=J_{\alpha A}(2z_B'-z'-\alpha z_C')$\\
%&$\|z-z'\|^2= 1$
%\end{tabular}
%\label{eq:ospep_main3}
%\end{align}
\begin{align}
\begin{tabular}{ll}
 $\displaystyle \maximize$ & $\displaystyle\|z-\theta(z_B-z_A)-z'+\theta(z_B'-z_A')\|^2$\\
 subject to & $A\in\Q_1$, $B\in\Q_2$, $C\in\Q_3$\\
 &$z_B=J_{\alpha B}z$\\
 &$z_C=\alpha Cz_B$\\
 &$z_A=J_{\alpha A}(2z_B-z-z_C)$\\
 &$z_B'=J_{\alpha B}z'$\\
 &$z_C'=\alpha Cz_B'$\\
 &$z_A'=J_{\alpha A}(2z_B'-z'- z_C')$\\
 &$\|z-z'\|^2= 1$
\end{tabular}
\label{eq:ospep_main3}
\end{align}
where $z,z'\in \hilbert$, $A$, $B$, and $C$ are the optimization variables.

\subsubsection{Operator interpolation}
\label{ss:interp-sdp}
For simplicity of exposition, we limit the generality and reformulate the convex SDP under the following operator classes
\begin{itemize}
    \item $A\in \Q_1=\mathcal{M}_\mu$ --- $\mu$-strongly maximal monotone
    \item $B\in \Q_2=\mathcal{C}_{\beta}\cap \mathcal{L}_L$ --- $\beta$-cocoercive and $L$-Lipschitz 
    \item $C\in \Q_3=\mathcal{C}_{\beta_C}$ --- $\beta_C$-cocoercive
\end{itemize}
%(Note that the cocoercivity and strong monotonicity assumptions imply monotonicity.)
% Other operator classes can be treated in a similar manner.
To clarify, the same analysis can be done in the general setup, and \textbf{we can freely pick and choose the assumptions}.
%(In Section~\ref{sec:tight_DRS}, for example, we consider the assumptions that $C=0$, which reduces the setup to DRS, that $A$ is strongly monotone.)
The general result is shown in the supplementary materials, in Section~\ref{appendix:full_primal_ospep}.

We use the interpolation results from Section~\ref{s:operator_interp}.
For operator $A$, we have
\begin{align*}
    &\exists A\in \mathcal{M}_{\mu}\text{ such that } z_A=J_{\alpha A}(2z_B-z- z_C),\,
    z_A'=J_{\alpha A}(2z_B'-z'- z_C')\\
    &\quad\Leftrightarrow\quad
    \{(z_A,\alpha^{-1}(2z_B-z- z_C-z_A)),(z_A',\alpha^{-1}(2z_B'-z'- z_C'-z_A'))\}\text{ is }\mathcal{M}_{\mu}\text{-interpolable}
    \\
    &\quad\Leftrightarrow\quad
    \langle z_A-z_A',2z_B-z- z_C-(2z_B'-z'- z'_C)\rangle
    \ge 
    (1+\alpha {\mu})\|z_A-z_A'\|^2.
\end{align*}
For operator $B$, we have
\begin{align*}
   \exists B\in\mathcal{C}_{\beta}\cap \mathcal{L}_{L} &\text{ such that }
    z_B=J_{\alpha B}z,\,
    z_B'=J_{\alpha B}z'\\
    \quad\Leftrightarrow\quad
    &\{(z_B,\alpha^{-1}(z-z_B)),(z_B',\alpha^{-1}(z'-z_B'))\}\text{ is }\mathcal{C}_{\beta}\text{-interpolable}\\
    &\{(z_B,\alpha^{-1}(z-z_B)),(z_B',\alpha^{-1}(z'-z_B'))\}\text{ is }\mathcal{L}_{L}\text{-interpolable}\\
    \quad\Leftrightarrow\quad&
    \langle z-z'-z_B+z_B',z_B-z_B'\rangle \ge (\beta/\alpha)\|z-z'-z_B+z_B'\|^2\\
    &\alpha^2 L^2\|z_B-z_B'\|^2\ge \|z-z'-z_B+z_B'\|^2.
\end{align*}
For operator $C$, we have\begin{align*}
    &\exists C\in\mathcal{C}_{\beta_C}\text{ such that }
    z_C=\alpha Cz_B,\,
    z_C'=\alpha Cz_B'\\
    &\quad\Leftrightarrow\quad
    \{(z_B,\alpha^{-1}z_C),(z_B',\alpha^{-1}z_C')\}\text{ is }\mathcal{C}_{\beta_C}\text{-interpolable}\\
    &\quad\Leftrightarrow\quad
    \langle z_B-z_B',z_C-z_C'\rangle \ge (\beta_C/\alpha)\|z_C-z_C'\|^2.
\end{align*}
Now we can drop the explicit dependence on the operators $A$, $B$, and $C$ and reformulate \eqref{eq:ospep_main3} into
\[
\begin{tabular}{ll}
 $\displaystyle \maximize_{}$ & $\displaystyle\|z-\theta(z_B-z_A)-z'+\theta(z_B'-z_A')\|^2$\\
 subject to & 
    $\langle z_A-z_A',2z_B-z- z_C-(2z_B'-z'- z'_C)\rangle
    \ge 
    (1+\alpha {\mu})\|z_A-z_A'\|^2$\\
    &$\langle z-z'-z_B+z_B',z_B-z_B'\rangle \ge (\beta/\alpha)\|z-z-z_B+z_B'\|^2$\\
    &$\alpha^2 L^2\|z_B-z_B'\|^2\ge \|z-z'-z_B+z_B'\|^2$\\
    &$\langle z_B-z_B',z_C-z_C'\rangle \ge (\beta_C/\alpha)\|z_C-z_C'\|^2$\\
 &$\|z-z'\|^2= 1$,
\end{tabular}
\]
where $z,z',z_A,z_A',z_B,z_B',z_C,z_C'\in \hilbert$ are the optimization variables.
Since the variables only appear as differences between the primed and non-primed variables, we can perform a change of variables $z-z'\mapsto z$, $z_A-z_A'\mapsto z_A$,  $z_B-z_B'\mapsto z_B$ and  $z_C-z_C'\mapsto z_C$ to get
\begin{align}
\begin{tabular}{ll}
 $\displaystyle \maximize_{}$ & $\displaystyle\|z-\theta(z_B-z_A)\|^2$\\
 subject to & 
    $\langle z_A,2z_B-z- z_C\rangle
    \ge 
    (1+\alpha {\mu})\|z_A\|^2$\\
    &$\langle z-z_B,z_B\rangle \ge (\beta/\alpha)\|z-z_B\|^2$\\
    &$\alpha^2 L^2\|z_B\|^2\ge \|z-z_B\|^2$\\
    &$\langle z_B,z_C\rangle \ge (\beta_C/\alpha)\|z_C\|^2$\\
 &$\|z\|^2= 1$,
\end{tabular}
\label{eq:ospep_main5}
\end{align}
where $z,z_A,z_B,z_C\in \hilbert$ are the optimization variables.

% \begin{rem}
% Using this line of reasoning, it is possible to show that the tight contraction factor and the tight \emph{quasi-contraction factor} are the same for DYS.
% We say $\rho<1$ is a quasi-contraction factor if
% \[
% \|Tx-x^\star\|\le \rho\|x-x^\star\|
% \]
% for any $x\in \hilbert$ and fixed point $x^\star\in\hilbert$ of $T$.
% Ditto for FBS and DRS.
% Quasi-contraction factors directly describe how fast we converge to the solution.
% In our setup, the notion of quasi-contraction factors and contraction factors coincide.
% \end{rem}

\subsubsection{Grammian representation}
The optimization problem \eqref{eq:ospep_main5} and all other operator classes in Section~\ref{s:operator_interp} are specified through inner products and squared norms.
This structure allows us to rewrite the problem with a Grammian representation:
\begin{equation}
G=
\begin{pmatrix}
\|z\|^2&\langle z,z_A\rangle &\langle z,z_B\rangle&\langle z,z_C\rangle\\
\langle z,z_A\rangle &\|z_A\|^2&\langle z_A,z_B\rangle &\langle z_A,z_C\rangle\\
\langle z,z_B\rangle&\langle z_A,z_B\rangle &\|z_B\|^2&\langle z_B,z_C\rangle\\
\langle z,z_C\rangle&\langle z_A,z_C\rangle&\langle z_B,z_C\rangle & \|z_C\|^2
\end{pmatrix}.
\label{eq:grammian}
\end{equation}

\begin{lemma}
\label{lem:cholesky}
If $\dim\hilbert\ge 4$, then 
\[
G\in \mathbb{S}_+^{4}
\quad\Leftrightarrow\quad
\exists z,z_A,z_B,z_C\in \hilbert
\text{ such that }G= \text{expression of \eqref{eq:grammian}}.
% \begin{pmatrix}
% \|z\|^2&\langle z,z_A\rangle &\langle z,z_B\rangle&\langle z,z_C\rangle\\
% \langle z,z_A\rangle &\|z_A\|^2&\langle z_A,z_B\rangle &\langle z_A,z_C\rangle\\
% \langle z,z_B\rangle&\langle z_A,z_B\rangle &\|z_B\|^2&\langle z_B,z_C\rangle\\
% \langle z,z_C\rangle&\langle z_A,z_C\rangle&\langle z_B,z_C\rangle & \|z_C\|^2
% \end{pmatrix}.
\]
\end{lemma}
\begin{proof}
$(\Leftarrow)$ For any $z,z_A,z_B,z_C\in \hilbert$, $G$ is positive semidefinite since
\[
x^TGx=
\left\|x_1z+x_2z_A+x_3z_B+x_4z_C\right\|^2\ge 0 
\]
for any $x=(x_1,x_2,x_3,x_4)\in \reals^4$.

$(\Rightarrow)$
Let $LL^T=G$ be a Cholesky factorization of $G$. Write
\[
L=\begin{bmatrix}
\tilde{z}^T\\\tilde{z}_A^T\\\tilde{z}_B^T\\\tilde{z}_C^T
\end{bmatrix}
\]
where $\tilde{z},\tilde{z}_A,\tilde{z}_B,\tilde{z}_C\in \reals^4$.
We can find orthonormal vectors $e_1,e_2,e_3,e_4\in  \hilbert$ since $\dim\hilbert\ge 4$.
Define
\[
z=\tilde{z}_1e_1+\tilde{z}_2e_2+\tilde{z}_3e_3+\tilde{z}_4e_4,
\qquad
z_A=(\tilde{z}_A)_1e_1+(\tilde{z}_A)_2e_2+(\tilde{z}_A)_3e_3+(\tilde{z}_A)_4e_4.
\]
Define $z_B,z_C\in \hilbert$ similarly.
Then $G$ is as given by \eqref{eq:grammian} with the constructed $z,z_A,z_B,z_C\in \hilbert$.
\end{proof}

% Since $\reals^4\cong \mathrm{span}\{e_1,e_2,e_3,e_4\}\subset \hilbert$, 
% we can identify
% $\tilde{z},\tilde{z}_A,\tilde{z}_B,\tilde{z}_C\in \reals^4$
% with $z,z_A,z_B,z_C\in \hilbert$.

Write\begin{align*}
M_I&=
\begin{pmatrix}
1&0&0&0\\
0&0&0&0\\
0&0&0&0\\
0&0&0&0
\end{pmatrix},
&
M_O&=
\begin{pmatrix}
1&\theta&-\theta&0\\
\theta&\theta^2&-\theta^2&0\\
-\theta&-\theta^2&\theta^2&0\\
0&0&0&0
\end{pmatrix},\\
M_\mu^A&=\begin{pmatrix}
0&-1/2&0&0\\
-1/2&-1-\alpha \mu&1&-1/2\\
0&1&0&0\\
0&-1/2&0&0
\end{pmatrix},
&
M_\beta^C&=\begin{pmatrix}
0&0&0&0\\
0&0&0&0\\
0&0&0&1/2\\
0&0&1/2&-\beta_C/\alpha\\
\end{pmatrix},\\
M_\beta^B&=
\begin{pmatrix}
-\beta/\alpha &0&1/2+\beta/\alpha&0\\
0&0&0&0\\
1/2+\beta/\alpha&0&-1-\beta/\alpha&0\\
0&0&0&0
\end{pmatrix},
&
M_L^B&=
\begin{pmatrix}
-1&0&1&0\\
0&0&0&0\\
1&0&-1+\alpha^2L^2&0\\
0&0&0&0
\end{pmatrix}.
\end{align*}
When $\dim\hilbert \ge 4$, we can use Lemma~\ref{lem:cholesky}
to reformulate \eqref{eq:ospep_main5} into the equivalent SDP
\begin{align}
\begin{tabular}{ll}
 $\displaystyle \maximize_{}$ & $ \tr(M_{O}G)$\\
 subject to 
    &$ \tr(M_\mu^A G)\ge 0$\\
    &$ \tr(M_\beta^B G)\ge 0$\\
    &$ \tr(M_L^B G)\ge 0$\\
    &$ \tr(M_\beta^C G)\ge 0$\\
    &$ \tr(M_{I}G)=1$\\
    &$G\succeq 0$
\end{tabular}
\label{eq:ospep_main6}
\end{align}
where $G\in \mathbb{S}_+^4$ is the optimization variable.
Since \eqref{eq:ospep_main6} is a finite-dimensional convex SDP, we can solve it efficiently with standard solvers.

%$G\in \mathbb{S}_+^{4}$

%A feasible point of \eqref{eq:ospep_main5} corresponds to a feasible point of \eqref{eq:ospep_main6} using $G$ defined as in \eqref{eq:grammian}.

These equivalent reformulations prove Theorem~\ref{thm:exactness} for this special case.
The general case follows from analogous steps, and we show the fully general SDP in the supplementary materials, in Section~\ref{appendix:full_primal_ospep}.

\begin{theorem}
\label{thm:exactness}
The OSPEP \eqref{eq:ospep_main} and the SDP of Section~\ref{appendix:full_primal_ospep} are equivalent
if $\dim\hilbert\ge 4$ and $\Q_1=\mathcal{M}_{\mu_A}\cap \mathcal{C}_{\beta_A}\cap \mathcal{L}_{L_A}$,
$\Q_2=\mathcal{M}_{\mu_B}\cap \mathcal{C}_{\beta_B}\cap \mathcal{L}_{L_B}$,
and 
$\Q_3=\mathcal{M}_{\mu_C}\cap \mathcal{C}_{\beta_C}\cap \mathcal{L}_{L_C}$.
%are equal and a solution from one problem can be transformed into another.
\end{theorem}

To clarify, Theorem~\ref{thm:exactness} states that the optimal values of the two problems are equal and that a solution from one problem can be transformed into a solution of another.
Given an optimal $G^\star$ of the SDP, we can take its Cholesky factorization as in Lemma~\ref{lem:cholesky} to get $z,z_A,z_B,z_C\in \hilbert$ and obtain evaluations of the worst-case operators
\begin{gather*}
A(z_A)\ni\alpha^{-1}(2z_B-z- z_C-z_A),\quad A(0)\ni 0,\quad\text{ where }A\in \Q_1\\
B(z_B)\ni\alpha^{-1}(z-z_B),\quad B(0)\ni 0,\quad\text{ where }B\in \Q_2\\
C(z_B)\ni \alpha^{-1}z_C,\quad C(0)\ni 0,\quad\text{ where }C\in \Q_3.
\end{gather*}

\subsection{Dual OSPEP}
The SDP \eqref{eq:ospep_main6} has a dual:
\begin{equation}
\begin{tabular}{ll}
 $\displaystyle \minimize{}$ & $ \rho^2$\\
 subject to &$\lambda_\mu^A,\lambda_\beta^B,\lambda_L^B,\lambda_\beta^C\ge 0$\\
 &$S(\rho^2,\lambda_\mu^A,\lambda_\beta^B,\lambda_L^B,\lambda_\beta^C,\theta,\alpha)\succeq 0$
\end{tabular}
\label{eq:dual-ospep}
\end{equation}
where $\rho^2,\lambda_\mu^A,\lambda_\beta^B,\lambda_L^B,\lambda_\beta^C\in\reals$ are the optimization variables and
\begin{align}
 &S(\rho^2,\lambda_\mu^A,\lambda_\beta^B,\lambda_L^B,\lambda_\beta^C,\theta,\alpha)
 =-M_O-\lambda_\mu^A M_\mu^A-\lambda_\beta^B M_\beta^B-\lambda_L^B M_L^B
-\lambda_\beta^{C}M_\beta^{C}
+\rho^2M_I
 \label{eq:Smat-S3}
\\
&=
\begin{pmatrix}
\rho^2+\frac{\lambda_\beta^B \beta}{\alpha}+\lambda_L^B-1
&\frac{\lambda_\mu^A}{2}-\theta
&-\lambda_\beta^B(\frac{1}{2}+\frac{\beta}{\alpha})-\lambda_L^B+\theta&0\\
\frac{\lambda_\mu^A}{2}-\theta&
\lambda_\mu^A(1+\alpha \mu)-\theta^2&-\lambda_\mu^A+\theta^2&\frac{\lambda_\mu^A}{2}\\
-\lambda_\beta^B(\frac{1}{2}+\frac{\beta}{\alpha})-\lambda_L^B+\theta&
-\lambda_\mu^A+\theta^2
&
\lambda_\beta^B(\frac{\beta}{\alpha}-1)+\lambda_L^B(1-\alpha^2L^2)-\theta^2&-\frac{\lambda_\beta^C}{2}
\\
0&\frac{\lambda_\mu^A}{2}&-\frac{\lambda_\beta^C}{2}&\frac{\lambda_\beta^C\beta_C}{\alpha}
\end{pmatrix}.\nonumber
\end{align}
We call \eqref{eq:dual-ospep} the \emph{dual OSPEP}. 
In contrast, we call the OSPEP \eqref{eq:ospep_main}, and equivalently \eqref{eq:ospep_main6}, the \emph{primal OSPEP}.
Again, this special case illustrates the overall approach.
We show the fully general dual OSPEP in the supplementary materials, in Section~\ref{appendix:full_dual_ospep}.

%Weak duality between the primal and dual OSPEPs is immediate.
To ensure strong duality between the primal and dual OSPEPs, we enforce Slater's constraint qualification with the following notion of degeneracy.
We say the intersections 
$\mathcal{C}_{\beta}\cap \mathcal{L}_{L}$,
$\mathcal{M}_{\mu}\cap \mathcal{C}_{\beta}$,
$\mathcal{M}_{\mu}\cap \mathcal{L}_{L}$, and
$\mathcal{M}_{\mu}\cap \mathcal{C}_{\beta}\cap \mathcal{L}_{L}$
are respectively \emph{degenerate} if
$\mathcal{C}_{\beta+\varepsilon}\cap \mathcal{L}_{L-\varepsilon}=\emptyset$,
$\mathcal{M}_{\mu+\varepsilon}\cap \mathcal{C}_{\beta+\varepsilon}=\emptyset$,
$\mathcal{M}_{\mu+\varepsilon}\cap \mathcal{L}_{L-\varepsilon}=\emptyset$, and 
$\mathcal{M}_{\mu+\varepsilon}\cap \mathcal{C}_{\beta+\varepsilon}\cap \mathcal{L}_{L-\varepsilon}=\emptyset$
for all $\varepsilon>0$.
For example, $\mathcal{M}_{3}\cap \mathcal{L}_{3}=\{3I\}$ is a degenerate intersection.
%We say an intersection is \emph{non-degenerate} if it is not degenerate.

\begin{theorem}
\label{thm:duality}
Weak duality holds between the primal and dual OSPEPs of Sections~\ref{appendix:full_primal_ospep} and \ref{appendix:full_dual_ospep}.
Furthermore, strong duality holds if each class $\Q_1$, $\Q_2$, and $\Q_3$ is a non-degenerate intersection of classes of Table~\ref{tab:classes_summary}. 
\end{theorem}
\begin{proof}
Weak duality follows from the fact that the SDP of Section~\ref{appendix:full_dual_ospep} is the Lagrange dual of the SDP of Section~\ref{appendix:full_primal_ospep}.
To establish strong duality, we show that the non-degeneracy assumption leads to Slater's constraint qualification \cite{rockafellar1974} for the primal OSPEP.

Since the intersections are non-degenerate, there is a small $\varepsilon>0$ and $A$, $B$, and $C$ such that
\begin{align*}
A&\in \mathcal{M}_{\mu_A+\varepsilon}\cap \mathcal{C}_{\beta_A+\varepsilon}\cap \mathcal{L}_{L_A-\varepsilon}\\
B&\in \mathcal{M}_{\mu_B+\varepsilon}\cap \mathcal{C}_{\beta_B+\varepsilon}\cap \mathcal{L}_{L_B-\varepsilon}\\
C&\in \mathcal{M}_{\mu_C+\varepsilon}\cap \mathcal{C}_{\beta_C+\varepsilon}\cap \mathcal{L}_{L_C-\varepsilon}.
\end{align*}
With any inputs $z,z'\in \hilbert$ such that $z\ne z'$, we can follow the arguments of Section~\ref{ss:conv-ospep-transformation}
and construct a $G$ matrix as defined in \eqref{eq:grammian}.
This $G$ satisfies
\[
\tr(M_\mu^AG)>0,\,\dots,\,\tr(M^C_LG)>0,\,\tr(M_IG)=1,\,G\succeq 0.
\]
%\PG{Comment:Notation already changed here.}\ER{What do you mean by this comment?}
Define $G_\delta=(1-\delta)G+\delta I$.
There exists a small $\delta>0$ such that
\[
\tr(M_\mu^AG_\delta)>0,\,\dots,\,\tr(M^C_LG_\delta)>0,\,\tr(M_IG_\delta)=1,\,G_\delta\succ 0.
\]
Note that the equality constraint $\tr(M_IG_\delta)=1$ holds since $\tr(M_I)=1$.
Since $G_\delta$ is a strictly feasible point, Slater's condition gives us strong duality.
\end{proof}

% \begin{rem}
More generally, the strong duality argument of Theorem~\ref{thm:duality} applies if each $\Q_1$, $\Q_2$, and $\Q_3$ is a single operator class of Table~\ref{tab:classes_summary} or is a non-degenerate intersection of those classes.
% \end{rem}

\subsection{Primal and dual interpretations and computer-assisted proofs}
A feasible point of the primal OSPEP provides a lower bound on any contraction factor as it corresponds to operator instances that exhibit a contraction corresponding to the objective value.
An optimal point of the primal OSPEP corresponds to the worst-case operators.
A feasible point of the dual OSPEP provides an upper bound as it corresponds to a proof of a contraction factor.
A convergence proof in optimization is a nonnegative combination of known valid inequalities.
The nonnegative variables of the dual OSPEP correspond to weights of such a nonnegative combination, and the objective value is the contraction factor the nonnegative combination of inequalities (i.e., the proof) proves.

We can use the OSPEP methodology as a tool for computer-assisted proofs.
Given the operator classes, we can choose specific numerical values for the parameters, such as the strong convexity and cocoercivity parameters, and numerically solve the SDP.
% If we could solve the SDP to infinite precision, the solution would be a rigorous proof of the tight contraction factor for the specific parameter values.
We do this for many parameter values, observe the pattern of primal and dual solutions, and guess the analytical, parameterized solution to the SDPs.
To put it differently, the SDP solver provides a valid and optimal proof for a given choice of parameters, and we use this to infer 

\subsection{Further remarks}
With analogous steps, the OSPEPs for FBS and DRS can be written as smaller $3\times 3$ SDPs.
Using the smaller SDP is preferred, as formulating these cases into larger $4\times 4$ SDPs, as a special case of the $4\times 4$ SDP for DYS, can lead to numerical difficulties.
% With analogous steps, it is possible to consider the operator class $\partial \mathcal{F}_{\mu,L}$.
% a strong duality result analogous to Theorem~\ref{thm:duality}.
% In theory, one could consider the OSPEP FBS and DRS are special cases of 
% the $4\times 4$ SDP will run into numerial problems can 

The tightness of the OSPEP methodology relies on the two-point interpolation results of Section~\ref{s:operator_interp}, which we can use because the operators $A$, $B$, and $C$ are evaluated \emph{once} per iteration.
(To analyze the contraction factor, we consider a single evaluation of the operator at two distinct points, which leads to two evaluations of each operator.)
For splitting methods without this property, methods that access one of the operators twice or more per iteration, the OSPEP loses the tightness guarantee.
Such methods include the extragradient method \cite{korpelevich1976}, FBF \cite{tseng2000}, PDFP \cite{Chen2016}, Extragradient-Based Alternating Direction Method for Convex Minimization \cite{Lin2017}, FBHF \cite{davis2018}, FRB \cite{malitsky2018}, Golden ratio algorithm \cite{malitsky2018_golden}, Shadow-Douglas-Rachford \cite{csetnek2019}, and BFRB/BRFB \cite{rieger2020}.
Nevertheless, the OSPEP is applicable for analyzing these types of methods and, in particular, can be used to find the convergence proofs presented in these references.

\section{Tight analytic contraction factors for DRS} 
\label{sec:tight_DRS}
In this section, we present tight analytic contraction factors for DRS under two sets of assumptions considered in \cite{giselsson2017tight,moursi2018douglas}.
The primary purpose of this section is to demonstrate the strength of the OSPEP methodology through proving results that are likely too complicated for a human to find bare-handed.
% The purpose of this section is twofold: firstly, to present the contraction factors, which we consider to be interesting and useful; and secondly,  
The proofs are computer-assisted in that their discoveries were assisted by a computer, but their verifications do not require a computer.
% Nevertheless, we provide code that performs symbolic manipulations to help readers verify the algebra.

% Again, we say $\rho\in(0,1)$ is a contraction factor of $T$ if
% \[ \norm{Tx - Ty}\leq \rho \norm{x-y} \] 
% for all $x,y\in\hilbert$. In Section 3, we formally defined tightness. Loosely speaking, a contraction factor is tight if it cannot be improved without additional assumptions. 

The results below are presented for $\alpha=1$.
 The general rate for $\alpha>0$ follows from the scaling $\mu\mapsto \alpha \mu$, $\beta\mapsto \beta/\alpha$, and $L\mapsto \alpha L$.
The proofs are presented in the supplementary materials, in Section~\ref{s:main-proof}.

\begin{theorem}\label{thm:DRS_coco_strmonotone}
Let $A\in\M_\mu$ and $B\in\C_\beta$ with $\mu,\beta>0$, and assume $\dim\hilbert \geq 3$.
The \textbf{tight} contraction factor of the DRS operator $I-\theta J_B+\theta J_A(2J_B-I)$ for $\theta\in(0,2)$ is
\begin{equation*}
	\rho=\left\{\begin{array}{ll}
	\lvert1-\theta \tfrac{\beta}{\beta+1}\rvert	&  \text{ if } \mu\beta-\mu+\beta<0 \text{ and } \theta \leq 2 \tfrac{(\beta+1) (\mu-\beta-\mu\beta)}{\mu+\mu\beta-\beta-\beta^2-2\mu\beta^2},\\
	\lvert 1-\theta\tfrac{1+\mu\beta}{(\mu+1)(\beta+1)}\rvert	&  \text{ if } \mu\beta-\mu-\beta>0 \text{ and } \theta \leq 2  \tfrac{\mu^2+\beta^2+\mu\beta+\mu+\beta-\mu^2\beta^2}{\mu^2+\beta^2+\mu^2\beta+\mu\beta^2+\mu +\beta-2\mu^2\beta^2},\\
	\lvert 1-\theta \rvert	&  \text{ if } \theta \geq 2 \tfrac{\mu\beta+\mu+\beta}{2\mu\beta+\mu+\beta},\\
	\lvert 1-\theta\tfrac{\mu}{\mu +1} \rvert 	& \text{ if } \mu\beta+\mu-\beta<0 \text{ and }\theta \leq 2\tfrac{(\mu+1) (\beta-\mu-\mu\beta)}{\beta+\mu\beta-\mu-\mu^2-2\mu^2\beta},\\
	\rho_5	& \text{ otherwise,}
	\end{array}\right.
\end{equation*}
with
\[\rho_5=\tfrac{\sqrt{2-\theta}}{2} \sqrt{\tfrac{((2-\theta)\mu (\beta+1)+{\theta}  \beta(1-\mu))\, ((2-\theta)\beta ({\mu}+1)+{\theta}\mu  (1-\beta))}{\mu\beta(2\mu \beta (1-\theta )  + (2-\theta ) (\mu+\beta +1))}}.\]
\end{theorem}
(In the first, second, and fourth cases,
the former parts of the conditions ensure that there is no division by $0$ in the latter parts.
We show this in Section~\ref{ss:theorem3_verification} case (a) part (ii), case (b) part (ii), and case (d) part (ii).)

\begin{corollary}\label{cor:DRS_coco_strmonotone_no_relax}
Let $A\in\M_\mu$ and $B\in\C_\beta$ with $\mu,\beta>0$, and assume $\dim\hilbert \geq 3$. The \textbf{tight} contraction factor of the DRS operator $I- J_B+ J_A(2J_B-I)$ is
\begin{equation*}
	\rho=\left\{\begin{array}{ll}
	\lvert1- \tfrac{\beta}{\beta+1}\rvert	&  \text{ if } \beta^2+\mu\beta+\beta-\mu \leq 0,\\
	\lvert 1-\tfrac{1+\mu\beta}{(\mu+1)(\beta+1)}\rvert	&  \text{ if } \mu\beta-\mu-\beta\geq 1,\\
	\lvert 1-\tfrac{\mu }{\mu +1} \rvert 	& \text{ if } \mu^2+\mu\beta+\mu-\beta\leq 0,\\
	\tfrac{1}{2} \tfrac{\beta +\mu}{\sqrt{\beta  \mu  (\beta +\mu +1)}}	& \text{ otherwise.}
	\end{array}\right.
	\end{equation*}
\end{corollary}
\begin{proof}
Plug $\theta = 1$ into Theorem~\ref{thm:DRS_coco_strmonotone} and simplify. We omit the details.
\end{proof}

\begin{theorem}\label{thm:DRS_Lipschitz_strmonotone}
Let $A\in\M_\mu$ and $B\in\M\cap\Li_L$ with $\mu,L>0$, and assume $\dim\hilbert\geq 3$. The \textbf{tight} contraction factor of the DRS operator $I-\theta J_B+\theta J_A(2J_B-I)$ for $\theta\in(0,2)$ is
\begin{equation*}
	\rho=\left\{\begin{array}{ll}
	\tfrac{\theta +\sqrt{\tfrac{(2 (\theta -1) \mu +\theta -2)^2+L^2 (\theta -2 (\mu +1))^2}{L^2+1}}}{2 (\mu +1)}	&  \text{ if } (a),\\
	\lvert 1-\theta\tfrac{L+\mu}{(\mu+1)(L+1)} \rvert	&  \text{ if } (b),\\
	\sqrt{\tfrac{(2-\theta)}{4 \mu  \left(L^2+1\right)}\, \tfrac{
	%\left(\theta  \left(1-2 \mu +L^2\right)-2 \mu  \left(L^2-1\right)\right)
	\left(
	\theta(L^2+1)-2\mu(\theta+L^2-1)
	\right)
	\left(\theta  \left(1+2 \mu +L^2\right)-2 (\mu +1) \left(L^2+1\right)\right)}{2 \mu  \left(\theta +L^2-1\right)-(2-\theta ) \left(1-L^2\right)}}	& \text{ otherwise,}
	\end{array}\right.
	\end{equation*}
	with
	\begin{itemize}
	\item[(a)] $\mu\tfrac{-\left(2 (\theta -1) \mu +\theta-2\right)+L^2\left(\theta -2(1+ \mu)\right)}{\sqrt{ (2 (\theta -1) \mu +\theta -2)^2+L^2 (\theta -2 (\mu +1))^2}}\leq \sqrt{L^2+1}$,
	\item[(b)] $L<1$, $\mu >\tfrac{L^2+1}{(L-1)^2}$, and $\theta \leq \tfrac{2 (\mu +1) (L+1)(\mu +\mu  L^2-L^2-2 \mu  L-1)}{2 \mu ^2-\mu +\mu  L^3-L^3-3 \mu  L^2-L^2-2 \mu ^2 L-\mu  L-L-1}$.
\end{itemize}
\end{theorem}
(In case (b), the former part of the condition ensures that there is no division by $0$ in the latter part.
We show this in Section~\ref{ss:lip-mon-upper} case (b) part (ii).)

\begin{corollary}\label{cor:DRS_Lipschitz_strmonotone_no_relax}
Let $A\in\M_\mu$ and $B\in\M\cap\Li_L$ with $\mu,L>0$, and assume $\dim\hilbert\geq 3$. The \textbf{tight} contraction factor of the DRS operator $I- J_B+ J_A(2J_B-I)$ is
\begin{equation*}
	\rho=\left\{\begin{array}{ll}
	\tfrac{1+\sqrt{\tfrac{(1-2 (\mu +1))^2 L^2+1}{L^2+1}}}{2(1+\mu)}	&  \text{ if } (\mu -1)(2 \mu +1)^2 L^2\geq 2 \mu ^2-2 \sqrt{2} \sqrt{\mu +1} \mu +\mu +1 \text{ or } \mu\leq 1,\\
	\tfrac{1+\mu  L}{(1+\mu) (1+L)}	&  \text{ if } L\leq \tfrac{2 \mu ^2 (L-1) L^2+\mu  (1-2 L)-1}{(\mu +1) \left(L^2+L+1\right)} \text{ and } L<1,\\
	\sqrt{\tfrac{(2\mu L^2+L^2+1)(2\mu L^2-L^2-1)}{4 \mu  \left(L^2+1\right) \left(2 \mu L^2+L^2 -1\right)}}	& \text{ otherwise.}
	\end{array}\right.
\end{equation*}
\end{corollary}
\begin{proof}
Plug $\theta = 1$ into Theorem~\ref{thm:DRS_Lipschitz_strmonotone} and simplify. We omit the details.
\end{proof}

% Here, we outline the proofs and describe the process through which we discovered them.

% Those contraction factors were found thanks to an intensive use of computer algebra software (CAS).
% Essentially, one can either find them by solving the primal problem, or by solving the dual one.

\subsection{Proof outline}
The discovery of these proofs relied heavily on a computer algebra system (CAS), Mathematica.
% Using the CAS, we symbolically solved the primal and dual problems.
When symbolically solving the primal problem, we conjectured that the worst-case operators would exist in $\reals^2$.
This is equivalent to conjecturing that the solution $G^\star\in\reals^{3\times 3}$ has rank $2$ or less, which is reasonable due to complementary slackness.
We then formulated the problem of finding this $2$-dimensional worst-case as a non-convex quadratic program, rather than an SDP, formulated the KKT system, and solved the stationary points using the CAS.
When symbolically solving the dual problem, we conjectured that the optimal solution would correspond to $S^\star\in \reals^{3\times 3}$ with rank $1$ or $2$, which is reasonable due to complementary slackness.
We then chose $\rho^2$ and the other dual variables so that $S^\star$ would have rank $1$ or $2$.
Finally, we minimized the contraction factor $\rho^2$ under those rank conditions to obtain the optimum.
These two approaches gave us analytic expressions for optimal primal and dual SDP solutions.
To verify the solutions, we formulated them into primal and dual feasible points and verified that their optimal values are equal for all parameter choices.

The written proof of Theorems~\ref{thm:DRS_coco_strmonotone} and \ref{thm:DRS_Lipschitz_strmonotone}, are deferred to supplementary materials, to Sections~\ref{s:main-proof} and \ref{ss:verification}.
The point we wish to make in this section is that the OSPEP is a powerful tool that enables us to prove incredibly complex results. The length and complexity of the proofs demonstrate this point.

The proofs provided on paper are complete and rigorous.
However, we help readers verify the calculations of Sections~\ref{s:main-proof} and \ref{ss:verification} with code that performs symbolic manipulations.
If a reader is willing to trust the CAS's symbolic manipulations, the proofs are not difficult to follow.
We also verified the results through the following alternative approach:
we finely discretized the parameter space and verified that the upper and lower bounds of Section~\ref{s:main-proof} are valid and that they match up to machine precision.
The link to the code is provided in the conclusion.
%We list the provided code in the conclusion.

\subsection{Further remarks}

% \begin{rem}
The third contraction factor of Theorem~\ref{thm:DRS_coco_strmonotone}, the factor $|1-\theta|$, matches the contraction factor of Theorem 5.6 of \cite{giselsson2017tight}. 
The contraction factor for the other 4 cases do not match.
This implies, Theorem 5.6 of \cite{giselsson2017tight} is tight when $\theta \geq 2 \tfrac{\mu\beta+\mu+\beta}{2\mu\beta+\mu+\beta}$ but not in the other cases.
%(and matches the tight rate of Theorem~\ref{thm:DRS_coco_strmonotone} in this case)
% \end{rem}

% \begin{rem}
The first contraction factor of Corollary~\ref{cor:DRS_Lipschitz_strmonotone_no_relax}, but not the second and third, matches the contraction factor of Theorem 5.2 of \cite{moursi2018douglas}
which instead assumes $B$ is a skew symmetric $L$-Lipschitz linear operator, a stronger assumption than $B\in \mathcal{M}\cap \mathcal{L}$.
% \end{rem}

% \begin{rem}
% \label{rem:symmetry}
One can show that the contraction factors of Theorems~\ref{thm:DRS_coco_strmonotone} and \ref{thm:DRS_Lipschitz_strmonotone} are symmetric in the assumptions.
Specifically, if we swap the assumptions and instead assume [$B\in\M_\mu$ and $A\in\C_\beta$] and [$B\in\M_\mu$ and $A\in\M\cap\Li_L$], the contraction factors of Theorems~\ref{thm:DRS_coco_strmonotone} and \ref{thm:DRS_Lipschitz_strmonotone} remain valid and tight.
The proof follows from using the ``scaled relative graph'' developed in the concurrent work by Ryu, Hannah, and Yin \cite[Theorem~7]{ryu2019scaled}.
% \end{rem}

% \begin{rem}
The optimal $\alpha$ and $\theta$ minimizing the contraction factor of Theorems~\ref{thm:DRS_coco_strmonotone} and \ref{thm:DRS_Lipschitz_strmonotone} can be computed with the algorithm presented in  Section~\ref{sec:parameter_selection}.
However, their analytical expressions seem to be quite complicated.

% \begin{rem}
If we furthermore assume $A$ and $B$ are subdifferential operators of closed convex proper functions, the contraction factors of Theorems~\ref{thm:DRS_coco_strmonotone} and \ref{thm:DRS_Lipschitz_strmonotone} remain valid but our proof no longer guarantees tightness; with the additional assumptions, it may be possible to obtain a smaller contraction factor.
Such setups can be analyzed with the machinery and interpolation results of \cite{taylor2017smooth}.
By numerically solving the SDP with the added subdifferential operator assumption, we find that Theorem~\ref{thm:DRS_coco_strmonotone} remains tight.
For subdifferential operators of convex functions, Lipschitz continuity implies cocoercivity by the Baillon--Haddad theorem, so there is no reason to consider Theorem~\ref{thm:DRS_Lipschitz_strmonotone}.
Indeed, numerical solutions of the SDP indicate Theorem~\ref{thm:DRS_Lipschitz_strmonotone} is not tight in this setup.

\begin{table}[!h]
	\begin{center}
		{\renewcommand{\arraystretch}{1.3}
			\begin{tabular}{@{}ccccc@{}}
				\specialrule{2pt}{1pt}{1pt}
				\bfseries \bfseries Properties for $A$ &  \bfseries Properties for $B$ &  \bfseries Reference &  \bfseries Tight  \\
				\hline
				  $\partial f$, $f$: str. cvx \& smooth             &  $\partial g$   & \cite{giselsson2014diagonal,giselsson2017linear} &    Y     \\  
				  $\partial f$, $f$: str. cvx        &  $\partial g$, $g$: smooth   & \cite{giselsson2017tight}      &   N                     \\  
				  str. mono. \& cocoercive                           &  -                      &    \cite{giselsson2017tight}    &   Y                      \\  
				  str. mono. \& Lipschitz                             &  -                                       &    \cite{giselsson2017tight}        &   Y \\  
				  str. mono.                        &  cocoercive                    &    \cite{giselsson2017tight}       &   N                   \\  
				  str. mono.                        &  Lipschitz                     &  \cite{moursi2018douglas}      &   N                       \\ 
				\specialrule{2pt}{1pt}{1pt}
			\end{tabular}
		}
\caption{Prior results on contraction factors of Douglas--Rachford splitting.}
	\label{tab:DRS_settings_summary}
\end{center}
\end{table}

Table~\ref{tab:DRS_settings_summary} lists other commonly considered assumptions providing linear convergence of DRS and the corresponding prior work analyzing them.
The results of Theorems~\ref{thm:DRS_coco_strmonotone} and \ref{thm:DRS_Lipschitz_strmonotone} provide the tight contraction factors for the three cases for which there had not been tight results.

%%%%%%%%%%%%%%%%%%%%%%%%%%%%%%%%%%%%%%%%%%%%%%%%%%%%%%%%%%%%%%%%%%%%%
%%%%%%%%%%%%%%%%%%%%%%%%%%%%%%%%%%%%%%%%%%%%%%%%%%%%%%%%%%%%%%%%%%%%%
%                                                                   %
%                                                                   %
%                       PARAMETER SELECTION                         %
%                                                                   %
%                                                                   %
%%%%%%%%%%%%%%%%%%%%%%%%%%%%%%%%%%%%%%%%%%%%%%%%%%%%%%%%%%%%%%%%%%%%%
%%%%%%%%%%%%%%%%%%%%%%%%%%%%%%%%%%%%%%%%%%%%%%%%%%%%%%%%%%%%%%%%%%%%%
\section{Automatic optimal parameter selection}
\label{sec:parameter_selection}
When using FBS, DRS, or DYS, how should one choose the parameters $\alpha > 0$ and $\theta \in (0,2)$?
One option is to find a contraction factor
%(by finding one in a publication or by directly proving it)
and choose the $\alpha$ and $\theta$ that minimizes it.
However, this may be suboptimal if the contraction factor is not tight or if no known contraction factors fully utilize a given set of assumptions

In this section, we use the OSPEP to automatically select the optimal algorithm parameters for FBS, DRS, and DYS.
Write
\begin{align*}
\rho^2_{\star}(\alpha,\theta)=\left(
\begin{tabular}{ll}
$\displaystyle \maximize$ & $\displaystyle\frac{\|T(z;A,B,C,\alpha,\theta)- T(z';A,B,C,\alpha,\theta)\|^2}{\|z-z'\|^2}$\\
 subject to & $A\in\Q_1$, $B\in\Q_2$, $C\in\Q_3$\\
& $z,z'\in\hilbert$, $z\neq z'$
\end{tabular}\right)
\end{align*}
where $z,$ $z'$, $A$, $B$, and $C$ are the optimization variables.
This is the tight contraction factor of \eqref{eq:ospep_main}, and we make explicit its dependence on $\alpha$ and $\theta$.
Define
\begin{align*}
    \rho_\star^2=
    \inf_{\alpha>0,\,\theta\in(0,2)}\rho_\star^2(\alpha,\theta)
\end{align*}
and write $\alpha_\star$ and $\theta_\star$ for the optimal parameters that attain the infimum, if they exist.
%Then $\alpha_\star$ and $\theta_\star$ are the optimal parameters that produce the smallest (best) tight contraction factor.

\begin{figure}
    \centering
    \includegraphics[width=\textwidth]{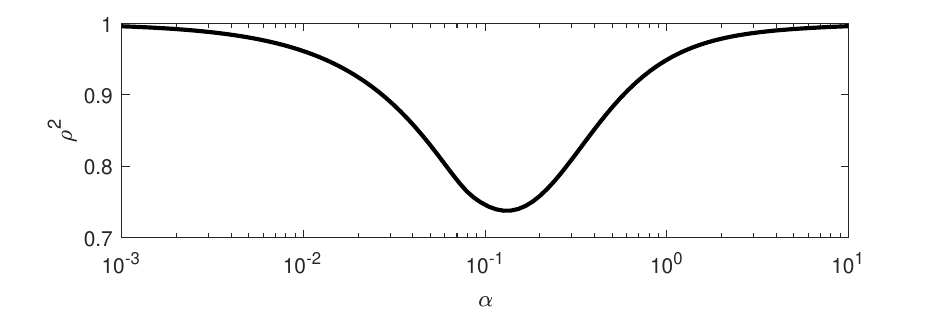}
    \caption{Plot of $\rho_\star^2(\alpha)$ under the assumptions
     $A\in\M_\mu$, $B\in\C_{\beta}\cap \Li_L$, and $C\in \mathcal{C}_{\beta_C}$ with $\mu=1$, $\beta=0.01$, $L=5$, and $\beta_C=9$. The optimal parameters are $\alpha_\star\approx0.131$ and $\theta_\star\approx1.644$, and they produce the optimal contraction factor $\rho^2_\star\approx0.737$. We used Matlab's \texttt{fminunc} for the minimization.
    }
    \label{fig:rho_alpha}
\end{figure}

%$\rho_\star(\alpha,\theta)$ defines the tight contraction factor the the example in Section~\ref{sec:PEP}.
Again, for simplicity of exposition, we limit the generality and consider the operator classes $\Q_1=\M_\mu$, $\Q_2=\C_{\beta}\cap \Li_L$, and $\Q_3=\mathcal{C}_{\beta_C}$, as in Section~\ref{ss:interp-sdp}.
For $\beta\in (0,\infty)$ and $L\in (0,\infty)$, the intersection $\C_{\beta}\cap \Li_L$ is non-degenerate.
So strong duality holds by Theorem~\ref{thm:duality}, and we use the dual OSPEP \eqref{eq:dual-ospep} to write
\begin{align*}
\rho^2_{\star}(\alpha,\theta)=\left(
\begin{tabular}{ll}
$\displaystyle \minimize$ & $ \rho^2$\\
 subject to &$\lambda^A_\mu,\lambda^{B}_\beta,\lambda^{B}_L,\lambda^C_\beta\ge 0$\\
 &$S(\rho^2,\lambda^A_\mu,\lambda^{B}_\beta,\lambda^{B}_L,\lambda^C_\beta,\theta,\alpha)\succeq 0$
\end{tabular}\right),
\end{align*}
where $\rho^2$, $\lambda^A_\mu$, $\lambda^{B}_\beta$, $\lambda^{B}_L$, and $\lambda^C_\beta$ are the optimization variables and $S$ is as in \eqref{eq:Smat-S3}.
Note that 
\begin{align*}
&S(\rho^2,\lambda^A_\mu,\lambda^{B}_\beta,\lambda^{B}_L,\lambda^C_\beta,\theta,\alpha)\\
&\quad=
\begin{pmatrix}
\rho^2+\frac{\lambda^{B}_\beta \beta}{\alpha}+\lambda^{B}_L-1
&\frac{\lambda^A_\mu}{2}
&-\lambda^{B}_\beta(\frac{1}{2}+\frac{\beta}{\alpha})-\lambda^B_L&0\\
\frac{\lambda^A_\mu}{2}&
\lambda^A_\mu(1+\alpha \mu)&-\lambda^A_\mu&\frac{\lambda^A_\mu}{2}\\
-\lambda^{B}_\beta(\frac{1}{2}+\frac{\beta}{\alpha})-\lambda^{B}_L&
-\lambda^A_\mu
&
\lambda^{B}_\beta(\frac{\beta}{\alpha}-1)+\lambda^{B}_L(1-\alpha^2L^2)&-\frac{\lambda^C_\beta}{2}
\\
0&\frac{\lambda^A_\mu}{2}&-\frac{\lambda^C_\beta}{2}&\frac{\lambda^C_\beta \beta_C}{\alpha}
\end{pmatrix}-\begin{pmatrix} 1\\\theta\\-\theta\\0\end{pmatrix}\begin{pmatrix} 1\\\theta\\-\theta\\0\end{pmatrix}^T
\end{align*}
is the Schur complement of
\begin{align*}
&\tilde{S}(\rho^2,\lambda^A_\mu,\lambda^{B}_\beta,\lambda^{B}_L,\lambda^C_\beta,\theta,\alpha)\\
&\quad=\begin{pmatrix}
\rho^2+\frac{\lambda^{B}_\beta\beta}{\alpha}+\lambda^{B}_L-1
&\frac{\lambda^A_\mu}{2}
&-\lambda^{B}_\beta(\frac{1}{2}+\frac{\beta}{\alpha})-\lambda^{B}_L&0&1\\
\frac{\lambda^A_\mu}{2}&
\lambda^A_\mu(1+\alpha \mu)&-\lambda^A_\mu&\frac{\lambda^A_\mu}{2}&\theta\\
-\lambda^{B}_\beta(\frac{1}{2}+\frac{\beta}{\alpha})-\lambda^{B}_L&
-\lambda^A_\mu
&
\lambda^{B}_\beta(\frac{\beta}{\alpha}-1)+\lambda^{B}_L(1-\alpha^2L^2)&-\frac{\lambda^C_\beta}{2}&-\theta
\\
0&\frac{\lambda^A_\mu}{2}&-\frac{\lambda^C_\beta}{2}&\frac{\lambda^C_\beta\beta_C}{\alpha}&0\\
1&\theta&-\theta&0&1
\end{pmatrix}\in \reals^{5\times 5}.
\end{align*}
Therefore $S\succeq 0$ if and only if $\tilde{S}\succeq 0$.
We use $\tilde{S}$ as it depends on $\theta$ linearly.
Define
$\rho^2_{\star}(\alpha)=\inf_{\theta\in(0,2)}\rho^2_\star(\alpha,\theta)$.
%which is the tight contraction factor for a given $\alpha$ with optimal $\theta$.
We evaluate $\rho^2_{\star}(\alpha)$ by solving the SDP
\begin{align*}
\rho^2_{\star}(\alpha)=\left(
\begin{tabular}{ll}
$\displaystyle \minimize$ & $ \rho^2$\\
 subject to &$\lambda^A_\mu,\lambda^{B}_\beta,\lambda^{B}_L,\lambda^C_\beta\ge 0$\\
 &$\tilde{S}(\rho^2,\lambda^A_\mu,\lambda^{B}_\beta,\lambda^{B}_L,\lambda^C_\beta,\theta,\alpha)\succeq 0$
\end{tabular}\right),
\end{align*}
where $\rho^2$, $\lambda^A_\mu$, $\lambda^{B}_\beta$, $\lambda^{B}_L$, $\lambda^C_\beta$, and $\theta$ are the optimization variables.

It remains to solve
\[
\rho_\star^2=\inf_{\alpha>0}\rho_\star^2(\alpha).
\]
The function $\rho^2_\star(\alpha)$ is non-convex in $\alpha$, and it does not seem possible to compute $\rho_\star^2$ with a single SDP.
However, $\rho^2(\alpha)$ seems to be continuous and unimodal for a wide range of operator classes and parameter choices.
Continuity is not surprising.
%and we can show it using arguments similar to the proof of Berge's maximum theorem \cite{berge}.
We do not know whether or why $\rho^2_\star(\alpha)$ is always unimodal.

% we plotted $\rho^2(\alpha)$ for a wide range of operator classes and parameter choices and $\rho^2(\alpha)$ was unimodal in all cases.

To minimize the apparently continuous univariate unimodal function, we use Matlab's derivative free optimization (DFO) solver \verb|fminunc|.
We provide a routine that evaluates $\rho^2_\star(\alpha)$ by solving an SDP, and the DFO solver calls it to evaluate $\rho^2_\star(\alpha)$ at various values of $\alpha$.
Figure~\ref{fig:rho_alpha} shows an example of the function $\rho^2_\star(\alpha)$, and its minimizer was approximated with this approach.
In Figure~\ref{fig:plot_array}, we plot $\rho_\star^2(\alpha)$ under several assumptions. In all cases, $\rho_\star^2(\alpha)$ is continuous and unimodal.

%The minimizer was found with the described approach with Matlab's \verb|fminunc|.

\begin{figure}
\centering
%(a)
\begin{subfigure}[b]{0.32\textwidth}
    \includegraphics[width=\textwidth]{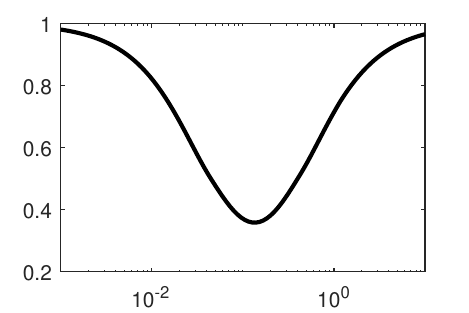}
    \caption{\footnotesize$\mu_A=1,\beta_A=0.07,$\\$\mu_B=4,\beta_B=0.02,\beta_C=9,$\\$\alpha_\star\approx0.13,\theta_\star\approx1.57,\rho^2_\star\approx0.36$}
\end{subfigure}
%(b)
\begin{subfigure}[b]{0.32\textwidth}
    \includegraphics[width=\textwidth]{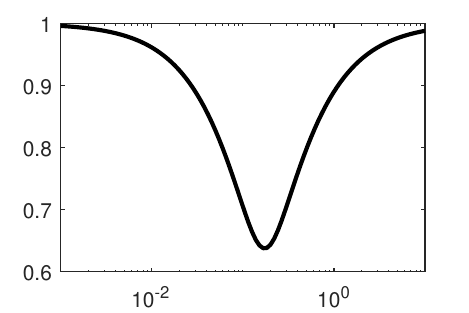}
    \caption{\footnotesize$\mu_A=1,\beta_B=0.03,C=0,$\\$\alpha_\star\approx0.17,\theta_\star\approx1.65,\rho^2_\star\approx0.64$\newline}
\end{subfigure}
%(c)
\begin{subfigure}[b]{0.32\textwidth}
    \includegraphics[width=\textwidth]{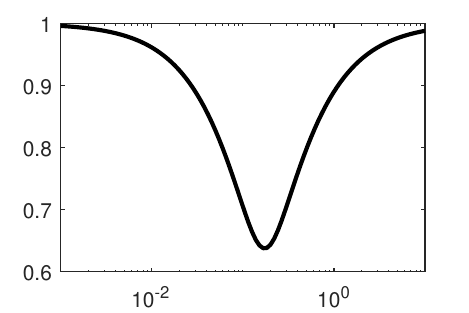}
    \caption{\footnotesize$\beta_A=0.03,\mu_B=1,C=0,$\\$\alpha_\star\approx0.17,\theta_\star\approx1.65,\rho^2_\star\approx0.64$\newline}
\end{subfigure}
%(d)
\begin{subfigure}[b]{0.32\textwidth}
    \includegraphics[width=\textwidth]{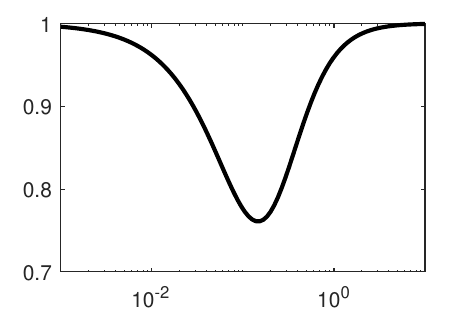}
    \caption{\footnotesize$\mu_A=1,L_B=4,C=0,$\\$\alpha_\star\approx0.15,\theta_\star\approx1.59,\rho^2_\star\approx0.76$\newline}
\end{subfigure}
%(e)
\begin{subfigure}[b]{0.32\textwidth}
    \includegraphics[width=\textwidth]{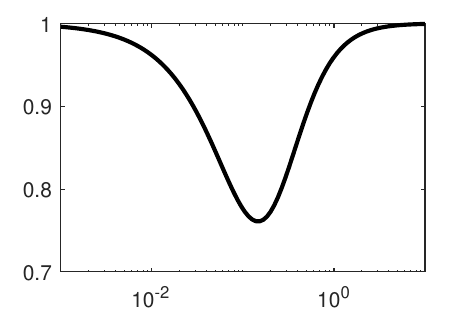}
    \caption{\footnotesize$L_A=4,\mu_B=1,C=0,$\\$\alpha_\star\approx0.15,\theta_\star\approx1.59,\rho^2_\star\approx0.76$\newline}
\end{subfigure}
%(f)
\begin{subfigure}[b]{0.32\textwidth}
    \includegraphics[width=\textwidth]{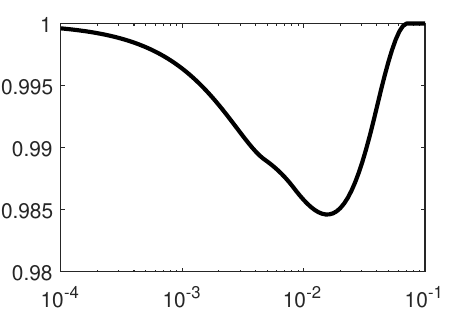}
    \caption{\footnotesize$A=0,\mu_B=1,\beta_B=0.1$\\$L_C=8,\alpha_\star\approx0.016,\theta_\star\approx1$\\$\rho^2_\star\approx0.98$}
\end{subfigure}
%(g)
\begin{subfigure}[b]{0.32\textwidth}
    \includegraphics[width=\textwidth]{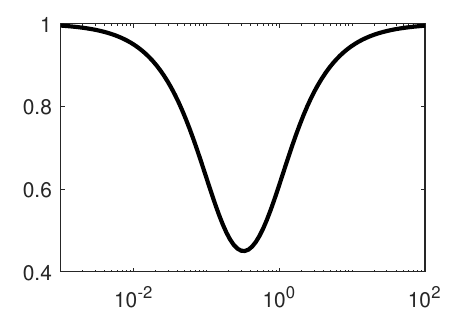}
    \caption{\footnotesize$\mu_A=1,\beta_A=0.07,L_A=7,$\\
    $\mu_B=0.03,\beta_B=0.02,L_A=2,$\\
    $\mu_C=0.01,\beta_C=9,L_C=0.05,$\\$\alpha_\star\approx0.32,\theta_\star\approx1.98,\rho^2_\star\approx0.45$}
\end{subfigure}
%(h)
\begin{subfigure}[b]{0.32\textwidth}
    \includegraphics[width=\textwidth]{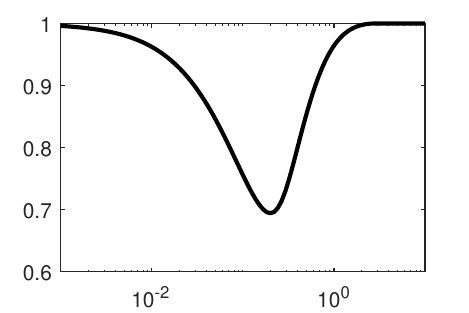}
    \caption{\footnotesize$A=0,\mu_B=1,\beta_C=0.1,$\\$\alpha_\star\approx0.2,\theta_\star\approx1,\rho^2_\star\approx0.69$\newline\newline}
\end{subfigure}\
%(i)
\begin{subfigure}[b]{0.32\textwidth}
    \includegraphics[width=\textwidth]{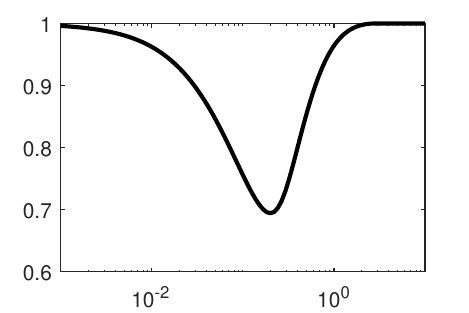}
    \caption{\footnotesize$\mu_A=1,B=0,\beta_C=0.1,$\\$\alpha_\star\approx0.2,\theta_\star\approx1,\rho^2_\star\approx0.69$\newline\newline}
\end{subfigure}
% \begin{subfigure}[b]{0.31\textwidth}
%     \includegraphics[width=\textwidth]{temp_plot}
%     \caption{\footnotesizeXXX placeholder XXX}
% \end{subfigure}
% \begin{subfigure}[b]{0.31\textwidth}
%     \includegraphics[width=\textwidth]{temp_plot}
%     \caption{\footnotesizeXXX placeholder XXX}
% \end{subfigure}
% \begin{subfigure}[b]{0.31\textwidth}
%     \includegraphics[width=\textwidth]{temp_plot}
%     \caption{\footnotesizeXXX placeholder XXX}
% \end{subfigure}
\caption{Plots of $\rho_\star^2(\alpha)$ under various assumptions. The plots are unimodal in all cases.
All operator classes are subsets of $\mathcal{M}$, and only the parameters used in the intersection are specified.
For example, subfigure (e) uses the classes
$Q_1=\mathcal{M}\cap \mathcal{L}_{L_A}$,
$Q_2=\mathcal{M}_{\mu_B}$, and
$Q_3=\{0\}$.
}
\label{fig:plot_array}
\end{figure}

\section{Conclusion}\label{sec:conclusion}
In this work, we presented the OSPEP methodology, proved its tightness,
and demonstrated its value by presenting two applications of it.
The first application was to prove tight analytic contraction factors for DRS and the second was to provide a method for automatic optimal parameter selection.

\paragraph{Code} 
With this paper, we release the following code:
Matlab script implementing OSPEP for FBS, DRS, and DYS;
Matlab script used to plot the figures of Section~5; and
Mathematica script to help readers verify the algebra of Section~\ref{s:main-proof}.
The code uses YALMIP \cite{Lofberg2004} and Mosek \cite{mosek} 
and is available at\\
\url{https://github.com/AdrienTaylor/OperatorSplittingPerformanceEstimation}.

For splitting methods applied to convex functions, one can use the Matlab toolbox PESTO~\cite{taylor2017performance},
available at\\
 \url{https://github.com/AdrienTaylor/Performance-Estimation-Toolbox}.

\section*{Acknowledgements}
Collaborations between the authors started during the LCCC Focus Period on Large-Scale and Distributed Optimization organized by the Automatic Control Department of Lund University. The authors thank the organizers and the other participants. Among others, we thank Laurent Lessard for insightful discussions on the topics of DRS and computer-assisted proofs.
Ernest Ryu was supported in part by NSF grant DMS-1720237 and ONR grant N000141712162.
Adrien Taylor was supported by the European Research Council (ERC) under the European Union's Horizon 2020 research and innovation program (grant agreement 724063). Pontus Giselsson was supported by the Swedish Foundation for Strategic Research and the Swedish Research Council.

\bibliographystyle{siamplain}
\bibliography{references}

%copy-paste supplemental material here and change the counter as follows.
%\newpage
\mbox{ }
\vspace{0.05in}
\begin{center}
\Large
\textbf{Appendix}
\vspace{0.25in}
\end{center}
\setcounter{section}{0}
\renewcommand{\thesection}{SM\arabic{section}}

%%\section*{Table of contents}
%{\Large \noindent\textbf{Contents}}
%\vspace{0.1in}
%\begin{itemize}[leftmargin=0.2in]
%\item \ref{appendix:full_primal_ospep}. Full primal OSPEP
%\item \ref{appendix:full_dual_ospep}. Full dual OSPEP
%\item \ref{s:main-proof}. Proofs of results in Section~\ref{sec:tight_DRS}
%\begin{itemize}[leftmargin=0.2in]
%\item
%\ref{ss:proof_thm_drs_coco_smonotone}. Proof of Theorem~\ref{thm:DRS_coco_strmonotone}
%\begin{itemize}[leftmargin=0.2in]
%\item \ref{ss:first_proof_upper}. Upper bounds
%\item \ref{sss:lower1}. Lower bounds
%\end{itemize}
%\item \ref{ss:proof_thm_drs_lip-smon}. Proof of Theorem~\ref{thm:DRS_Lipschitz_strmonotone}
%\begin{itemize}[leftmargin=0.2in]
%\item\ref{sec:upper_bound_Lips}. Upper bounds
%\item \ref{sss:lower2}. Lower bounds
%\end{itemize}
%\end{itemize}
%\item
%\ref{ss:verification}. Algebraic verification of inequalities
%\begin{itemize}[leftmargin=0.2in]
%\item \ref{ss:thm3_ineq}. Inequalities for Theorem~\ref{thm:DRS_coco_strmonotone}
%\begin{itemize}[leftmargin=0.2in]
%\item \ref{ss:theorem3_verification}. Upper bounds
%\item \ref{sss:lower3}. Lower bounds
%\end{itemize}
%\ref{ss:thm4ineq}. Inequalities for Theorem~\ref{thm:DRS_Lipschitz_strmonotone}
%\begin{itemize}[leftmargin=0.2in]
%\item \ref{ss:lip-mon-upper}. Upper bounds
%\item\ref{ss:lowerbound-Lipschitz-monotone}. Lower bounds
%\end{itemize}
%\end{itemize}
%\item \ref{s:visualize}. Visualization
%\end{itemize}

%\newpage
\section{Full primal OSPEP}
\label{appendix:full_primal_ospep}
We state the full primal OSPEP with the operator classes
$\Q_1=\mathcal{M}_{\mu_A}\cap \mathcal{C}_{\beta_A}\cap \mathcal{L}_{L_A}$,
$\Q_2=\mathcal{M}_{\mu_B}\cap \mathcal{C}_{\beta_B}\cap \mathcal{L}_{L_B}$,
and 
$\Q_3=\mathcal{M}_{\mu_C}\cap \mathcal{C}_{\beta_C}\cap \mathcal{L}_{L_C}$.
The primal OSPEP with fewer assumptions will be of an analogous form with fewer constraints.

\begin{align*}
\begin{tabular}{ll}
 $\displaystyle \maximize_{}$ & $ \tr(M_{O}G)$\\
 subject to 
    &$ \tr(M_{\mu}^AG)\ge 0$,\,\, $ \tr(M_{\beta}^AG)\ge 0$,\,\, $\tr(M_{L}^AG)\ge 0$\\
    &$ \tr(M_{\mu}^BG)\ge 0$\,\, ,$ \tr(M_{\beta}^BG)\ge 0$,\,\, $\tr(M_{L}^BG)\ge 0$\\
    &$ \tr(M_{\mu}^CG)\ge 0$,\,\, $ \tr(M_{\beta}^CG)\ge 0$,\,\, $ \tr(M_{L}^CG)\ge 0$\\
    &$ \tr(M_{I}G)=1$\\
    &$G\succeq 0$
\end{tabular}
\end{align*}
where $G\in \mathbb{S}^4_+$ is the optimization variable and
\begin{gather*}
M_I=
\begin{pmatrix}
1&0&0&0\\
0&0&0&0\\
0&0&0&0\\
0&0&0&0
\end{pmatrix},
\qquad
M_O=
\begin{pmatrix}
1&\theta&-\theta&0\\
\theta&\theta^2&-\theta^2&0\\
-\theta&-\theta^2&\theta^2&0\\
0&0&0&0
\end{pmatrix}\\
M^A_\mu=
\begin{pmatrix}
 0 & -\frac{1}{2} & 0 & 0 \\
 -\frac{1}{2} & -\alpha  \mu_A-1 & 1 & -\frac{1}{2} \\
 0 & 1 & 0 & 0 \\
 0 & -\frac{1}{2} & 0 & 0 \\
\end{pmatrix},
\qquad
M^A_\beta=
\begin{pmatrix}
 -\tfrac{\beta_A}{\alpha } & -\tfrac{\beta_A}{\alpha }-\tfrac{1}{2} & \tfrac{2 \beta_A}{\alpha } & -\tfrac{\beta_A}{\alpha } \\
 -\tfrac{\beta_A}{\alpha }-\tfrac{1}{2} & -\tfrac{\beta_A}{\alpha }-1  & \tfrac{2 \beta_A}{\alpha }+1 & -\tfrac{\beta_A}{\alpha }-\tfrac{1}{2} \\
 \tfrac{2 \beta_A}{\alpha } & \tfrac{2 \beta_A}{\alpha }+1 & -\tfrac{4 \beta_A}{\alpha } & \tfrac{2 \beta_A}{\alpha } \\
 -\tfrac{\beta_A}{\alpha } & -\tfrac{\beta_A}{\alpha }-\tfrac{1}{2} & \tfrac{2 \beta_A}{\alpha } & -\tfrac{\beta_A}{\alpha } 
\end{pmatrix},\\
M^A_L=
\begin{pmatrix}
 -1 & -1 & 2 & -1 \\
 -1 & \alpha ^2L_A^2 -1 & 2 & -1 \\
 2 & 2 & -4 & 2 \\
 -1 & -1 & 2 & -1
\end{pmatrix},
\qquad
M^B_\mu=
\begin{pmatrix}
 0 & 0 & \tfrac{1}{2} & 0 \\
 0 & 0 & 0 & 0 \\
 \tfrac{1}{2} & 0 & -\alpha  \mu_B-1 & 0 \\
 0 & 0 & 0 & 0
\end{pmatrix},\\
M^B_\beta=
\begin{pmatrix}
 -\tfrac{\beta_B}{\alpha } & 0 & \tfrac{\beta_B}{\alpha }+\tfrac{1}{2} & 0 \\
 0 & 0 & 0 & 0 \\
 \tfrac{\beta_B}{\alpha }+\tfrac{1}{2} & 0 & -\tfrac{\beta_B}{\alpha }-1 & 0 \\
 0 & 0 & 0 & 0 
\end{pmatrix},
\qquad
M^B_L=
\begin{pmatrix}
 -1 & 0 & 1 & 0 \\
 0 & 0 & 0 & 0 \\
 1 & 0 &  \alpha ^2L_B^2-1 & 0 \\
 0 & 0 & 0 & 0
\end{pmatrix},\\
M^C_\mu=
\begin{pmatrix}
 0 & 0 & 0 & 0 \\
 0 & 0 & 0 & 0 \\
 0 & 0 & -\alpha \mu_C & \tfrac{1}{2} \\
 0 & 0 & \tfrac{1}{2} & 0
\end{pmatrix},
\qquad
M^C_\beta=
\begin{pmatrix}
 0 & 0 & 0 & 0 \\
 0 & 0 & 0 & 0 \\
 0 & 0 & 0 & \tfrac{1}{2} \\
 0 & 0 & \tfrac{1}{2} & -\tfrac{\beta_C}{\alpha }
\end{pmatrix},
\qquad
M^C_L=
\begin{pmatrix}
 0 & 0 & 0 & 0 \\
 0 & 0 & 0 & 0 \\
 0 & 0 & \alpha ^2L_C^2 & 0 \\
 0 & 0 & 0 & -1
\end{pmatrix}.
\end{gather*}

The objective $\tr(M_{O}G)$ corresponds to $\|z-\theta(z_B-z_A)-z'+\theta(z_B'-z_A')\|^2$.
The equality constraint $ \tr(M_{I}G)=1$ corresponds to $\|z-z'\|^2= 1$.
The other $9$ inequality constraints correspond to the three assumptions on the three operators. In particular, 
$\tr(M_\mu^AG)\ge 0$, $\tr(M_\beta^AG)\ge 0$, and $\tr(M_L^AG)\ge 0$ respectively correspond to the $\mu$-strong monotonicity, $\beta$-cocoercivity, and $L$-Lipschitz continuity assumptions on $A$ respectively.
The assumptions on $B$ and $C$ have analogous correspondences.

\section{Full dual OSPEP}
\label{appendix:full_dual_ospep}
We state the full dual OSPEP with the same operator classes as in Section~\ref{appendix:full_primal_ospep}.
The dual OSPEP with fewer assumptions will be of an analogous form with fewer $\lambda$-variables.

\begin{equation*}
\begin{tabular}{ll}
 $\displaystyle \minimize{}$ & $ \rho^2$\\
 subject to & 
 $\lambda^A_\mu, \lambda^A_\beta, \lambda^A_L\ge 0$\\
 &$\lambda^B_\mu, \lambda^B_\beta, \lambda^B_L\ge 0$\\
 &$\lambda^C_\mu, \lambda^C_\beta, \lambda^C_L\ge 0$\\
 &$S(\rho^2,\lambda^A_\mu, \lambda^A_\beta, \lambda^A_L, 
 \lambda^B_\mu, \lambda^B_\beta, \lambda^B_L, 
 \lambda^C_\mu, \lambda^C_\beta, \lambda^C_L,\theta,\alpha)\succeq 0$
\end{tabular}
\end{equation*}
where $\rho^2,\lambda^A_\mu, \lambda^A_\beta, \lambda^A_L, 
 \lambda^B_\mu, \lambda^B_\beta, \lambda^B_L, 
 \lambda^C_\mu, \lambda^C_\beta, \lambda^C_L\in \reals$ are the optimization variables and
\begin{equation*}
    \begin{aligned}
S(\rho^2,\lambda^A_\mu, \lambda^A_\beta, \lambda^A_L, 
 \lambda^B_\mu, \lambda^B_\beta, \lambda^B_L, 
 \lambda^C_\mu, \lambda^C_\beta, \lambda^C_L,\theta,\alpha)=\rho^2 M_I - M_O&-\lambda_\mu^A M_\mu^A-\lambda_\beta^A M_\beta^A-\lambda_L^A M_L^A\\
    &-\lambda_\mu^B M_\mu^B-\lambda_\beta^B M_\beta^B-\lambda_L^B M_L^B\\
    &-\lambda_\mu^C M_\mu^C-\lambda_\beta^C M_\beta^C-\lambda_L^C M_L^C
    \end{aligned}
\end{equation*}
is symmetric. The matrix can also explicitly be written as
\[
S(\rho^2,\lambda^A_\mu, \lambda^A_\beta, \lambda^A_L, 
\lambda^B_\mu, \lambda^B_\beta, \lambda^B_L, 
\lambda^C_\mu, \lambda^C_\beta, \lambda^C_L,\theta,\alpha)=\begin{pmatrix}
S_{1,1} & S_{2,1} & S_{3,1} & S_{4,1} \\
S_{2,1} & S_{2,2} & S_{3,2} & S_{4,2} \\
S_{3,1} & S_{3,2} & S_{3,3} & S_{4,3} \\
S_{4,1} & S_{4,2} & S_{4,3} & S_{4,4}
\end{pmatrix}
\]
with
\begin{align*}
S_{1,1}&=\rho ^2-1+\tfrac{\beta_A}{\alpha} \lambda_\beta^A+\tfrac{\beta_B}{\alpha} \lambda_\beta^B+\lambda_L^A+\lambda_L^B,\\
S_{2,1}&=\tfrac{1}{2} (2 \tfrac{\beta_A}{\alpha} \lambda_\beta^A-2 \theta +\lambda_\beta^A+2 \lambda_L^A+\lambda_\mu^A),\\
S_{3,1}&=-2 \tfrac{\beta_A}{\alpha} \lambda_\beta^A-\tfrac{\beta_B}{\alpha} \lambda_\beta^B+\theta -2 \lambda_L^A-\tfrac{\lambda_\beta^B}{2}-\lambda_L^B-\tfrac{\lambda_\mu^B}{2},\\
S_{4,1}&=\tfrac{\beta_A}{\alpha} \lambda_\beta^A+\lambda_L^A,\\
S_{2,2}&=\tfrac{\beta_A}{\alpha} \lambda_\beta^A-\theta ^2+\lambda_\beta^A+\lambda_L^A+\lambda_\mu^A \alpha\mu_A+\lambda_\mu^A-\lambda_L^A \alpha^2 L_A^2,\\
S_{3,2}&=-(2 \tfrac{\beta_A}{\alpha}+1) \lambda_\beta^A+\theta ^2-2 \lambda_L^A-\lambda_\mu^A,\\
S_{4,2}&=\tfrac{1}{2} (2 \tfrac{\beta_A}{\alpha} \lambda_\beta^A+\lambda_\beta^A+2 \lambda_L^A+\lambda_\mu^A),\\
S_{3,3}&=4 \tfrac{\beta_A}{\alpha} \lambda_\beta^A+\tfrac{\beta_B}{\alpha} \lambda_\beta^B-\theta ^2+4 \lambda_L^A+\lambda_\beta^B+\lambda_L^B+\lambda_\mu^B \alpha\mu_B+\lambda_\mu^B+\lambda_\mu^C \alpha\mu_C-\lambda_L^B \alpha^2 L_B^2-\lambda_L^C \alpha^2 L_C^2,\\
S_{4,3}&=\tfrac{1}{2} (-4 \tfrac{\beta_A}{\alpha} \lambda_\beta^A-4 \lambda_L^A-\lambda_\beta^C-\lambda_\mu^C),\\
S_{4,4}&=\tfrac{\beta_A}{\alpha} \lambda_\beta^A+\tfrac{\beta_C}{\alpha} \lambda_\beta^C+\lambda_L^A+\lambda_L^C.
\end{align*}

\section{Proofs of results in Section~\ref{sec:tight_DRS}}
\label{s:main-proof}
We now prove Theorems~\ref{thm:DRS_coco_strmonotone} and \ref{thm:DRS_Lipschitz_strmonotone}.
The approach is to provide an upper bound and a lower bound for each case (5 cases for Theorem~\ref{thm:DRS_coco_strmonotone} and 3 cases for Theorem~\ref{thm:DRS_Lipschitz_strmonotone}).
Since the upper and lower bounds match, weak duality tells us that the bounds are optimal, i.e., the contraction factors are tight.

In the language of the SDPs, the upper and lower bounds correspond to primal and dual feasible points, and their optimality is certified since they match ($0$ duality gap).
Note that the strong duality result of Theorem~\ref{thm:duality} guarantees the existence of lower bounds matching the optimal upper bounds. Here, we explicitly provide lower bounds to certify the upper bounds are indeed optimal.

The proofs rely on inequalities that we assert by saying ``It is possible to verify that ....''
Whenever we do so, we provide a rigorous (and arduous) verification separately in Section~\ref{ss:verification}.
We make this separation because the verifications are purely algebraic and do not illuminate the main proof.
As an alternative means of verification, we provide code that uses symbolic manipulation to verify the inequalities.

%%%%%%%%%%%%%%%%%%%%%%%%%%%%%%%%%%%%%%%%%%%%%%%%%%%%%%%%%%%%%%%%%%%%%
%%%%%%%%%%%%%%%%%%%%%%%%%%%%%%%%%%%%%%%%%%%%%%%%%%%%%%%%%%%%%%%%%%%%%
%                                                                   %
%                                                                   %
%   PROOF OF THE DRS THEOREM 1 (cocoercive + str monotone)          %
%                                                                   %
%                                                                   %
%%%%%%%%%%%%%%%%%%%%%%%%%%%%%%%%%%%%%%%%%%%%%%%%%%%%%%%%%%%%%%%%%%%%%
%%%%%%%%%%%%%%%%%%%%%%%%%%%%%%%%%%%%%%%%%%%%%%%%%%%%%%%%%%%%%%%%%%%%%
\subsection{Proof of Theorem~\ref{thm:DRS_coco_strmonotone}}
\label{ss:proof_thm_drs_coco_smonotone}
Define
\begin{align*}
R_\mathrm{(a)}&=
\left\{(\mu,\beta,\theta)\,\Big|\,
 \mu\beta-\mu+\beta<0,\, \theta \leq 2 \tfrac{(\beta+1) (\mu-\beta-\mu\beta)}{\mu+\mu\beta-\beta-\beta^2-2\mu\beta^2},\,\mu>0,\,\beta>0,\,\theta\in(0,2)
 \right\}\\
 R_\mathrm{(b)}&=
\left\{(\mu,\beta,\theta)\,\Big|\,
\mu\beta-\mu-\beta>0,\, \theta \leq 2  \tfrac{\mu^2+\beta^2+\mu\beta+\mu+\beta-\mu^2\beta^2}{\mu^2+\beta^2+\mu^2\beta+\mu\beta^2+\mu +\beta-2\mu^2\beta^2},\,\mu>0,\,\beta>0,\,\theta\in(0,2)
\right\}\\
R_\mathrm{(c)}&=
\left\{(\mu,\beta,\theta)\,\Big|\,
\theta \geq 2 \tfrac{\mu\beta+\mu+\beta}{2\mu\beta+\mu+\beta},\,\mu>0,\,\beta>0,\,\theta\in(0,2)
 \right\}\\
 R_\mathrm{(d)}&=
\left\{(\mu,\beta,\theta)\,\Big|\,
 \mu\beta+\mu-\beta<0,\,\theta \leq 2\tfrac{(\mu+1) (\beta-\mu-\mu\beta)}{\beta+\mu\beta-\mu-\mu^2-2\mu^2\beta},\,\mu>0,\,\beta>0,\,\theta\in(0,2)
 \right\}\\
 R_\mathrm{(e)}&=
\left\{(\mu,\beta,\theta)\,\Big|\,
\mu>0,\,\beta>0,\,\theta\in(0,2)
 \right\}\backslash R_\mathrm{(a)}\backslash R_\mathrm{(b)}\backslash R_\mathrm{(c)}\backslash R_\mathrm{(d)}
\end{align*}
which correspond to the 5 cases of Theorem~\ref{thm:DRS_coco_strmonotone}.

\subsubsection{Upper bounds}
\label{ss:first_proof_upper}
By weak duality between the primal and dual OSPEP, $\rho$ is a valid contraction factor if there exists $\rho$, $\lambda^A_\mu\geq 0$, and $\lambda_\beta^B\geq 0$ such that
%\begin{equation}\label{eq:S_DRS_coco_strmonotone}
\[
S=\begin{pmatrix}
\rho ^2+\beta  \lambda_\beta^B-1 & -\theta+\frac{\lambda^A_\mu}{2}  & \theta -(\frac{1}{2}+\beta)\lambda_\beta^B \\
 -\theta+\frac{\lambda^A_\mu}{2}  & -\theta ^2+(1+\mu)\lambda^A_\mu  & \theta ^2-\lambda^A_\mu \\
 \theta -(\frac{1}{2}+\beta)\lambda_\beta^B & \theta ^2-\lambda^A_\mu & -\theta ^2+(1+\beta)\lambda_\beta^B 
\end{pmatrix}
\succeq 0
.
\]
For each of the 5 cases, we establish an upper bound by providing values for  $\rho$, $\lambda^A_\mu\ge 0$, and $\lambda_\beta^B\ge 0$ such that $S\succeq 0$.
We establish $S\succeq 0$ with a sum-of-squares factorization 
\begin{equation}
\label{eq:proof_thm_A_strmonotone_B_cocoercive}
\mathrm{Tr}(SG(z,z_A,z_B))=K_1\normsq{m_1 z_A + m_2 z_B +m_3 z}+K_2 \normsq{m_4 z_B+m_5 z},
\end{equation}
for some $m_1,m_2,m_3,m_4,m_5\in \reals$ and $K_1,K_2\ge0$,
where
\begin{equation}
G(z,z_A,z_B)=
\begin{pmatrix}
\|z\|^2&\langle z,z_A\rangle &\langle z,z_B\rangle\\
\langle z,z_A\rangle &\|z_A\|^2&\langle z_A,z_B\rangle\\
\langle z,z_B\rangle&\langle z_A,z_B\rangle &\|z_B\|^2\\
\end{pmatrix}\in \mathbb{S}^3_+
\label{eq:grammian_small}
\end{equation}
for $z,z_A,z_B\in \hilbert$.
By arguments similar to that of Lemma~\ref{lem:cholesky},
 $G(z,z_A,z_B)\in\mathbb{S}^3_+$ can be any $3\times 3$ positive semidefinite matrix.
Therefore
\[
\mathrm{Tr}(SG(z,z_A,z_B))\ge 0,\,
\forall
z,z_A,z_A\in \hilbert
\quad\Leftrightarrow\quad
\mathrm{Tr}(SM)\ge 0,\,
\forall
M\succeq 0
\quad\Leftrightarrow\quad
S\succeq 0,
\]
i.e., the sum-of-squares factorization proves $S\succeq 0$.
(We only need 2 terms in the sum-of-squares factorization, because it turns out that the optimal $S$ has rank at most $2$.)

%Here $G$ is defined as in \eqref{eq:grammian}.

		\paragraph{Case (a)}
		When $(\mu,\beta,\theta)\in R_\mathrm{(a)}$, we use
		 \[
		 \rho^2=\left(1-\theta\tfrac{\beta}{\beta+1}\right)^2,\quad\lambda^A_\mu=2\theta\tfrac{1+\beta}{1-\beta}\left(1-\theta\tfrac{\beta}{\beta+1}\right),\quad\lambda_\beta^B=2 \theta \left(1-\theta\tfrac{\beta}{\beta+1}\right).
		 \]
		 This gives us the sum-of-square factorization \eqref{eq:proof_thm_A_strmonotone_B_cocoercive} with
		\begin{gather*}
		m_1=-1,\quad	m_2=\tfrac{(2-\theta )(\beta +1) }{(2-\theta )(\beta +1) +2 \mu  (1+\beta-\theta\beta)},\quad
		m_3=-\tfrac{(2-\theta)\beta}{(2-\theta )(\beta +1) +2 \mu  (1+\beta-\theta\beta)},\\
		m_4=-\tfrac{\beta +1}{\beta },
		 m_5=1,\quad K_1=\theta\tfrac{(2-\theta)(\beta+1)+2\mu(1+\beta-\theta\beta)}{1-\beta},\quad	K_2=2 \beta ^2 \theta  \tfrac{1+\beta-\theta \beta}{(1-\beta) (\beta +1)^2}\tfrac{(\beta -1) \mu  (2 \beta  (\theta -1)+\theta -2)-(2-\theta )\beta  (\beta+1)}{(2-\theta )(\beta+1)+2 \mu  (1+\beta-\theta\beta )}.
		\end{gather*}
		It is possible to verify that there is no division by $0$ in the definitions and that $\lambda^A_\mu,\lambda_\beta^B,K_1,K_2\geq 0$ when $(\mu,\beta,\theta)\in R_\mathrm{(a)}$.
		
\paragraph{Case (b)}
		When $(\mu,\beta,\theta)\in R_\mathrm{(b)}$, we use
		\[\rho^2=\left(1-\theta\tfrac{1+\mu\beta}{(\mu+1)(\beta+1)}  \right)^2,\quad
		\lambda^A_\mu=2 \theta \tfrac{\beta+1}{\beta-1} \left(1-\theta\tfrac{1+\mu\beta}{(\mu +1) (\beta+1)}\right),\quad
		\lambda_\beta^B= 2 \theta  \tfrac{\mu-1}{\mu+1}\left(1-\theta\tfrac{1+\mu\beta}{(\mu +1) (\beta+1)}\right).
		\]
		 This gives us the sum-of-square factorization \eqref{eq:proof_thm_A_strmonotone_B_cocoercive} with
		\begin{gather*}
		m_1=-1,\quad	m_2=\tfrac{2}{\mu +1}-\tfrac{(2-\theta)(\beta +1) }{(2-\theta) (\beta+1)+2 \mu  (1+\beta-\theta\beta)},\quad m_3=-\tfrac{1}{\mu +1}+\tfrac{(2-\theta )\beta}{(2-\theta) (\beta+1)+2 \mu  (1+\beta-\theta\beta)},\\
		m_4=-\tfrac{\beta +1}{\beta },\quad 	m_5=1,  \quad	K_1=\theta\tfrac{(2-\theta)(\beta+1)+2\mu(1+\beta-\theta\beta)}{\beta-1},\\
		K_2=2 \theta\beta^2 \tfrac{(\beta+1)(\mu+1)-\theta(1+\mu\beta)}{(\mu +1)^2 (\beta-1) (\beta+1)^2}\, \tfrac{ \mu\beta^2  (-2 \theta  \mu +\theta +2 \mu )+\theta\beta^2 +(\theta -2) \mu  (\mu +1) +\beta \left(\theta  \mu ^2+\theta -2 \mu -2\right)-2\beta^2}{(2-\theta) (\beta+1)+2 \mu  (1+\beta-\theta\beta )}.
		\end{gather*}
		It is possible to verify that there is no division by $0$ in the definitions and that $\lambda^A_\mu,\lambda_\beta^B,K_1,K_2\geq 0$ when $(\mu,\beta,\theta)\in R_\mathrm{(b)}$.

\paragraph{Case (c)}
		When $(\mu,\beta,\theta)\in R_\mathrm{(c)}$, we use
		 \[
		 \rho^2=(\theta -1)^2,\quad \lambda^A_\mu= 2\theta (\theta -1),\quad\lambda_\beta^B= 2\theta (\theta -1)
		 \]
		 This gives us the sum-of-square factorization \eqref{eq:proof_thm_A_strmonotone_B_cocoercive} with
		 \begin{gather*}
		m_1=-1,\quad m_2=-\tfrac{2-\theta }{2 (\theta -1) \mu +\theta -2}, \quad m_3=\tfrac{2-\theta }{2 (\theta -1) \mu +\theta -2},\quad m_4=-1,\quad m_5=1,\\
		K_1=\theta  (2 (\theta -1) \mu +\theta -2),\quad K_2=2 (\theta -1) \theta \tfrac{ \theta\beta +\theta  \mu  (1+2\beta)-2 (\mu +\beta+\mu\beta)}{2 (\theta -1) \mu +\theta -2}.
		\end{gather*}
		It is possible to verify that there is no division by $0$ in the definitions and  that $\lambda^A_\mu,\lambda_\beta^B,K_1,K_2\geq 0$ when $(\mu,\beta,\theta)\in R_\mathrm{(c)}$.

\paragraph{Case (d)}
		When $(\mu,\beta,\theta)\in R_\mathrm{(d)}$, we use
		\[
		\rho^2=\left(1-\theta\tfrac{\mu}{\mu+1}\right)^2,\quad \lambda^A_\mu=2\theta\left(1-\theta\tfrac{\mu}{\mu +1}\right),\quad\lambda_\beta^B=2 \theta  \tfrac{1-\mu}{1+\mu}\left(1-\theta\tfrac{ \mu}{\mu +1}\right).
		\]
		 This gives us the sum-of-square factorization \eqref{eq:proof_thm_A_strmonotone_B_cocoercive} with
		 \begin{gather*}
		m_1=-1, \quad m_2=\tfrac{2-\theta }{2 (\theta -1) \mu +\theta -2}+\tfrac{2}{\mu +1},\quad m_3=-\tfrac{\theta  \mu }{(\mu +1) (2 (\theta -1) \mu +\theta -2)},\quad m_4=-1,\quad m_5=1,\\
		K_1=-\theta  (2 (\theta -1) \mu +\theta -2),\quad K_2=2 \theta  ((\theta -1) \mu -1)\tfrac{ -\theta  (\beta-\mu^2 (1+2\beta) -\mu(1-\beta))-2 (\mu +1) (\mu\beta +\mu  -\beta)}{(\mu +1)^2 (2 (\theta -1) \mu +\theta -2)}.
		\end{gather*}
		It is possible to verify that there is no division by $0$ in the definitions and  that $\lambda^A_\mu,\lambda_\beta^B,K_1,K_2\geq 0$ when $(\mu,\beta,\theta)\in R_\mathrm{(d)}$.

\paragraph{Case (e)}
%\ER{I factored the denominator of $\rho^2$ }
		When $(\mu,\beta,\theta)\in R_\mathrm{(e)}$, we use
		\begin{gather*}
		\rho^2=\tfrac{2-\theta}{4} {\tfrac{((2-\theta)\mu (\beta+1)+{\theta}  \beta(1-\mu))\, ((2-\theta)\beta ({\mu}+1)+{\theta}\mu  (1-\beta))}{\mu\beta(2\mu \beta (1-\theta )  + (2-\theta ) (\mu+\beta +1))}}
		,\quad\lambda^A_\mu=\theta\tfrac{(2-\theta)\mu(\beta+1)+\theta\beta(1-\mu)}{\beta},\\
		\lambda_\beta^B=\tfrac{\theta  (2-\theta )}{\beta}\tfrac{(2-\theta ) \mu(\beta +1)   +\theta\beta     (1-\mu )}{2\mu \beta (1-\theta )  + (2-\theta ). (\mu+\beta +1)}.
		\end{gather*}
		 This gives us the sum-of-square factorization \eqref{eq:proof_thm_A_strmonotone_B_cocoercive} with
		 \begin{gather*}
		m_1=-2 \mu  \tfrac{2 (\theta -1) \mu\beta +\theta\beta +(\theta -2) (\mu +1)-2\beta}{\beta},\quad m_2=2 \mu \tfrac{ (\beta +1) (\theta -2)+\beta  \theta }{\beta },\quad m_3=(\theta -2)-\mu\tfrac{  (\beta +1) (\theta -2)+\beta  \theta }{\beta }\\
		m_4=0,\quad m_5=0,\quad 	K_1=\tfrac{\beta}{4\mu}\,\,\tfrac{\theta }{2 \mu\beta  (1-\theta) +(2-\theta) (\mu +\beta+1)},\quad K_2=0.
		\end{gather*}
		It is possible to verify that there is no division by $0$ in the definitions and
		 that $\rho^2,\lambda^A_\mu,\lambda_\beta^B,K_1,K_2\geq 0$ when $(\mu,\beta,\theta)\in R_\mathrm{(e)}$.

\begin{rem}[Constructing a classical proof with a dual solution]
\label{rem:classical_proof}
Given $\rho^2$, $\lambda^A_\mu\ge 0$, and $\lambda_\beta^B\ge0$ such that $S\succeq 0$, one can construct a classical proof establishing $\rho^2$ as a valid contraction factor without relying on the OSPEP methodology.
With $G$ defined as in \eqref{eq:grammian_small}, we have
\[
\mathrm{Tr}(SG(z,z_A,z_B))=\rho^2 \normsq{z}-\normsq{z-\theta(z_B-z_A)}-\lambda^A_\mu \bigg( \inner{\Delta A}{z_A}-\mu \normsq{z_A}\bigg)-\lambda_\beta^B \bigg( \inner{\Delta B}{z_B}- \beta \normsq{\Delta B}\bigg),
\]
with $\Delta A=2z_b-z-z_A$ and $\Delta B=z-z_B$.
The sum-of-square factorization gives us
\begin{align*}
&\rho^2 \normsq{z}-\normsq{z-\theta(z_B-z_A)}-	\lambda^A_\mu \bigg( \inner{\Delta A}{z_A}-\mu \normsq{z_A}\bigg) - \lambda_\beta^B \bigg( \inner{\Delta B}{z_B}- \beta \normsq{\Delta B}\bigg)\\
&\qquad\qquad=K_1\normsq{m_1 z_A + m_2 z_B +m_3 z}+K_2 \normsq{m_4 z_B+m_5 z}.
\end{align*}
Reorganizing, we get
\begin{align*}
\normsq{z-\theta(z_B-z_A)}&=
\rho^2 \normsq{z}-	\lambda^A_\mu \bigg( \inner{\Delta A}{z_A}-\mu \normsq{z_A}\bigg) - \lambda_\beta^B \bigg( \inner{\Delta B}{z_B}- \beta \normsq{\Delta B}\bigg)\\
&\qquad\qquad-(\text{sum of squares}).
\end{align*}
Now revert the change of variables of Section~\ref{ss:conv-ospep-transformation} by substituting
$z\mapsto z-z'$,
$z-\theta(z_B-z_A)\mapsto T^\mathrm{DRS}(z)-T^\mathrm{DRS}(z')$,
$z_B\mapsto J_Bz-J_Bz'$, and 
$z_A\mapsto J_A(2J_Bz-z)-J_A(2J_Bz'-z')$
to get a classical proof of the form
\begin{align*}
\|T^\mathrm{DRS}(z;A,B,1,\theta)&-T^\mathrm{DRS}(z';A,B,1,\theta)\|^2
= \rho^2 \normsq{z-z'}\\
&-	\lambda^A_\mu \bigg( \inner{\Delta A}{J_A(2J_Bz-z)-J_A(2J_Bz'-z')}-\mu \normsq{J_A(2J_Bz-z)-J_A(2J_Bz'-z')}\bigg)\\
&- \lambda_\beta^B \bigg( \inner{\Delta B}{J_Bz-J_Bz'}- \beta \normsq{\Delta B}\bigg)\\
&-(\text{sum of squares})
\end{align*}
where now $\Delta A=2J_Bz-z-J_A(2J_Bz-z)-2J_Bz'+z'+J_A(2J_Bz'-z')$ and $\Delta B=z-J_Bz-z'+J_Bz'$.
Since $A$ is $\mu$-strong monotone, we have
\[
\inner{\Delta A}{J_A(2J_Bz-z)-J_A(2J_Bz'-z')}-\mu \normsq{J_A(2J_Bz-z)-J_A(2J_Bz'-z')}\ge 0.
\]
Since $B$ is $\beta$-cocoercive, we have
\[
\inner{\Delta B}{J_Bz-J_Bz'}- \beta \normsq{\Delta B}\ge 0.
\]
Since $\lambda^A_\mu\ge0$ and $\lambda_\beta^B\ge 0$, we have a valid proof establishing
\[
\|T^\mathrm{DRS}(z;A,B,1,\theta)-T^\mathrm{DRS}(z';A,B,1,\theta)\|^2
\le \rho^2 \normsq{z-z'}.
\]
\end{rem}

\subsubsection{Lower bounds}
\label{sss:lower1}
We now show that for the five cases, there are operators $A\in\M_\mu$ and $B\in\C_\beta$ and inputs $z_1,z_2\in\hilbert$ such that
\[ \norm{Tz_1-Tz_2}\geq \rho \norm{z_1-z_2},\]
where  $T=I-\theta J_B + \theta J_A (2J_B-I)$ % the overrelaxed Douglas-Rachford operator.
and $\rho$ is given by Theorem~\ref{thm:DRS_coco_strmonotone}.
We construct the lower bounds for $\reals^2$, since the construction can be embedded into the higher dimensional space $\hilbert$.
%When $\hilbert\ge 3$, we can embed this 2-dimensional example into the higher dimensional space.

% with either $\rho=\rho_1$ or $\rho=\rho_2$ as below:
% \begin{itemize}
%     \item (lower bounds (a) to (d) are valid for all values $\mu,L,\theta>0$) $\rho=\rho_1$ with \[\rho_1=\max\{\lvert1-\theta\tfrac{\beta}{1+\beta}\rvert,\lvert1-\theta\tfrac{1+\beta\mu}{(\beta+1)(\mu+1)}\rvert,\lvert1-\theta\rvert,\lvert1-\theta\tfrac{\mu}{1+\mu}\rvert\}\]
%     \item (lower bound for mode (e) is valid under certain restrictions on $\mu,L,\theta>0$) $\rho=\rho_2$ with \[\rho_2=\max\{\rho_1, \tfrac{\sqrt{2-\theta}}{2} \sqrt{\tfrac{((2-\theta)\mu (\beta+1)-{\theta\beta}  (\mu-1))\, ((2-\theta)\beta(\mu+1)-{\theta}\mu (\beta-1))}{(2-\theta)\mu\beta({\mu}+1) (\beta+1)-   {\theta}\mu^2\beta^2}} \}\] and $\rho_1$ as above.
% \end{itemize}
% The proof simply consists in providing examples reaching the desired values of $\rho$ for the different modes.
\paragraph{Case (a)}
% (one-dimensional worst-case)
	$A=N_{\{0\}}$, $B=\tfrac{1}{\beta}I$, and $T=\left(1-\theta\tfrac{\beta}{1+\beta}\right)I$.
	%; hence $J_A=0$ (projection on $\{0\}$) and $J_B=\tfrac{\beta}{1+\beta}$, and 
	This construction provides the lower bound $\rho=\lvert1-\theta \tfrac{\beta}{\beta+1}\rvert$, and it is valid when $(\mu,\beta,\theta)\in R_\mathrm{(a)}$. (In fact, it is always valid.)

\paragraph{Case (b)}
% (one-dimensional worst-case)
	$A=\mu I$, $B=\tfrac{1}{\beta}I$, and $T=\left(1-\theta\tfrac{1+\beta\mu}{(\beta+1)(\mu+1)}\right)I$.
	%; hence $J_A=\tfrac{1}{1+\mu}$ and $J_B=\tfrac{\beta}{1+\beta}$, and 
	This construction provides the lower bound $\rho=\lvert 1-\theta\tfrac{1+\mu\beta}{(\mu+1)(\beta+1)}\rvert$, and it is valid when $(\mu,\beta,\theta)\in R_\mathrm{(b)}$. (In fact, it is always valid.)

\paragraph{Case (c)}
 %(one-dimensional worst-case)
	$A=N_{\{0\}}$, $B=0 $, and $T=(1-\theta)I$.
	%; hence $J_A=0$ (projection on $\{0\}$) and $J_B=$
	This construction provides the lower bound $\rho=\lvert 1-\theta \rvert$, and it is valid when $(\mu,\beta,\theta)\in R_\mathrm{(c)}$. (In fact, it is always valid.)

\paragraph{Case (d)}
 %(one-dimensional worst-case)
	$A=\mu I$, $B=0$, and $T=\left(1-\theta\tfrac{\mu}{\mu+1}\right)I$.
	%; hence $J_A=\tfrac{1}{1+\mu}$ and $J_B=I$
	This construction provides the lower bound $\rho=\lvert 1-\theta\tfrac{\mu}{\mu +1} \rvert$, and it is valid when $(\mu,\beta,\theta)\in R_\mathrm{(d)}$. (In fact, it is always valid.)

\paragraph{Case (e)}	
%(two-dimensional worst-case)
Define	\[K=\tfrac{((2-\theta ) \mu  (\mu +1)+\beta  (\mu -1) (2-\theta+ 2\mu (1-\theta )   ))\, ((2-\theta ) \mu+\beta  (2 (1-\theta ) \mu -\theta +2) )}{\beta ^2 (\theta -2) \mu  (2 \beta  (\theta -2)-\theta -2)+\beta ^2 (2 \beta +1) (\theta -2)^2+(2 \beta -1) \mu ^3 (2 \beta  (\theta -1)+\theta -2)^2-(2 \beta -1) \mu ^2 (2 \beta -\theta +2) (2 \beta  (\theta -1)+\theta -2)}.\]
It is possible to verify that there is no division by $0$ in the definition of $K$ and that $0<K< 1/\beta^2$ when $(\mu,\beta,\theta)\in R_{\mathrm{(e)}}$.
%\ER{So we can do $<1/\beta^2$?}
Let 
\[A=\begin{pmatrix}
\mu  & -a \\
a & \mu  \\
\end{pmatrix},
\qquad
B=\begin{pmatrix} 
\beta K & -\sqrt{K-K^2\beta^2} \\
\sqrt{K-K^2\beta^2} & \beta K
\end{pmatrix},
\]
where
\[a=\tfrac{2 \theta  \mu +\theta-2 \beta  \theta  K \mu -\theta  K+2 K \mu +2 K-2 \mu -2+\sqrt{4 (\theta -2)^2 (\mu +1)^2 \left(K-\beta ^2 K^2\right)+((\theta -2) (K-1)-2 \mu  (\theta -\beta  \theta  K+K-1))^2}}{2 (\theta -2) \sqrt{K-\beta ^2 K^2}}.\]
Since $\tfrac12(A+A^T)$ has two eigenvalues equal to $\mu$
and $\tfrac12 (B^{-T}+B^{-1})$ has two eigenvalues equal to $\beta$, we have $A\in \mathcal{M}_\mu$ and $B\in \mathcal{C}_\beta$.
Then
\[T=\begin{pmatrix}
T_1 & -T_2 \\
T_2 & T_1\\
\end{pmatrix} \]
with 
\begin{equation*}
    \begin{aligned}
    T_1&=\tfrac{-\theta  \left(a^2 (\beta  K+1)+(\mu +1) (K (\beta  \mu +\beta +1)+\mu )\right)+\left(a^2+(\mu +1)^2\right) (2 \beta  K+K+1)-2 a \theta  \sqrt{K-\beta ^2 K^2}}{\left(a^2+(\mu +1)^2\right) (2 \beta  K+K+1)},\\
    T_2&=\tfrac{\theta  \left(a^2 \sqrt{K-\beta ^2 K^2}+a (K-1)+\left(\mu ^2-1\right) \sqrt{K-\beta ^2 K^2}\right)}{\left(a^2+(\mu +1)^2\right) (2 \beta  K+K+1)}
    \end{aligned},
\end{equation*}
and 
\[T^TT=\begin{pmatrix}
T_1^2+T_2^2 & 0 \\ 0 & T_1^2+T_2^2
\end{pmatrix}.\]
So this construction provides the lower bound $\rho^2=T_1^2+T_2^2$.
Under the assumption $\theta <2\tfrac{\beta  \mu +\beta +\mu}{2 \beta  \mu +\beta +\mu }$ (i.e., when $(\mu,\beta,\theta)\notin R_\mathrm{(c)}$), 
the lower bound simplifies to
% \[\rho^2=\tfrac{(2-\theta ) ((\beta  \theta  (1-\mu )+(\beta +1) (2-\theta ) \mu ) (\beta  (2-\theta ) (\mu +1)+(1-\beta ) \theta  \mu ))}{(4 \beta  \mu ) ((\beta +1) (2-\theta ) (\mu +1)-\beta  \theta  \mu )}.\]
\[
\rho=
%\tfrac{\sqrt{2-\theta}}{2} \sqrt{\tfrac{((2-\theta)\mu (\beta+1)-{\theta\beta}  (\mu-1))\, ((2-\theta)\beta(\mu+1)-{\theta}\mu (\beta-1))}{(2-\theta)\mu\beta({\mu}+1) (\beta+1)-   {\theta}\mu^2\beta^2}},
\tfrac{\sqrt{2-\theta}}{2}
\sqrt{\tfrac{((2-\theta)\mu (\beta+1)-{\theta\beta}  (\mu-1))\, ((2-\theta)\beta(\mu+1)-{\theta}\mu (\beta-1))}{\mu\beta(2\mu \beta (1-\theta )  + (2-\theta ) (\mu+\beta +1))}},
\]
and it is valid when $(\mu,\beta,\theta)\in R_\mathrm{(e)}$.

%%%%%%%%%%%%%%%%%%%%%%%%%%%%%%%%%%%%%%%%%%%%%%%%%%%%%%%%%%%%%%%%%%%%%
%%%%%%%%%%%%%%%%%%%%%%%%%%%%%%%%%%%%%%%%%%%%%%%%%%%%%%%%%%%%%%%%%%%%%
%                                                                   %
%                                                                   %
%   PROOF OF THE DRS THEOREM 2 (Lipschitz/monotone + str monotone)  %
%                                                                   %
%                                                                   %
%%%%%%%%%%%%%%%%%%%%%%%%%%%%%%%%%%%%%%%%%%%%%%%%%%%%%%%%%%%%%%%%%%%%%
%%%%%%%%%%%%%%%%%%%%%%%%%%%%%%%%%%%%%%%%%%%%%%%%%%%%%%%%%%%%%%%%%%%%%
\subsection{Proof of Theorem~\ref{thm:DRS_Lipschitz_strmonotone}}
\label{ss:proof_thm_drs_lip-smon}
%We now prove Theorem~\ref{thm:DRS_Lipschitz_strmonotone}. The approach is similar to that of Theorem~\ref{thm:DRS_coco_strmonotone}.
Define
\begin{align*}
R_\mathrm{(a)}&=
\left\{(\mu,L,\theta)\,\Big|\,\mu\tfrac{-\left(2 (\theta -1) \mu +\theta-2\right)+L^2\left(\theta -2(1+ \mu)\right)}{\sqrt{ (2 (\theta -1) \mu +\theta -2)^2+L^2 (\theta -2 (\mu +1))^2}}\leq \sqrt{L^2+1},\,\mu>0,\,L>0,\,\theta\in(0,2)
 \right\}\\
 R_\mathrm{(b)}&=
\left\{(\mu,L,\theta)\,\Big|\,
L<1,\, \mu >\tfrac{L^2+1}{(L-1)^2},\,
\theta \leq \tfrac{2 (\mu +1) (L+1)(\mu +\mu  L^2-L^2-2 \mu  L-1)}{2 \mu ^2-\mu +\mu  L^3-L^3-3 \mu  L^2-L^2-2 \mu ^2 L-\mu  L-L-1},\,\mu>0,\,L>0,\,\theta\in(0,2)
\right\}\\
 R_\mathrm{(c)}&=
\left\{(\mu,L,\theta)\,\Big|\,
\mu>0,\,L>0,\,\theta\in(0,2)
 \right\}\backslash R_\mathrm{(a)}\backslash R_\mathrm{(b)}
\end{align*}
which correspond to the 3 cases of Theorem~\ref{thm:DRS_Lipschitz_strmonotone}.

\subsubsection{Upper bounds}\label{sec:upper_bound_Lips}
%\AT{I used $\lambda^B_\mu$ as notation for monotonicity (not strongly!) of $B$; OK?}
The approach is similar to that of Section~\ref{ss:first_proof_upper}.
By weak duality between the primal and dual OSPEP, $\rho$ is a valid contraction factor if there exists $\rho$, $\lambda^A_\mu\ge 0$, $\lambda^B_L\ge 0$, and $\lambda^B_\mu\ge 0$ such that
\[ S= \begin{pmatrix}\rho ^2+\lambda^B_L-1 & \frac{\lambda^A_\mu}{2}-\theta  & \theta -\lambda^B_L-\frac{\lambda^B_\mu}{2} \\
 \frac{\lambda^A_\mu}{2}-\theta  & -\theta ^2+\lambda^A_\mu+\lambda^A_\mu \mu  & \theta ^2-\lambda^A_\mu \\
 \theta -\lambda^B_L-\frac{\lambda^B_\mu}{2} & \theta ^2-\lambda^A_\mu & -\lambda^B_L L^2-\theta ^2+\lambda^B_L+\lambda^B_\mu \\
\end{pmatrix}\succeq 0.
\]
We establish $S\succeq 0$ with a sum-of-squares factorization 
\begin{equation}
\label{eq:proof_thm_A_strmonotone_B_Lipschitz}
\mathrm{Tr}(SG(z,z_A,z_B))=K_1\normsq{m_1 z_A + m_2 z_B +m_3 z}+K_2 \normsq{m_4 z_B+m_5 z},
\end{equation}
for some $m_1,m_2,m_3,m_4,m_5\in \reals$ and $K_1,K_2\ge0$,
where $G(z,z_A,z_B)$ is as defined in \eqref{eq:grammian_small}.

\begin{rem}
As before, the dual matrix $S$ satisfies
\begin{equation*}
	\begin{aligned}
\mathrm{Tr}(SG(z,z_A,z_B))=\rho^2 \normsq{z}&-\normsq{z_+}-\lambda^A_\mu \left( \inner{\Delta A}{z_A}-\mu \normsq{z_A}\right) - \lambda^B_\mu \inner{\Delta B}{z_B} -\lambda^B_L \left(L^2\normsq{z_B}- \normsq{\Delta B}\right),
	\end{aligned}
\end{equation*}
with $\Delta A=2z_b-z-z_A$ and $\Delta B=z-z_B$.
One can use this to construct a classical proof establishing $\rho^2$ as a valid contraction factor without relying on the OSPEP methodology,
given $\rho$, $\lambda^A_\mu\ge 0$, $\lambda^B_L\ge 0$, and $\lambda^B_\mu\ge 0$ such that $S\succeq 0$.
\end{rem}

% As before, the dual matrix $S$ is chosen such that
% \begin{equation*}
% 	\begin{aligned} \mathrm{Tr}({S}{G})=\rho^2 \normsq{z}&-\normsq{z_+}-\lambda^A_\mu \bigg( \inner{\Delta A}{z_A}-\mu \normsq{z_A}\bigg) - \lambda^B_\mu \inner{\Delta B}{z_B} -\lambda^B_L (\normsq{z_B}-\tfrac{1}{L^2} \normsq{\Delta B}),
% 	\end{aligned}
% \end{equation*}
% with $G$ the Gram matrix defined as in~\eqref{eq:grammian} and using the same notations $\Delta A:=2z_b-z-z_A$ and $\Delta B:=z-z_B$. The corresponding dual matrix is as follows

\paragraph{Case (a)}
 When $(\mu,L,\theta)\in R_\mathrm{(a)}$, we define
		\[C=\sqrt{\tfrac{ (2 (\theta -1) \mu +\theta -2)^2+L^2 (\theta -2 (\mu+1))^2}{L^2+1}},\]
		which satisfies $C>0$ for all values of $L,\mu,\theta>0$, and we use
	\begin{equation*}
	\begin{aligned}
	\rho^2  =\left(\tfrac{\theta+C}{2(\mu+1)}\right)^2,\quad 	\lambda^A_\mu=\tfrac{\theta(\theta +C)}{(\mu +1)},\quad \lambda^B_L=\tfrac{(2-\theta) \theta  \mu}{(\mu +1)\left(L^2+1\right)} \tfrac{  \theta+C }{C},\\
	\lambda^B_\mu=\tfrac{ \theta\left(\theta+C\right) }{(\mu +1)^2 C}\left(C+ \mu  \tfrac{\left(2 (\theta -1) \mu +\theta-2\right) -L^2\left(\theta -2 (\mu +1)\right)}{L^2+1}\right).
	\end{aligned}
	\end{equation*}
	This gives us the sum-of-square factorization \eqref{eq:proof_thm_A_strmonotone_B_Lipschitz} with
	\begin{gather*}
	K_1 = \theta  C,\quad m_1=-1,\quad 	m_2 = \tfrac{C-\theta  \mu }{C (1+\mu)},\quad	m_3 = \tfrac{ 2 (\mu +1)-(C+\theta)}{2 C (\mu +1)},\\
	K_2 = 0  ,\quad m_4=0,\quad 	m_5 = 0.
	\end{gather*}
	When $\mu,L>0$ and $\theta\in (0,2)$, we have $\lambda^A_\mu,\lambda^B_L,K_1,K_2\geq0$.
	Furthermore, $\lambda^B_\mu\geq 0$ when
	\[ C\ge \mu  \tfrac{-\left(2 (\theta -1) \mu +\theta-2\right) +L^2\left(\theta -2 (\mu +1)\right)}{L^2+1}, \]
	which is immediately equivalent to the main condition defining $R_\mathrm{(a)}$.

\paragraph{Case (b)}
 When $(\mu,L,\theta)\in R_\mathrm{(b)}$, we use
	\begin{gather*}
	\rho^2=\left( 1-\theta\tfrac{L+\mu}{(\mu+1)(L+1)}\right)^2\\
	\lambda^A_\mu=2 \theta\tfrac{1+L}{1-L} \left(1-\theta\tfrac{\mu +L}{(\mu +1) (L+1)}\right),\quad \lambda^B_L=\tfrac{\theta}{L} \tfrac{\mu -1}{\mu +1}  \left(1-\theta\tfrac{\mu +L}{(\mu +1) (L+1)}\right),\quad 	\lambda^B_\mu=0.
	\end{gather*}
	This gives us the sum-of-square factorization \eqref{eq:proof_thm_A_strmonotone_B_Lipschitz} with
	\begin{gather*}
	K_1 = \theta\tfrac{2 (\mu +1) (L+1)-\theta  (2 \mu +L+1)}{1-L},\\
	K_2 = \theta\tfrac{(L+1)(\mu+1)-\theta(L+\mu)}{(\mu +1)^2 (1-L) L (L+1)^2}\,\tfrac{2 (\mu +1) (L+1) \left(\mu  (1-L)^2-\left(L^2+1\right)\right)+\theta  \left(\mu  \left(1+L+3 L^2-L^3\right)+\left(1+L+L^2+L^3\right)+2 \mu ^2 (L-1)\right)}{2 (\mu +1) (L+1)-\theta  (2 \mu +L+1)}\\
	m_1 = -1,\qquad
	m_2 = \tfrac{2}{\mu +1}-\tfrac{(2-\theta ) (L+1)}{2 (\mu +1) (L+1)-\theta  (2 \mu +L+1)},\\
	m_3 = \tfrac{1}{1+\mu}\, \tfrac{\theta  (\mu+L) -2 L(\mu +1)}{2 (\mu +1) (L+1)-\theta  (2 \mu +L+1)},\qquad
	m_4 = -(1+L),\qquad
	m_5 = 1.
	\end{gather*}
%	\ER{I re-wrote the denominators of $m_2$ and $m_3$ so that they look like the numerator of $K_1$. Reason: this make showing positivity of the denominators of $m_2$ and $m_3$ easier.}
	It is possible to verify that there is no division by $0$ in the definitions and 
	that $\lambda^A_\mu,\lambda^B_L,\lambda^B_\mu,K_1,K_2\geq 0$ when $(\mu,L,\theta)\in R_\mathrm{(b)}$.

\paragraph{Case (c)}
 When $(\mu,L,\theta)\in R_\mathrm{(c)}$, we use
	\begin{gather*}
	\rho^2=\tfrac{(2-\theta)}{4 \mu  \left(L^2+1\right)}\, \tfrac{\left(	\theta(L^2+1)-2\mu(\theta+L^2-1)	\right)
	\left(\theta  \left(1+2 \mu +L^2\right)-2 (\mu +1) \left(L^2+1\right)\right)}{2 \mu  \left(\theta +L^2-1\right)-(2-\theta ) \left(1-L^2\right)}
	\\
	\lambda^A_\mu=\theta  \left(\theta -\tfrac{2 \mu  \left(\theta +L^2-1\right)}{L^2+1}\right),\quad \lambda^B_L=\tfrac{(2-\theta) \theta   }{(L^2+1)}\, \tfrac{\theta  \left(L^2+1\right)-2 \mu  \left(\theta +L^2-1\right)}{(2-\theta) \left(1-L^2\right)-2 \mu \left(\theta +L^2-1\right)},\quad 
	\lambda^B_\mu=0.
	\end{gather*}
	%\ER{I re-factored the denominator of $\rho^2$ from%
%	\[
%		\rho^2=\tfrac{(2-\theta)}{4 \mu  \left(L^2+1\right)}\, \tfrac{ \left(\theta  \left(1-2 \mu +L^2\right)-2 \mu  \left(L^2-1\right)\right) \left(\theta  \left(1+2 \mu +L^2\right)-2 (\mu +1) \left(L^2+1\right)\right)}{\theta  \left(1+2 \mu -L^2\right)-2 (\mu +1) \left(1-L^2\right)}\\
%	\]
%	}
	This gives us the sum-of-square factorization \eqref{eq:proof_thm_A_strmonotone_B_Lipschitz} with
	\begin{gather*}
	K_1 = \tfrac{\theta }{4 \mu  \left(L^2+1\right) \left((2-\theta)(1-L^2)-2\mu (\theta+L^2-1)\right)}, \quad 	m_1 = 4 \mu ^2 \left(1-L^2-\theta\right)+2\mu (2-\theta )   \left(1-L^2\right),\\
	m_2 = 4 \mu  \left(L^2+\theta-1\right),\quad 
	m_3 = 2 \mu  \left(1-L^2-\theta\right)-(2-\theta ) \left(L^2+1\right),\\
	K_2=0,\quad m_4=0,\quad m_5=0.
	\end{gather*}
	It is possible to verify that there is no division by $0$ in the definitions and 
	that $\rho^2,\lambda^A_\mu,\lambda^B_L,\lambda^B_\mu,K_1,K_2\geq 0$ when $(\mu,L,\theta)\in R_\mathrm{(c)}$.

\subsubsection{Lower bounds}
\label{sss:lower2}
We now show that for the three cases, there are operators $A\in\M_\mu$ and $B\in\Li_L\cap\M$ and inputs $z_1,z_2\in\hilbert$ such that
\[ \norm{Tz_1-Tz_2}\geq \rho \norm{z_1-z_2},\]
where  $T=I-\theta J_B + \theta J_A (2J_B-I)$ % the overrelaxed Douglas-Rachford operator.
and $\rho$ is given by Theorem~\ref{thm:DRS_Lipschitz_strmonotone}.
We construct the lower bounds for $\reals^2$, since the construction can be embedded into the higher dimensional space $\hilbert$.

% In this section, we show that for all $\mu,L\geq 0$, there exists operators $A\in\M_\mu$ and $B\in\Li_L\cap\M$, and a pair of points $x_1,x_2\in\hilbert$ such that
% 	\begin{equation}
% 	    \norm{Tx_1-Tx_2} \geq \rho \norm{x_1-x_2},\label{eq:LB_sm_coco}
% 	\end{equation}
% 	with $T=\mathrm{I}_d - \theta J_{B} + \theta J_{A}(2  J_{B} - \mathrm{I}_d)$, the overrelaxed Douglas-Rachford operator.
	%, with either $\rho=\rho_1$ or $\rho=\rho_2$ as below:
% \begin{itemize}
%     \item (lower bounds (a) and (b) are valid for all values $\mu,L,\theta>0$) $\rho=\rho_1$ with \[\rho_1=\max\left\{\tfrac{\theta +\sqrt{\tfrac{(2 (\theta -1) \mu +\theta -2)^2+L^2 (\theta -2 (\mu +1))^2}{L^2+1}}}{2 (\mu +1)},\lvert1-\theta\tfrac{L+\mu}{(L+1)(\mu+1)}\rvert\right\}\]
%     \item (lower bound for mode (c) is valid under certain restrictions on $\mu,L,\theta>0$) $\rho=\rho_2$ with \[\rho_2=\max\left\{\rho_1, \sqrt{\tfrac{(2-\theta)}{4 \mu  \left(L^2+1\right)}\, \tfrac{ \left(\theta  \left(1-2 \mu +L^2\right)-2 \mu  \left(L^2-1\right)\right) \left(\theta  \left(1+2 \mu +L^2\right)-2 (\mu +1) \left(L^2+1\right)\right)}{\theta  \left(1+2 \mu -L^2\right)-2 (\mu +1) \left(1-L^2\right)}} \right\}\] and $\rho_1$ as above.
% \end{itemize}
% The proof simply consists in providing examples reaching the desired values of $\rho$ for the different modes.
		\paragraph{Case (a)} %(two-dimensional worst-case; largely inspired by~\cite[Example 5.3]{moursi2018douglas})
		Let 
		\[A=\mu\, \mathrm{Id}+N_{\{0\}\times \mathbb{R}},\quad B=L\begin{bmatrix}
		0 & 1 \\ -1 & 0
		\end{bmatrix}.\] Then
		\[ J_A=\tfrac{1}{\mu +1}\begin{bmatrix}
		0 & 0 \\
		0 & 1 \\
		\end{bmatrix},\quad J_B=\tfrac{1}{L^2+1}
		\begin{bmatrix}
		1 & -L \\
		L & 1 \\
		\end{bmatrix},
		 \]
		and
		\[T=\begin{bmatrix}
		1-\tfrac{\theta }{L^2+1} & \tfrac{L \theta }{L^2+1} \\
		-\tfrac{L \theta  (\mu -1)}{\left(L^2+1\right) (\mu +1)} & \tfrac{(-\theta +\mu +1) L^2-\theta  \mu +\mu +1}{\left(L^2+1\right) (\mu +1)} 
		\end{bmatrix}. \]
		The eigenvalues of $T^TT$ are given by
		\[\tfrac{\left(\theta \pm\sqrt{\tfrac{(2 (\theta -1) \mu +\theta -2)^2+L^2 (\theta -2 (\mu +1))^2}{L^2+1}}\right)^2}{4 (\mu +1)^2}.\]
		This construction provides the lower bound
		\[
		\rho=\tfrac{\theta +\sqrt{\tfrac{(2 (\theta -1) \mu +\theta -2)^2+L^2 (\theta -2 (\mu +1))^2}{L^2+1}}}{2 (\mu +1)},
		\]
		and it is valid when $(\mu,L,\theta)\in R_\mathrm{(a)}$. (In fact, it is always valid.)
		This construction was inspired by Example 5.3 of \cite{moursi2018douglas}.
% 		and hence the convergence rate is given by
% 		\[\norm{TT^T}^{1/2}=\tfrac{\theta +\sqrt{\tfrac{(2 (\theta -1) \mu +\theta -2)^2+L^2 (\theta -2 (\mu +1))^2}{L^2+1}}}{2 (\mu +1)},\]
% 		which conclude the proof.

\paragraph{Case (b)}
		%(one-dimensional worst-case)
		$A=\mu I$, $B=LI$, and $T=\left(1-\theta\tfrac{L+\mu}{(L+1)(\mu+1)}\right)$.
		%hence $J_A=\tfrac{1}{1+\mu}$ and $J_B=\tfrac{1}{1+L}$,
		This construction provides the lower bound $\rho=\lvert 1-\theta\tfrac{L+\mu}{(\mu+1)(L+1)} \rvert$, and it is valid when $(\mu,L,\theta)\in R_\mathrm{(b)}$. (In fact, it is always valid.)
%		\item[(c)] probably more involved --- two-dimensional --- the operator $B$ should be monotone, but without the corresponding contraint being tight; so try something like \[A=\mu\, \mathrm{Id}+N_{\{0\}\times \mathbb{R}},\quad B=L\begin{bmatrix}
%		0 & 1 \\ 1 & 0
%		\end{bmatrix},\]or a variant like
%		\[B=L\begin{bmatrix}
%			1 & a \\ b & 1 
%		\end{bmatrix}\]
%		could do the job

        \paragraph{Case (c)}
        %(two-dimensional worst-case)
        Define
        \[K=\tfrac{L^2+1}{2 L}\,\tfrac{(\mu -1) \left(L^2 (\theta -2 (\mu +1))-(2 (\theta -1) \mu +\theta -2)\right)^2-4 (\theta -2)^2 (\mu +1) L^2}{ 4 \mu ^2 \left(\theta +L^2-1\right) \left((1-\theta ) (L^2-\mu) +L^2\mu-1\right)+(\theta -2)^2 (\mu +1) \left(L^2+1\right)^2}.\]
It is possible to verify that there is no division by $0$ in the definition of $K$ and that $0\le K\le 1$ when $(\mu,L,\theta)\in R_\mathrm{(c)}$.
 Let
 \[A=\mu\, \mathrm{Id}+N_{\mathbb{R}\times \{0\}},\quad 
B=L\begin{pmatrix} 
 K  & -\sqrt{1-K^2}  \\
 \sqrt{1-K^2}  & K  \\
\end{pmatrix}.
\]
Since the eigenvalues of $B^{T}B$ and $\tfrac12(B^T+B)$ are respectively both equal to $L^2$ and $KL$, we have $B\in \mathcal{M}\cap \mathcal{L}_L$.
% The corresponding resolvents are
Then
\[ J_A=\tfrac{1}{\mu +1}\begin{bmatrix}
		1 & 0 \\
		0 & 0 \\
		\end{bmatrix},\quad J_B=\begin{pmatrix}
		 \tfrac{K L+1}{L^2+2 K L+1} & \tfrac{\sqrt{1-K^2} L}{L^2+2 K L+1} \\
 -\tfrac{\sqrt{1-K^2} L}{L^2+2 K L+1} & \tfrac{K L+1}{L^2+2 K L+1} 
		\end{pmatrix},
\]
and 
  \[T=\begin{pmatrix}1-\tfrac{K L \theta +\theta }{L^2+2 K L+1} & -\tfrac{\sqrt{1-K^2} L \theta }{L^2+2 K L+1} \\
 \tfrac{\sqrt{1-K^2} L \theta  (\mu -1)}{\left(L^2+2 K L+1\right) (\mu +1)} & \tfrac{(-\theta +\mu +1) L^2-K (\theta -2) (\mu +1) L-\theta  \mu +\mu +1}{\left(L^2+2 K L+1\right) (\mu +1)} \end{pmatrix}.\]
% \[TT^T=\begin{pmatrix}
%  \tfrac{L^2-2 K (\theta -1) L+(\theta -1)^2}{L^2+2 K L+1} & \tfrac{\sqrt{1-K^2} L (\theta -2) \theta }{\left(L^2+2 K L+1\right) (\mu +1)} \\
%  \tfrac{\sqrt{1-K^2} L (\theta -2) \theta }{\left(L^2+2 K L+1\right) (\mu +1)} & \tfrac{L^2 (-\theta +\mu +1)^2+(-\theta  \mu +\mu +1)^2+2 K L \left(\mu  \theta ^2-(\mu +1)^2 \theta +(\mu +1)^2\right)}{\left(L^2+2 K L+1\right) (\mu +1)^2} 
% \end{pmatrix},\]
The eigenvalues of $T^TT$ are
\begin{equation*}
    \begin{aligned}
    \tfrac{T_1 \pm \theta  \sqrt{T_2}}{2 (\mu +1)^2 \left(2 K L+L^2+1\right)}
    \end{aligned}
\end{equation*}
with
\begin{equation*}
    \begin{aligned}
    T_1&=\theta ^2 \left(2 K \mu  L+2 \mu  (\mu +1)+L^2+1\right)-2 \theta  (\mu +1) \left(2 K (\mu +1) L+2 \mu +L^2+1\right)+2 (\mu +1)^2 \left(2 K L+L^2+1\right),\\
    T_2&=\left(-2 K L+L^2+1\right) \left((2 (\theta -1) \mu +\theta -2)^2+2 K L (\theta -2 (\mu +1)) (2 (\theta -1) \mu +\theta -2)+L^2 (\theta -2 \mu -2)^2\right).
    \end{aligned}
\end{equation*}
This construction provides the lower bound
\[
\rho^2=    \tfrac{T_1 + \theta  \sqrt{T_2}}{2 (\mu +1)^2 \left(2 K L+L^2+1\right)}.
\]
Under the assumption 
\begin{equation*}
\begin{aligned}
&\tfrac{\mu  \left(-2 (\theta -1) \mu -\theta +L^2 (\theta -2 (\mu +1))+2\right)}{\sqrt{(2 (\theta -1) \mu +\theta -2)^2+L^2 (\theta -2 (\mu +1))^2}}>\sqrt{L^2+1},
%&\theta  \left(L^2+1\right)-2 \mu  \left(\theta +L^2-1\right)\geq0,\\
%&(2-\theta ) \left(1-L^2\right)-2 \mu  \left(\theta +L^2-1\right)>0,\\
\end{aligned}
\end{equation*}
(i.e., when $(\mu,L,\theta)\notin R_\mathrm{(a)}$), the lower bound simplifies to
\[
\rho=\sqrt{
\tfrac{(2-\theta)}{4 \mu  \left(L^2+1\right)}\, \tfrac{
	%\left(\theta  \left(1-2 \mu +L^2\right)-2 \mu  \left(L^2-1\right)\right)
	\left(
	\theta(L^2+1)-2\mu(\theta+L^2-1)
	\right)
	\left(\theta  \left(1+2 \mu +L^2\right)-2 (\mu +1) \left(L^2+1\right)\right)}{2 \mu  \left(\theta +L^2-1\right)-(2-\theta ) \left(1-L^2\right)}},
\]
and it is valid when $(\mu,L,\theta)\in R_\mathrm{(c)}$.

\section{Algebraic verification of inequalities}
\label{ss:verification}
%%%%%%%%%%%%%%%%%%%%%%%%%%%%%%%%%%%%%%%%%%%%%%%%%%%%%%%%%%%%%%%%%%%%%
%%%%%%%%%%%%%%%%%%%%%%%%%%%%%%%%%%%%%%%%%%%%%%%%%%%%%%%%%%%%%%%%%%%%%
%                                                                   %
%                                                                   %
%   PROOF OF THE DRS THEOREM 1 (cocoercive + str monotone)          %
%                                                                   %
%                                                                   %
%%%%%%%%%%%%%%%%%%%%%%%%%%%%%%%%%%%%%%%%%%%%%%%%%%%%%%%%%%%%%%%%%%%%%
%%%%%%%%%%%%%%%%%%%%%%%%%%%%%%%%%%%%%%%%%%%%%%%%%%%%%%%%%%%%%%%%%%%%%
We now provide the algebraic verifications of the inequalities asserted in Section~\ref{s:main-proof}.
The proofs of this section are based on elementary arguments and arduous algebra.
To help readers follow and verify the basic but tedious computation, we provide code that uses symbolic manipulation to verify the steps.

\subsection{Inequalities for Theorem~\ref{thm:DRS_coco_strmonotone}}
\label{ss:thm3_ineq}
\subsubsection{Upper bounds}
\label{ss:theorem3_verification}
It remains to show that there is no division by $0$ and $\lambda^A_\mu,\lambda_\beta^B,K_1,K_2\geq 0$ in each case.

\paragraph{Case (a)}
		Assume $0<\mu$, $0<\beta$, $0<\theta<2$, and
 		\begin{align}
 		&\mu-\mu\beta-\beta > 0 \tag{a1}\label{eq:casea_conda}\\
 		&\theta \leq 2\tfrac{(\beta+1) (\mu -\mu\beta-\beta)}{\mu+ \mu \beta - \beta-\beta^2 -2 \mu \beta^2   } 
 		\quad \text{(no division by $0$ implied, see (ii) below)}
 		\tag{a2}\label{eq:casea_condb}
 		\end{align}
 		Then:
		\begin{itemize}
			\item[(i)] From~\eqref{eq:casea_conda} it is direct to note that $1-\beta>\beta{\mu}^{-1}>0$ and hence also $\beta<1$.
			\item[(ii)] As the numerator of~\eqref{eq:casea_condb} is positive (from $\beta>0$ and~\eqref{eq:casea_conda}) and as $\theta>0$, the denominator of~\eqref{eq:casea_condb} is positive, i.e., $\mu  + \mu \beta - \beta-\beta^2 -2 \mu \beta^2 \geq 0$.  To prove strict positivity, assume for contradiction that 
			$\mu  + \mu \beta - \beta-\beta^2 -2 \mu \beta^2=0.$ This implies 
			\[\beta  (\beta +2 (\beta-1) \mu )=\mu-\mu \beta-\beta\overset{\text{\eqref{eq:casea_conda}}}{>}0,
			\]
			and hence \[0< \beta +2 (\beta-1) \mu\overset{\text{\eqref{eq:casea_conda}}}{<}\mu-\mu\beta+2(\beta-1)\mu=\mu(\beta-1)\overset{\text{(i)}}{<}0,\]
			which is a contradiction. Therefore, we conclude 
			\[\mu  + \mu \beta - \beta-\beta^2 -2 \mu \beta^2 > 0.\]
			\item[(iii)] Multiply both sides of \eqref{eq:casea_condb} by $\beta$, reorganize, and use \eqref{eq:casea_conda} to get 
			\[
			1+\beta-\theta \beta \overset{\text{\eqref{eq:casea_condb}}}{\geq} \tfrac{\left(1-\beta^2\right) (\mu -\beta )}{\mu  + \mu \beta - \beta-\beta^2-2 \mu\beta^2} \overset{\text{\eqref{eq:casea_conda}}}{>} \tfrac{\mu\beta \left(1-\beta^2\right) }{\mu  + \mu \beta - \beta-\beta^2-2 \mu\beta^2} > 0.
			\]
			\item[(iv)] Multiply both sides of \eqref{eq:casea_condb} by the denominator of \eqref{eq:casea_condb} (which is positive by (ii))
			and reorganize to get 
			\[
			(\beta -1) \mu  (2 \beta  (\theta -1)+\theta -2)-(2-\theta )\beta  (\beta+1)\ge 0.
			\]
			\item[(v)]  Multiply both sides of \eqref{eq:casea_condb} by $\beta/(1+\beta)$ and reorganize to get
			\[
			1-\theta\tfrac{\beta}{\beta+1}\ge \tfrac{(1-\beta) (\mu-\beta)}{ \mu +\mu\beta   -\beta- \beta ^2-2 \mu \beta ^2 }>0.
			\]
		\end{itemize} 
		
		(v) shows $\lambda^A_\mu$ and $\lambda_\beta^B$ are nonnegative.
		(i) shows the denominator of $K_1$ is positive.
		(iii) shows the numerator of $K_1$ is nonnegative.
		(i) and (iii) show the denominator of $K_2$ is positive.
		(iii) and (iv) show the numerator of $K_2$ is nonnegative.
		(iii) shows the denominator of $m_2$ and $m_3$ are positive.

\paragraph{Case (b)}
		Assume $0<\mu$, $0<\beta$, $0<\theta<2$, and
		\begin{align}
		&\mu\beta-\mu -\beta >0 \tag{b1}\label{eq:caseb_conda}\\
		& \theta \leq \tfrac{2 (\mu +1) \left(\mu+\beta+\beta^2 -\mu \beta^2\right)}{\mu ^2 +\beta^2+\mu^2\beta+\mu\beta^2+\mu  +\beta-2\mu^2\beta^2}
 		\quad \text{(no division by $0$ implied, see (ii) below))}
\tag{b2}\label{eq:caseb_condb}
		\end{align}
        Then:
    	\begin{itemize}
    	\item[(i)] From \eqref{eq:caseb_conda}, we have $\mu\beta>\mu+\beta$. Therefore $\mu>1+\mu/\beta>1$ and $\beta>1+\beta/\mu>1$.
% 			\item From~\eqref{eq:caseb_conda}, it is direct to note $1<\beta-{\mu}^{-1}\beta<\beta$.
% 			\item Similarly,~\eqref{eq:caseb_conda} has the following consequence: $\mu\overset{\text{\eqref{eq:caseb_conda}}}{>} \tfrac{\beta}{\beta-1} > 1$.
        \item[(ii)] The numerator of \eqref{eq:caseb_condb} is negative since $ \mu+\beta+\beta^2-\mu \beta^2=(1+\beta)(\mu -\mu\beta+\beta) \overset{\text{\eqref{eq:caseb_conda}}}{<} 0$.
        Since $\theta>0$, the denominator of \eqref{eq:caseb_condb} is nonpositive. To prove strict negativity, we view the denominator of~\eqref{eq:caseb_condb} as a quadratic function of $\mu$:
        \begin{align*}
        \phi_\beta(\mu)
        &=\mu ^2 +\beta^2+\mu^2\beta+\mu\beta^2+\mu  +\beta-2\mu^2\beta^2\\
        &=
        \underbrace{(1+\beta-2\beta^2)}_{<0\text{ by (i)}}\mu^2+(1+\beta^2)\mu+\beta+\beta^2.
        \end{align*}
        %whose leading coefficient is $1+\beta-2\beta^2$, which is nonpositive by (i). Hence $\phi_\beta(\mu)$ is a nonpositive definite quadratic in $\mu$.
        This quadratic is nonnegative only between its roots and
        \[
        \phi_\beta(0)=\beta+\beta^2>0,\quad\phi_\beta(\beta/(\beta-1))=-\tfrac{\beta^2}{\beta-1}<0.
        \]
        Therefore
        $\phi_\beta(\mu)<0$ for any $\mu>\beta/(\beta-1)$, which holds by~\eqref{eq:caseb_conda}.
        Therefore we conclude denominator of \eqref{eq:caseb_condb} is strictly negative.
        \item[(iii)] Multiply both sides of \eqref{eq:caseb_condb} by $(1+\beta+2\mu\beta)$ and reorganize to get
        	\[
        	(2-\theta) (\beta+1)+2 \mu  (1+\beta-\theta\beta)\ge -\tfrac{2 \mu ^2 (\mu +1) (\beta^2-1)}{\mu^2+\beta^2+\mu ^2\beta+\mu\beta^2 +\mu  +\beta-2 \mu ^2\beta^2}>0.
        	\]
        	The latter inequality follows from (i) and (ii).
			\item[(iv)]
			Multiply both sides of \eqref{eq:caseb_condb} by $(1+\mu\beta)$ and reorganize to get
			\[
			(\beta+1)(\mu+1)-\theta(1+\mu\beta) \ge -\tfrac{(\mu ^2-1) (\beta^2-1) (\mu  +\beta)}{\mu ^2+\beta^2+\mu^2\beta+\beta^2\mu +\mu +\beta-2\mu^2\beta}>0.
			\]
        	The latter inequality follows from (i) and (ii).
			\item[(v)]
			Multiply both sides of \eqref{eq:caseb_condb} by
			$-(\mu ^2 +\beta^2+\mu^2\beta+\mu\beta^2+\mu  +\beta-2\mu^2\beta^2)$
			(which is positive by (ii))
			and reorganize to get
			\[
			\mu \beta^2 (-2 \theta  \mu +\theta +2 \mu )+\theta\beta^2 +(\theta -2) \mu  (\mu +1) +\beta \left(\theta  \mu ^2+\theta -2 \mu -2\right)-2\beta^2 \ge 0.
			\]
		\end{itemize}
% 		By combining those inequalities, it is now direct to note that $\lambda_A,\lambda_B,K_1,K_2\geq 0$, and that $m_1,\hdots,m_5$ are well defined in the region of interest.
		
		        (i) and (iv) show $\lambda^A_\mu$ and $\lambda_\beta^B$ are nonnegative. 
				(i) shows the denominator of $K_1$ is positive.
				(iii) shows the numerator of $K_1$ is nonnegative.
				(i) and (iii) show the denominator of $K_2$ is positive.
				(iv) and (v) show the numerator of $K_2$ is nonnegative.
				(iii) shows the denominators of $m_2$ and $m_3$ are positive.

\paragraph{Case (c)}
		Assume $0<\mu$, $0<\beta$, $0<\theta<2$, and
		\begin{align}
		&\theta \geq 2 \tfrac{\mu  +\beta+\mu\beta}{\mu +\beta+ 2\mu\beta} \tag{c1}\label{eq:casec_conda}
		\end{align}
Then:
        \begin{itemize}
			\item[(i)] 
			Multiply both sides of \eqref{eq:casec_conda} by $(1+2\mu)$ and reorganize to get
			\[
			2 (\theta -1) \mu +\theta -2 \ge \tfrac{2 \mu ^2}{\mu +\beta+2\mu\beta} >0.
			\]
			\item[(ii)] Reorganize \eqref{eq:casec_conda} to get
			\[
			\theta \ge 2 \tfrac{\mu +\beta+\mu\beta}{\mu+\beta+2\mu\beta}=2 \tfrac{1}{1+\tfrac{\mu\beta}{\mu +\beta+\mu\beta}}>1.
			\]
			\item[(iii)] Multiply both sides of \eqref{eq:casec_conda} by $(\mu+\beta+2\mu\beta)$ to get
			$\theta\beta +\theta  \mu  (1+2\beta)-2 (\mu+\beta+\mu\beta) \ge 0$.
		\end{itemize}

		(ii) shows $\lambda^A_\mu$ and $\lambda_\beta^B$ are nonnegative.
		(i) shows $K_1$ is nonnegative.
		(i) shows the denominator of $K_2$ is positive.
		(ii) and (iii) show the numerator of $K_2$ is nonnegative.
		(i) shows the denominators of $m_2$ and $m_3$ are positive.

\paragraph{Case (d)}
		Assume $0<\mu$, $0<\beta$, $0<\theta<2$, and
		\begin{align}
		&\mu +\mu\beta-\beta<0 \tag{d1}\label{eq:cased_conda}\\
		& \theta \leq \tfrac{2 (\mu +1) (\beta-\mu-\mu\beta)}{\beta+\beta  \mu -\mu -\mu ^2-2 \beta  \mu ^2}
 		\quad \text{(no division by $0$ implied, see (ii) below))}
    \tag{d2}\label{eq:cased_condb}
		\end{align}
Then:
\begin{itemize}
			%\item[(i)] From~\eqref{eq:cased_conda} it is direct to note that $1-\mu>\mu{\beta}^{-1}>0$ and hence also $\mu<1$.
			\item[(i)] From~\eqref{eq:cased_conda} it is direct to note that $\mu < \tfrac{\beta}{\beta+1}<1$
			and $ \beta > \tfrac{\mu}{1-\mu}>\mu$.
			\item[(ii)] As the numerator of \eqref{eq:cased_condb} is positive (from $\mu>0$ and \eqref{eq:cased_conda}) and as $\theta>0$,
			the denominator of \eqref{eq:cased_condb} is nonnegative.
			To prove strict positivity of the of the denominator, 
			assume for contradiction that $\beta+\beta  \mu -\mu -\mu ^2-2 \beta  \mu ^2=0$.
% 			, i.e., $\beta+\beta  \mu -\mu -\mu ^2-2 \beta  \mu ^2 > 0$. For proving the strict inequality, let us assume $\beta+\beta  \mu -\mu -\mu ^2-2 \beta  \mu ^2=0.$
            This implies that 
            \[
            -2 \beta  \mu ^2+2 \beta  \mu -\mu ^2=\beta  \mu -\beta +\mu\overset{\eqref{eq:cased_conda}}{<}0,
            \]
            and hence \[0>2 \beta (1-\mu)  -\mu  \overset{\text{(i)}}{>}\mu>0,\]
			which is a contradiction. Therefore we conclude
			\[\beta+\beta  \mu -\mu -\mu ^2-2 \beta  \mu ^2 > 0.\]
% 			As a consequence of the numerator of~\eqref{eq:cased_condb} being positive --- which is direct by~\eqref{eq:cased_conda}---, its denominator should be nonnegative: 
			\item[(iii)]
			Multiply both sides of \eqref{eq:cased_condb} by $(1+2\mu)$ and reorganize to get
			\[
			2 (\theta -1) \mu +\theta -2 \le -\tfrac{2\mu^2(\mu +1)}{\beta+\beta  \mu -\mu -\mu ^2-2 \beta  \mu ^2}<0.
			\]
			\item[(iv)]
			Multiply both sides of \eqref{eq:cased_condb} by $\mu$ and reorganize to get
			\[
			(\theta -1) \mu -1 \le -\tfrac{(1-\mu^2)(\beta-\mu)}{\beta+\beta  \mu -\mu -\mu ^2-2 \beta  \mu ^2}<0,
			\]
			where the latter inequality follows from (i) and (ii).

			\item[(v)]
			Multiply both sides of \eqref{eq:cased_condb} by the denominator of \eqref{eq:cased_condb} (which is positive by (ii)) and reorganize to get
         \[
         -\theta  (\beta-\mu^2 (1+2\beta) -\mu(1-\beta))-2 (\mu +1) (\mu\beta +\mu -\beta) \ge 0.
         \]
         \item[(vi)]
         Multiply both sides of \eqref{eq:cased_condb} by $\mu/(1+\mu)$ and reorganize to get
         		\[1-\theta\tfrac{\mu}{\mu+1}\ge \tfrac{(1-\mu) (\beta -\mu )}{\beta+\beta  \mu -\mu -\mu ^2-2 \beta  \mu ^2}>0, \]
			where the latter inequality follows from (i) and (ii).
		\end{itemize}
		
		(vi) shows $\lambda^A_\mu$ is nonnegative.
		(i) and (vi) show $\lambda_\beta^B$ is nonnegative.
		(iii) shows $K_1$ is nonnegative.
		(iv) and (v) show the numerator of $K_2$ is nonpositive.
		(iii) shows the denominator of $K_2$ is negative.
		(iii) shows the denominators of $m_2$ and $m_3$ are negative.

\paragraph{Case (e)}
Assume $0<\mu$, $0<\beta$, and $0<\theta<2$.
First, we show that if $(\mu,\beta,\theta)\in R_\mathrm{(e)}$, then
		\begin{align}
		&\theta <\tfrac{2 (\mu +1) (1+\beta)}{2 \mu\beta +\mu +\beta+1} \tag{e1}\label{eq:casee_conda}
		\end{align}
If \eqref{eq:casee_conda} does not hold, i.e.,
\[\theta \geq 2\tfrac{ \mu + \beta +\mu \beta +1}{\mu + \beta +2 \mu \beta+1},\]
then
\begin{equation*}
	\begin{aligned}
	\theta (\mu   +\beta+ 2\mu\beta) - 2(\mu +\beta +\mu\beta )&\geq  2\tfrac{ \mu  + \beta +\mu\beta +1}{\mu  +\beta+2 \mu\beta +1} (\mu +\beta + 2\mu\beta ) - 2(\mu  +\beta +\mu\beta)\\&=\tfrac{2 \mu\beta }{\mu  +\beta+2 \mu \beta+1}> 0,
	\end{aligned}
	\end{equation*}
which implies \eqref{eq:casec_conda}.
So
%In other words, not \eqref{eq:casee_conda} implies $(\mu,\beta,\theta)\in R_\mathrm{(c)}$, and therefore, 
\[
(\mu,\beta,\theta)\in R_\mathrm{(e)}
\quad\Rightarrow\quad
(\mu,\beta,\theta)\notin R_\mathrm{(c)}
\quad\Rightarrow\quad
\text{\eqref{eq:casee_conda}}.
\]

		\begin{itemize}
			\item[(i)]
			We have
			\[
			(2-\theta)\mu(\beta+1)+\theta\beta(1-\mu)=\theta(\beta-(2\beta+1)\mu)+2(\beta+1)\mu> 0,
			\]
			because either $\beta-(2\beta+1)\mu\geq 0$ and the inequality immediately follows or
			$\beta-(2\beta+1)\mu< 0$ and we use \eqref{eq:casee_conda} to get
			\begin{align*}
			\theta(\beta-(2\beta+1)\mu)+2(\beta+1)\mu&>
			\tfrac{2 (\mu +1) (1+\beta)}{2 \mu\beta +\mu +\beta+1}(\beta-(2\beta+1)\mu)
			+2(\beta+1)\mu\\
			&=
			\tfrac{2\beta (\beta+1)}{\mu  +\beta+2 \mu \beta+1}> 0.
			\end{align*}
			\item[(ii)] We have $2 \mu \beta (1-\theta) +(2-\theta) (\mu +\beta+1 )>0$, because
			\begin{align*}
			2 \mu \beta (1-\theta) +(2-\theta) (\mu +\beta+1 )&=
			-\theta(\mu+\beta+2 \mu\beta+1)+2 (\mu +1) (\beta+1)\\
			&>-\tfrac{2 (\mu +1) (\beta+1)}{\mu  +\beta+2 \mu\beta +1}(\mu +\beta+2 \mu\beta+1)+2 (\mu +1) (\beta+1)=0,
			\end{align*}
			where the inequality follows from plugging in \eqref{eq:casee_conda}.
			\item[(iii)] We have
			\[
			(2-\theta)(1+\mu)\beta-{\theta}\mu(\beta-1)=2 (\mu +1)\beta+\theta  (\mu  (1-2\beta)-\beta)> 0,
			\]
			because either $\mu  (1-2\beta)-\beta\geq 0$ and the inequality immediately follows or
			$\mu  (1-2\beta)-\beta<0$ and we use
			\begin{align*}
			2 (\mu +1)\beta+\theta  (\mu  (1-2\beta)-\beta) 
			&>
			2 (\mu +1)\beta+\tfrac{2 (\mu +1) (1+\beta)}{2 \mu\beta +\mu +\beta+1}  (\mu  (1-2\beta)-\beta)\\
			&=
			\tfrac{2 \mu  (\mu +1)}{2 \mu\beta +\mu +\beta+1}>0.
			\end{align*}
		\end{itemize}
		
		(ii) shows that the denominator of $\rho^2$ is positive.
		(i) and (iii) show that the numerator of $\rho^2$ is nonnegative.
		(i) shows that $\lambda^A_\mu$ is nonnegative.
		(i) shows that the numerator of $\lambda_\beta^B$ is nonnegative.
		(ii) shows that the denominator of $\lambda_\beta^B$ is positive.
		(ii) shows that the denominator of $K_1$ is positive.
		Hence we arrive to $\rho,\lambda^A_\mu,\lambda_\beta^B,K_1,K_2\geq 0$ in the region of interest.

\subsubsection{Lower bounds}
\label{sss:lower3}
It remains to show that $(\mu,\beta,\theta)\in R_{\mathrm{(e)}}\Rightarrow 0<K< 1/\beta^2$. In what follows, we first show $0<K$ in part I, and we then show $K< 1/\beta^2$ in part II.

Before we proceed, let us introduce the symbol $\neg$ that denotes the logical negation, and point out the following elementary fact.
If $f(\theta)=a\theta+b$ is an affine function of $\theta$, then either
\[
f(\theta)<\max\{f(\theta_\mathrm{min}),f(\theta_\mathrm{max})\}
\quad\text{ for all }\theta\in (\theta_\mathrm{min},\theta_\mathrm{max})
\]
or
\[
f(\theta)=
f(\theta_\mathrm{min})=f(\theta_\mathrm{max}).
\]
The two cases respectively correspond to $a\ne0$ and $a=0$.

% In what follows, we use the symmetry between $\mu$ and $\beta$ for reducing the number of cases to prove. Formally, there are eight settings to be verified. Using symmetry, one can assume $\beta \leq \mu$ and only study four settings (setting (d) is never active under this assumption), as the case $\beta\geq \mu$ would involve the exact same computations with $\beta$ and $\mu$ swapped.

\paragraph{Part I}
We now prove $0<K$.
Define
\begin{equation*}
\begin{aligned}
%a_1&=\beta  (2 (1-\theta ) \mu -\theta +2)+(2-\theta ) \mu,\\
a_1&=-\theta(2\beta\mu+\beta+\mu)+2(\beta\mu+\beta+\mu)\\
a_2&=\theta  \left(-(2 \beta +1) \mu ^2+(\beta -1) \mu +\beta\right)+2 (\mu +1) (\beta  (\mu -1)+\mu )\\
\end{aligned}
\end{equation*}
and
\begin{equation*}
    \begin{aligned}
    K_\mathrm{den}=&\theta ^2 \left(8 \beta ^3 \mu ^3+2 \beta ^3 \mu +2 \beta ^3+4 \beta ^2 \mu ^3+4 \beta ^2 \mu ^2-\beta ^2 \mu +\beta ^2-2 \beta  \mu ^3-\mu ^3-\mu ^2\right)\\&-4 \theta  \left(4 \beta ^3 \mu ^3+2 \beta ^3 \mu ^2+2 \beta ^3 \mu +2 \beta ^3+4 \beta ^2 \mu ^3+3 \beta ^2 \mu ^2+\beta ^2-\beta  \mu ^3-\mu ^3-\mu ^2\right)\\&+4 (\mu +1) \left(2 \beta ^3 \mu ^2+2 \beta ^3+3 \beta ^2 \mu ^2+\beta ^2-\mu ^2\right)
    \end{aligned}
\end{equation*}
so that 
\[
K=\frac{a_1a_2}{K_\mathrm{den}}.
\]

We have $a_1> 0$ since $(\mu,\beta,\theta)\notin R_\mathrm{(c)}$, i.e., $\theta<2\tfrac{\beta  \mu +\beta +\mu }{2 \beta  \mu +\beta +\mu }$.

We now show $a_2>0$. Since $(\mu,\beta,\theta)\notin R_\mathrm{(d)}$, we have 2 cases:
\begin{enumerate}
    \item[(i)] Assume $\neg\text{\eqref{eq:cased_conda}}$. Since $a_2$ is an affine function of $\theta$ and $\theta\in\left(0,2\tfrac{\beta  \mu +\beta +\mu}{2 \beta  \mu +\beta +\mu }\right)$, either
    \[
    a_2>\min\left\{
    a_2\big|_{\theta=0},a_2\big|_{\theta=2\tfrac{\beta  \mu +\beta +\mu}{2 \beta  \mu +\beta +\mu }}
    \right\}=
    \min\left\{2 (\mu +1) (\beta \mu -\beta +\mu ),
    \tfrac{4 \beta  \mu ^2}{2 \beta  \mu +\beta +\mu }
    \right\}\geq 0,
    \]
    where the first term is nonnegative by $\neg\text{\eqref{eq:cased_conda}}$, or 
    \[
    a_2=2 (\mu +1) (\beta \mu -\beta +\mu )=\tfrac{4 \beta  \mu ^2}{2 \beta  \mu +\beta +\mu } > 0.
    \]
    In both cases, $a_2>0$.
    \item[(ii)] Assume $\text{\eqref{eq:cased_conda}}$ and $\neg\text{\eqref{eq:cased_condb}}$.
    In Section~\ref{ss:theorem3_verification} case (d) part (ii), we proved that $\text{\eqref{eq:cased_conda}}$ implies 
    $-2 \beta  \mu ^2+\beta  \mu +\beta -\mu ^2-\mu >0$.
    Since $a_2$ is an affine function of $\theta$ and $\theta\in\left(\tfrac{2 (\mu +1) (\beta  (-\mu )+\beta -\mu )}{-2 \beta  \mu ^2+\beta  \mu +\beta -\mu ^2-\mu },2\tfrac{\beta  \mu +\beta +\mu}{2 \beta  \mu +\beta +\mu }\right)$, either
    \[
    a_2>\min\left\{
    a_2\big|_{\theta=\tfrac{2 (\mu +1) (\beta  (-\mu )+\beta -\mu )}{-2 \beta  \mu ^2+\beta  \mu +\beta -\mu ^2-\mu }},a_2\big|_{\theta=2\tfrac{\beta  \mu +\beta +\mu}{2 \beta  \mu +\beta +\mu }}
    \right\}=
    \min\left\{0,
    \tfrac{4 \beta  \mu ^2}{2 \beta  \mu +\beta +\mu }
    \right\}=0,
    \]
    or 
    \[
    a_2=0=\tfrac{4 \beta  \mu ^2}{2 \beta  \mu +\beta +\mu }.
    \]
    The latter case is impossible and we conclude $a_2>0$.
\end{enumerate}

We now show $K_\mathrm{den}>0$.
We will use the following elementary fact.
Let $f(\theta)=a\theta^2+b\theta+c$ be a quadratic function.
If $a\le 0$, $f(\theta_\mathrm{min})>0$, and $f(\theta_\mathrm{max})>0$, then $f(\theta)>0$ for all $\theta\in (\theta_\mathrm{min},\theta_\mathrm{max})$.
If $a>0$, 
$f(\theta_\mathrm{min})>0$, $f(\theta_\mathrm{max})>0$, and $f'(\theta_\mathrm{max})<0$, then $f(\theta)>0$ for all $\theta\in (\theta_\mathrm{min},\theta_\mathrm{max})$.
Consider the 2 cases:
\begin{itemize}
    \item[(i)] Assume $\neg\text{\eqref{eq:casea_conda}}$.
    Under this case, $\theta\in\left(0,2\tfrac{\beta  \mu +\beta +\mu}{2 \beta  \mu +\beta +\mu }\right)$.
    We view $K_\mathrm{den}$ as a quadratic function of $\theta$.
First, note
\[
K_\mathrm{den}\big|_{\theta=0}=4(\mu+1)((2 \beta +1) \beta ^2+(\beta +1)^2 (2 \beta -1) \mu ^2)
\]
and define  $\phi_\mu(\beta)=(2 \beta +1) \beta ^2+(\beta +1)^2 (2 \beta -1) \mu ^2$.
Since
        \begin{equation*}
            \begin{aligned}
            \phi_\mu\left(\tfrac{\mu }{\mu +1}\right)&=\tfrac{4 \mu ^5}{(\mu +1)^3}>0\\
            \tfrac{d\phi_\mu}{d\beta}(\beta)&=2 \beta  \left(3 (\beta +1) \mu ^2+3 \beta +1\right)>0
            \end{aligned}
        \end{equation*}
        and $\beta\geq \tfrac{\mu}{\mu+1}$ by $\neg\text{\eqref{eq:casea_conda}}$, we have $\phi_\mu(\beta)>0$ and $K_\mathrm{den}\big|_{\theta=0}>0$.
We have
\[
K_\mathrm{den}\big|_{\theta=2\tfrac{\beta  \mu +\beta +\mu}{2 \beta  \mu +\beta +\mu }}=\frac{16 \beta ^3 \mu ^2 (\beta +\mu )}{(2 \beta  \mu +\beta +\mu )^2}>0.
\]
Finally, note
\[
\tfrac{dK_{\mathrm{den}}}{d\theta}\big|_{\theta=2\tfrac{\beta  \mu +\beta +\mu}{2 \beta  \mu +\beta +\mu }}=-\frac{4 \beta  \mu  \left(4 \beta ^3 \mu +2 \beta ^3+4 \beta ^2 \mu ^2+2 \beta ^2+\beta  \mu -\mu ^2\right)}{2 \beta  \mu +\beta +\mu }
\]
and define
$\phi_\mu(\beta)= \left(4 \beta ^3 \mu +2 \beta ^3+4 \beta ^2 \mu ^2+2 \beta ^2+\beta  \mu -\mu ^2\right)$.
Since 
        \begin{equation*}
        \begin{aligned}
            \phi_\mu\left(\tfrac{\mu }{\mu +1}\right)&=\tfrac{\mu ^2 \left(3 \mu  (\mu +1)^2+2\right)}{(\mu +1)^3}>0\\
            \tfrac{d\phi_\mu}{d\beta}(\beta)&=2 \beta  \left(\beta  (6 \mu +3)+4 \mu ^2+2\right)+\mu>0
        \end{aligned}
        \end{equation*}
        and $\beta\geq \tfrac{\mu}{\mu+1}$ by $\neg\text{\eqref{eq:casea_conda}}$, we have $\phi_\mu(\beta)>0$ and $
\tfrac{dK_{\mathrm{den}}}{d\theta}\big|_{\theta=2\tfrac{\beta  \mu +\beta +\mu}{2 \beta  \mu +\beta +\mu }}<0$.
        % hence $\phi_\mu(\beta)>0$ is a consequence of $\beta\geq \tfrac{\mu}{\mu+1}$  (i.e.,$\neg\text{\eqref{eq:casea_conda}}$).
    Therefore, the quadratic function satisfies $K_\mathrm{den}>0$ for $\theta\in\left(0,2\tfrac{\beta  \mu +\beta +\mu}{2 \beta  \mu +\beta +\mu }\right)$.

    \item[(ii)] Assume $\text{\eqref{eq:casea_conda}}$ and $\neg\text{\eqref{eq:casea_condb}}$.
    Under this case,
    $\theta\in\left(\tfrac{2 (\beta +1) (\mu-\mu\beta-\beta )}{\mu  + \mu \beta - \beta-\beta^2 -2 \mu \beta^2 },2\tfrac{\beta  \mu +\beta +\mu}{2 \beta  \mu +\beta +\mu }\right)$, and, in Section~\ref{ss:theorem3_verification} case (a) part (ii), we proved that $\text{\eqref{eq:casea_conda}}$ implies $\mu  + \mu \beta - \beta-\beta^2 -2 \mu \beta^2 > 0$.
    
    We first show $\mu>\beta$.
    The fact that $\theta$ is in the interval implies
    \begin{align*}
            \tfrac{2 (\beta +1) (\mu-\mu\beta-\beta  )}{\mu  + \mu \beta - \beta-\beta^2 -2 \mu \beta^2 }&< 2\tfrac{\beta  \mu +\beta +\mu}{2 \beta  \mu +\beta +\mu }\\
            &\Leftrightarrow\quad
            \frac{4\beta^2\mu}{\beta ^2 (2 \mu +1)+\beta  (1-\mu )-\mu}=({2 \beta  \mu +\beta +\mu })\tfrac{2 (\beta +1) (\mu-\mu\beta-\beta  )}{\mu  + \mu \beta - \beta-\beta^2 -2 \mu \beta^2 }- 2({\beta  \mu +\beta +\mu})<0\\
             &\Leftrightarrow\quad \beta ^2 (2 \mu +1)+\beta  (1-\mu )-\mu< 0.
            \end{align*}
    Define 
    \[
    \phi_\mu(\beta)=\beta ^2 \underbrace{(2 \mu +1)}_{>0}+\beta  (1-\mu )-\mu.
    \]
    Since the coefficient of the quadratic term is positive, $\phi_\mu(0) =-\mu<0$, and $\phi_\mu(\mu)=2 \mu ^3>0$, 
    we conclude that $\phi_\mu(\beta)<0$ is only possible when $\beta<\mu$, i.e., $\phi_\mu(\beta)<0\Rightarrow \mu >\beta$.

    Now we view $K_\mathrm{den}$ as a quadratic function of $\theta$. We have
    \begin{align}
    & K_\mathrm{den}\big|_{\theta=\tfrac{2 (\beta +1) (\mu-\mu\beta-\beta )}{\mu  + \mu \beta - \beta-\beta^2 -2 \mu \beta^2 }}=\tfrac{16 \beta ^5 \mu ^2 (\mu -\beta )}{(\mu -\beta  (2 \beta  \mu +\beta -\mu +1))^2}>0\quad\text{(since $\mu>\beta$)}\nonumber\\
    & K_\mathrm{den}\big|_{\theta=2\tfrac{\beta  \mu +\beta +\mu}{2 \beta  \mu +\beta +\mu }}=\tfrac{16 \beta ^3 \mu ^2 (\beta +\mu )}{(2 \beta  \mu +\beta +\mu )^2}>0.\label{eq:kden_pos}
    \end{align}
    Define
        \[
        \phi(\beta,\mu)=8 \beta ^3 \mu ^3+2 \beta ^3 \mu +2 \beta ^3+4 \beta ^2 \mu ^3+4 \beta ^2 \mu ^2-\beta ^2 \mu +\beta ^2-2 \beta  \mu ^3-\mu ^3-\mu ^2,
        \]
    which is the coefficient for the quadratic term of $K_\mathrm{den}$.
    If $\phi(\beta,\mu)\le 0$ (i.e., the curvature $K_\mathrm{den}$ of is nonpositive), then the two inequalities \eqref{eq:kden_pos} implies $K_\mathrm{den}>0$ for  $\theta\in\left(\tfrac{2 (\beta +1) (\mu-\mu\beta-\beta )}{\mu  + \mu \beta - \beta-\beta^2 -2 \mu \beta^2 },2\tfrac{\beta  \mu +\beta +\mu}{2 \beta  \mu +\beta +\mu }\right)$.

Now assume $\phi(\beta,\mu)>0$.
In this case, we have
        \begin{align*}
            \tfrac{dK_\mathrm{den}}{d\theta}\big|_{\theta=2\tfrac{\beta  \mu +\beta +\mu}{2 \beta  \mu +\beta +\mu }} 
            &=-\tfrac{4 \beta  \mu  \left(\beta ^3 (4 \mu +2)+\beta ^2 \left(4 \mu ^2+2\right)+\beta  \mu -\mu ^2\right)}{2 \beta  \mu +\beta +\mu }\\
            &<-\tfrac{4 \beta  \mu }{2 \beta  \mu +\beta +\mu }
            \left(\beta ^3 4 \mu +\beta ^24 \mu ^2-\mu ^2\right)\\
            &=-\tfrac{4 \beta  \mu^2 }{2 \beta  \mu +\beta +\mu }
            \left(4\beta^2(\beta  + \mu)   -\mu \right)
        \end{align*}
        If $4 \beta ^2 (\beta +\mu )-\mu >0$, then $
            \tfrac{dK_\mathrm{den}}{d\theta}\big|_{\theta=2\tfrac{\beta  \mu +\beta +\mu}{2 \beta  \mu +\beta +\mu }} <0$,
            and we conclude that $K_\mathrm{den}>0$.

        We now show
        $         4 \beta ^2 (\beta +\mu )-\mu >0$.
        Assume for contradiction that $4 \beta ^2 (\beta +\mu )-\mu \leq0,$ or equivalently,        \begin{equation*}
         4 \beta ^2 (\beta +\mu )- \mu \le 0
         \quad\Leftrightarrow\quad
         4\beta^3\le \mu(1-4\beta^2),
         \label{eq:cont-assump}
        \end{equation*}
        which implies $\beta< 1/2$ and hence
        \begin{equation*}
         4 \beta ^2 (\beta +\mu )- \mu \le 0
         \quad\Leftrightarrow\quad
         4\beta^3\le \mu(1-4\beta^2) \quad\Leftrightarrow\quad \mu\ge \tfrac{4\beta^3}{1-4\beta^2},
        \end{equation*}
        and also
        \[
        \tfrac{d\phi}{d\mu}(\beta,\mu)=(2 \beta -1) \left(\beta ^2+4 \beta  \mu +3 (2 \beta  \mu +\mu )^2+2 \mu \right)<0.
        \]
        We have assumed for contradiction that
        $\mu\ge \tfrac{4\beta^3}{1-4\beta^2}$, and on the other hand, $\text{\eqref{eq:casea_conda}}$ and $\beta<1/2$ imply $\mu>\tfrac{\beta}{1-\beta}$.
        If \[
        \phi\left(\beta,\max\left\{
        \tfrac{4\beta^3}{1-4\beta^2},\tfrac{\beta}{1-\beta}
        \right\}\right)<0,
        \] then $\phi(\beta,\mu)<0$, the contradiction forces us to conclude $4 \beta ^2 (\beta +\mu )-\mu >0$.
        \begin{itemize}
            \item[(A)] If
         $\frac{4 \beta ^3}{1-4 \beta ^2}\geq\frac{\beta }{1-\beta }\,\Leftrightarrow\,-4 \beta ^3+8 \beta ^2-1\geq 0$, then 
         \[\phi\left(\beta,\tfrac{4 \beta ^3}{1-4 \beta ^2}\right)=-\tfrac{\beta ^2 \left(64 \beta ^7-16 \beta ^5-16 \beta ^4+4 \beta ^3+8 \beta ^2-1\right)}{(2 \beta -1)^2 (2 \beta +1)}<0,\]
         since
        \begin{align*}
        64 \beta ^7-16 \beta ^5-16 \beta ^4+4 \beta ^3+8 \beta ^2-1&\geq64 \beta ^7-16 \beta ^5-16 \beta ^4+8 \beta ^3\quad \text{(since $-4 \beta ^3+8 \beta ^2-1\geq 0$)}\\&=8 \beta ^3 (2 \beta -1)^2 \left(2 \beta ^2+2 \beta +1\right)\\
        &>0\quad \text{(using $\beta<\tfrac12$).}
        \end{align*}
            \item[(B)] If $\frac{4 \beta ^3}{1-4 \beta ^2}<\frac{\beta }{1-\beta }\,\Leftrightarrow\,-4 \beta ^3+8 \beta ^2-1< 0$),
            then
                    \[\phi\left(\beta,\tfrac{\beta }{1-\beta }\right)=\tfrac{\beta ^3 \left(8 \beta ^3+3 \beta -2\right)}{(1-\beta )^3}<0,\]
        since
        \[
        -4 \beta ^3+8 \beta ^2-1< 0 \text{ and }\beta\in (0,1/2)
        %&\quad\Leftrightarrow\quad        \beta\in (0,0.394622)\\
        \quad\Rightarrow\quad
        8 \beta ^3+3 \beta -2< 0.
        \]
        
        This final point can be verified with simple plotting or with the following algebraic argument.
        Define $\chi(\beta)=-4 \beta ^3+8 \beta ^2-1$ 
        and $\psi(\beta)=8 \beta ^3+3 \beta -2$.
        Since $\chi'(\beta)=4 (4-3 \beta ) \beta$ and $\psi'(\beta)=3 \left(8 \beta ^2+1\right)$, 
        the functions $\chi$ and $\psi$ are increasing on $\beta\in(0,1/2)$.
        We have $\chi(0)<0$ and $\psi(0)<0$.
        Furthermore, 
                \[\chi\left(\tfrac{\sqrt[3]{2} \left(3 \sqrt{2}+4\right)^{2/3}-2^{2/3}}{4 \sqrt[3]{3 \sqrt{2}+4}}\right)=\tfrac{1}{4} \left(\tfrac{\sqrt[3]{2} \left(48 \sqrt{2}+67\right)}{\left(3 \sqrt{2}+4\right)^{5/3}}+\tfrac{\sqrt[6]{2} \left(62 \sqrt{2}+87\right)}{\left(3 \sqrt{2}+4\right)^{4/3}}-16\right)\approx 0.207156>0\]
        and
                \[ \psi\left(\tfrac{\sqrt[3]{2} \left(3 \sqrt{2}+4\right)^{2/3}-2^{2/3}}{4 \sqrt[3]{3 \sqrt{2}+4}}\right)=0.\]
So $\chi$ and $\psi $ are increasing functions and $\chi$ hits its first positive root before $\psi$.
Therefore  $\chi(\beta)<0\,\Rightarrow\,\psi(\beta)<0$ on $\beta\in(0,\tfrac12)$.

    \end{itemize}

\end{itemize}

{\paragraph{Part II}
We now prove $K<1/\beta^2$.}
Define
\begin{equation*}
\begin{aligned}
a_3&=\theta  \left(\beta ^2 \left(2 \mu ^2-\mu -1\right)-\beta  \left(\mu ^2+1\right)-\mu  (\mu +1)\right)+2 (\beta +1) (\mu +1) (  \beta +\mu-\mu\beta ),\\
a_4&=\theta  \left(\beta ^2 (2 \mu +1)-\beta  \mu +\beta -\mu \right)-2 (\beta +1) (\beta  \mu +\beta -\mu ).
\end{aligned}
\end{equation*}
Then $\beta^2\left(K-1/\beta^2\right)K_\mathrm{den}=a_3a_4$,
where we know $K_\mathrm{den}>0$ from part I.
Now verifying $K<\tfrac1{\beta^2}$  is equivalent to verifying $a_3 a_4< 0$.

% The only inequality to show for the lower bounds is for case (e).
% It is possible to verify that $0<K\leq 1/\beta^2$ when $(\mu,\beta,\theta)\in R_{\mathrm{(e)}}$.

Given $0<\mu$, $0<\beta$, and $0<\theta<2$, we have
\[
(\mu,\beta,\theta)\in R_{\mathrm{(a)}}
\quad\Leftrightarrow\quad
\text{\eqref{eq:casea_conda}}\text{ and }\text{\eqref{eq:casea_condb}}
\]
and
\[
(\mu,\beta,\theta)\in R_{\mathrm{(b)}}
\quad\Leftrightarrow\quad
\text{\eqref{eq:caseb_conda}}\text{ and }\text{\eqref{eq:caseb_condb}},
\]
where \eqref{eq:casea_conda}, \eqref{eq:casea_condb}, \eqref{eq:caseb_conda}, and \eqref{eq:caseb_condb} are defined as in Section~\ref{ss:theorem3_verification}.
We show
\[
(\mu,\beta,\theta)\in R_{\mathrm{(e)}}
\quad\Rightarrow\quad
(\mu,\beta,\theta)\notin R_{\mathrm{(a)}}
\text{ and }
(\mu,\beta,\theta)\notin R_{\mathrm{(b)}}
\text{ and }
(\mu,\beta,\theta)\notin R_{\mathrm{(c)}}
\quad\Rightarrow\quad
a_3 a_4< 0
\]
by considering the following 3 cases:
\begin{enumerate}
    \item[(i)]
    $    \neg\text{\eqref{eq:casea_conda}} \text{ and }\neg\text{\eqref{eq:caseb_conda}} 
    \text{ and }
    (\mu,\beta,\theta)\notin R_{\mathrm{(c)}}
    \quad\Rightarrow\quad a_3 a_4<0$
    %Not \eqref{eq:casea_conda}, not \eqref{eq:caseb_conda}, $(\mu,\beta,\theta)\notin R_{\mathrm{(c)}}$ implies 
    \item[(ii)]
    %Not \eqref{eq:casea_conda}, not \eqref{eq:caseb_condb}, $(\mu,\beta,\theta)\notin R_{\mathrm{(c)}}$ implies $K\leq 1/\beta^2$;
    $
    \neg\text{\eqref{eq:casea_conda}} \text{ and }
    \text{\eqref{eq:caseb_conda}} \text{ and }
    \neg\text{\eqref{eq:caseb_condb}} 
    \text{ and }
    (\mu,\beta,\theta)\notin R_{\mathrm{(c)}}
    \quad\Rightarrow\quad a_3 a_4< 0
    $
    \item[(iii)]
    %    \eqref{eq:casea_conda}, not \eqref{eq:casea_condb}, not \eqref{eq:caseb_conda}, $(\mu,\beta,\theta)\notin R_{\mathrm{(c)}}$ implies $K\leq 1/\beta^2$;
    $
    \text{\eqref{eq:casea_conda}} \text{ and }\neg\text{\eqref{eq:casea_condb}}
    \text{ and }
    (\mu,\beta,\theta)\notin R_{\mathrm{(c)}}
    \quad\Rightarrow\quad a_3 a_4< 0
    $
\end{enumerate}

\begin{enumerate}
    \item[(i) \& (ii)]
    %(Using $\beta  \mu +\beta -\mu \ge 0$ and $\theta <2\tfrac{\beta  \mu +\beta +\mu}{2 \beta  \mu +\beta +\mu }$)
    Since $a_4$ is an affine function of $\theta$ and $\theta\in\left(0,2\tfrac{\beta  \mu +\beta +\mu}{2 \beta  \mu +\beta +\mu }\right)$, either
    \[
    a_4<\max\left\{
    a_4\big|_{\theta=0},a_4\big|_{\theta=2\tfrac{\beta  \mu +\beta +\mu}{2 \beta  \mu +\beta +\mu }}
    \right\}=
    \max\left\{-2 (\beta +1) (\beta  \mu +\beta -\mu ),
    -\tfrac{4 \beta ^2 \mu }{2 \beta  \mu +\beta +\mu }
    \right\}\le 0,
    \]
    where the first term is nonpositive by the assumption $\beta  \mu +\beta -\mu \ge0$,
    or 
    \[
    a_4=-2 (\beta +1) (\beta  \mu +\beta -\mu )=
    -\tfrac{4 \beta ^2 \mu }{2 \beta  \mu +\beta +\mu }
    < 0.
    \]
    In both cases, $a_4<0$.
    %($\neg\text{\eqref{eq:casea_conda}}$).
    
    % Let us start with $a_4$; we proceed in two steps: (i) assume $\beta ^2 (2 \mu +1)-\beta  \mu +\beta -\mu \geq 0$, then one can use $\theta <2\tfrac{\beta  \mu +\beta +\mu}{2 \beta  \mu +\beta +\mu }$ and obtain the desired result as
    % \[a_4< -\tfrac{4 \beta ^2 \mu }{2 \beta  \mu +\beta +\mu }<0;\]
    % on the other hand, when (ii) $\beta ^2 (2 \mu +1) \leq \beta  \mu -\beta +\mu$ one can do
    % \[a_4\leq \theta  \left(\beta  \mu -\beta +\mu-\beta  \mu +\beta -\mu \right)-2 (\beta +1) (\beta  \mu +\beta -\mu )=-2 (\beta +1) (\beta  \mu +\beta -\mu )<0.\]
     \begin{itemize}
        \item[(i)] 
        %XXX(Using $\beta  \mu +\beta -\mu \geq 0$, $\beta  \mu -\beta -\mu \leq 0$, and $\theta <2\tfrac{\beta  \mu +\beta +\mu}{2 \beta  \mu +\beta +\mu }$)
    Since $a_3$ is an affine function of $\theta$ and $\theta\in\left(0,2\tfrac{\beta  \mu +\beta +\mu}{2 \beta  \mu +\beta +\mu }\right)$, either
            \[
    a_3> \min\left\{
    a_3\big|_{\theta=0},a_3\big|_{\theta=2\tfrac{\beta  \mu +\beta +\mu}{2 \beta  \mu +\beta +\mu }}
    \right\}=
    \min\left\{
    2 (\beta +1) (\mu +1) (  \beta +\mu-\beta\mu ),
    \tfrac{4 \beta  \mu  (\beta +\mu )}{2 \beta  \mu +\beta +\mu }
    \right\}\ge 0,
    \]
     where the first term is nonnegative by $\neg\text{\eqref{eq:caseb_condb}}$, which is $\beta  \mu -\beta -\mu \leq 0$,
     or
     \[
    a_3=
    2 (\beta +1) (\mu +1) (  \beta +\mu-\beta\mu )=
    \tfrac{4 \beta  \mu  (\beta +\mu )}{2 \beta  \mu +\beta +\mu }
    > 0.
    \]
    In both cases, we have $a_3>0$.
    
    \item[(ii)]
As we had shown in
Section~\ref{ss:theorem3_verification} case (b) part (ii), we have
    \[
    \text{\eqref{eq:caseb_conda}}\quad\Rightarrow\quad
    -2 \beta ^2 \mu ^2+\beta ^2 \mu +\beta ^2+\beta  \mu ^2+\beta +\mu ^2+\mu <0,
    \]
    and we have no division by $0$ in considering $\neg\text{\eqref{eq:caseb_condb}}$.

Since $a_3$ is an affine function of $\theta$ and $\theta\in\left(\tfrac{2 \left(\beta ^2+\beta  \mu +\beta +\mu ^2+\mu -\mu ^2\beta ^2\right)}{-2 \beta ^2 \mu ^2+\beta ^2 \mu +\beta ^2+\beta  \mu ^2+\beta +\mu ^2+\mu },2\tfrac{\beta  \mu +\beta +\mu}{2 \beta  \mu +\beta +\mu }\right)$, either
            \[
    a_3> \min\left\{
    a_3\big|_{\theta=\tfrac{2 \left(\beta ^2+\beta  \mu +\beta +\mu ^2+\mu -\mu ^2\beta ^2\right)}{-2 \beta ^2 \mu ^2+\beta ^2 \mu +\beta ^2+\beta  \mu ^2+\beta +\mu ^2+\mu }},a_3\big|_{\theta=2\tfrac{\beta  \mu +\beta +\mu}{2 \beta  \mu +\beta +\mu }}
    \right\}=
    \min\left\{
    0,
    \tfrac{4 \beta  \mu  (\beta +\mu )}{2 \beta  \mu +\beta +\mu }
    \right\}= 0,
    \]
    or
    \[
    a_3=0=
    \tfrac{4 \beta  \mu  (\beta +\mu )}{2 \beta  \mu +\beta +\mu }.
    \]
    The latter case is impossible, and we conclude $a_3>0$.
    \end{itemize}
    
    \item[(iii)]
As we had shown in
Section~\ref{ss:theorem3_verification} case (a) part (ii), we have
    \[
    \text{\eqref{eq:casea_conda}}\quad\Rightarrow\quad
    \mu+\beta  \mu-\beta -\beta ^2 -2 \beta ^2 \mu>0,
    \]
    and we have no division by $0$ in considering $\neg\text{\eqref{eq:casea_condb}}$.

    Since $a_4$ is an affine function of $\theta$ and $\theta\in\left(2 \tfrac{(\beta +1) (\mu-\beta\mu-\beta )}{\mu+\beta  \mu-\beta -\beta ^2 -2 \beta ^2 \mu },2\tfrac{\beta \mu +\beta +\mu}{2 \beta  \mu +\beta +\mu }\right)$, either
    \[
    a_4< \max\left\{
    a_4\big|_{\theta=2 \tfrac{(\beta +1) (\mu-\beta\mu-\beta )}{\mu+\beta  \mu-\beta -\beta ^2 -2 \beta ^2 \mu }},a_4\big|_{\theta=2\tfrac{\beta  \mu +\beta +\mu}{2 \beta  \mu +\beta +\mu }}
    \right\}=
    \max\left\{0
    ,
    -\tfrac{4 \beta ^2 \mu }{2 \beta  \mu +\beta +\mu }
    \right\}= 0.
    \]
    or
    \[
    a_4=0
    =
    -\tfrac{4 \beta ^2 \mu }{2 \beta  \mu +\beta +\mu }
    \]
    The latter case is impossible, and we conclude $a_4<0$.
    
 Since $a_3$ is an affine function of $\theta$ and $\theta\in(2 \tfrac{(\beta +1) (\mu-\beta\mu-\beta )}{\mu+\beta  \mu-\beta -\beta ^2 -2 \beta ^2 \mu },2\tfrac{\beta \mu +\beta +\mu}{2 \beta  \mu +\beta +\mu })$, either
            \[
    a_3>\min\left\{
    a_3\big|_{\theta=2 \tfrac{(\beta +1) (\mu-\beta\mu-\beta )}{\mu+\beta  \mu-\beta -\beta ^2 -2 \beta ^2 \mu }},a_3\big|_{\theta=2\tfrac{\beta  \mu +\beta +\mu}{2 \beta  \mu +\beta +\mu }}
    \right\}=
    \min\left\{
    \tfrac{4 \beta  \left(1-\beta ^2\right) \mu ^2}{\mu+\beta  \mu-\beta -\beta ^2 -2 \beta ^2 \mu },
    \tfrac{4 \beta  \mu  (\beta +\mu )}{2 \beta  \mu +\beta +\mu }
    \right\}> 0,
    \]
    where the first term is positive since $1>\beta$ follows from \eqref{eq:casea_conda},
    or
    \[
    a_3=
    \tfrac{4 \beta  \left(1-\beta ^2\right) \mu ^2}{\mu+\beta  \mu-\beta -\beta ^2 -2 \beta ^2 \mu }=
    \tfrac{4 \beta  \mu  (\beta +\mu )}{2 \beta  \mu +\beta +\mu }>0.
    \]
    In both cases, we have $a_3>0$.

\end{enumerate}

\subsection{Inequalities for Theorem~\ref{thm:DRS_Lipschitz_strmonotone}}
\label{ss:thm4ineq}

\subsubsection{Upper bounds}
\label{ss:lip-mon-upper}
It remains to show that there is no division by $0$ and $\lambda^A_\mu,\lambda_L^B,K_1,K_2\geq 0$ in each case.

\paragraph{Case (a)}
Since the quantities in Section~\ref{sec:upper_bound_Lips} Case (a) are simple, 
we verify by inspection that there is no division by $0$ and $\lambda^A_\mu,\lambda_L^B,K_1,K_2\geq 0$.

\paragraph{Case (b)}
Assume $0<\mu$, $0<L$, $0<\theta<2$, and
	\begin{align}
	& \label{eq:thm112b1} L<1\\
	& \label{eq:thm112b2}\mu >\tfrac{L^2+1}{(L-1)^2} \\
	&\label{eq:thm112b3}\theta \leq 2 (\mu +1) (L+1)\tfrac{\mu +\mu  L^2-L^2-2 \mu  L-1}{2 \mu ^2-\mu +\mu  L^3-L^3-3 \mu  L^2-L^2-2 \mu ^2 L-\mu  L-L-1}\\
	 		&\hspace{1.5in}\text{(no division by $0$ implied by (ii) below)}\nonumber
	\end{align}
	\begin{itemize}
		\item[(i)] From~\eqref{eq:thm112b1} and~\eqref{eq:thm112b2}, it is direct to obtain $\mu>1$.
		\item[(ii)] The numerator of \eqref{eq:thm112b3} is positive since, by \eqref{eq:thm112b2}, we have
		\[
		\mu +\mu  L^2-L^2-2 \mu  L-1
		=\mu  (1-L)^2-\left(L^2+1\right)>0.
		\]
        Since $\theta>0$, the denominator of \eqref{eq:thm112b3} is nonnegative.
        To prove strict positivity, we view the denominator of \eqref{eq:thm112b3} as a quadratic function of $\mu$:
        \begin{align*}
        \phi_{L,\theta}(\mu)&=2 \mu ^2-\mu +\mu  L^3-L^3-3 \mu  L^2-L^2-2 \mu ^2 L-\mu  L-L-1\\
        &=\underbrace{2(1-L)}_{>0\text{ by \eqref{eq:thm112b1}}}\mu^2+(L^3-3L^2-L-1)\mu-(1+L)(1+L^2).
        \end{align*}
        This quadratic is nonpositive only between its roots and
        \[
        \phi_{L,\theta}(0)=-(1+L)(1+L^2)<0,\quad \phi_{L,\theta}\left(\tfrac{L^2+1}{(L-1)^2}\right)=2 L\tfrac{ \left(L^2+1\right)^2}{(1-L)^3}>0.
        \]
        Therefore $\phi_{L,\theta}(\mu)>0$ for all $\mu>\tfrac{L^2+1}{(L-1)^2}$, which holds by \eqref{eq:thm112b2}.
        Therefore we conclude the denominator of \eqref{eq:thm112b3} is strictly positive.
		\item[(iii)]
		%The factor $\left(1-\theta\tfrac{\mu +L}{(\mu +1) (L+1)}\right)$ appearing in the multipliers is also nonnegative as
		Multiply both sides of \eqref{eq:thm112b3} by $(\mu +L)/(\mu +1) (L+1)$ and reorganize to get
		\begin{equation*}
		\begin{aligned}
		1-\theta\tfrac{\mu +L}{(\mu +1) (L+1)}&\ge 1-2(\mu+L)\tfrac{\mu +\mu  L^2-L^2-2 \mu  L-1}{2 \mu ^2-\mu +\mu  L^3-L^3-3 \mu  L^2-L^2-2 \mu ^2 L-\mu  L-L-1}\\
		&=\tfrac{(\mu -1) (1-L) (1+2\mu L+L^2)}{2 \mu ^2-\mu +\mu  L^3-L^3-3 \mu  L^2-L^2-2 \mu ^2 L-\mu  L-L-1}\\
		&>0,
		\end{aligned} 
		\end{equation*}
		where the latter inequality follows from \eqref{eq:thm112b1}, (i), and (ii).
		\item[(iv)]
		%For the numerator of $K_1$ (also appearing in $m_1$ and $m_2$), we have
		Multiply both sides of \eqref{eq:thm112b3} by $(2 \mu +L+1)$ and reorganize to get
		\begin{equation*}
		\begin{aligned}
		2 (\mu +1) (L+1)-\theta  (2 \mu +L+1)\ge\tfrac{4 \mu ^2 (\mu +1) L \left(1-L^2\right)}{2 \mu ^2-\mu +\mu  L^3-L^3-3 \mu  L^2-L^2-2 \mu ^2 L-\mu  L-L-1}>0,
		\end{aligned}
		\end{equation*}
		where the latter inequality follows from \eqref{eq:thm112b1} and (ii).
		\item[(v)]
				Multiply both sides of \eqref{eq:thm112b3} by the denominator of \eqref{eq:thm112b3} (which is positive by (ii)) and reorganize to get
		\begin{equation*}
		\begin{aligned}
		2 &(\mu +1) (L+1) \left(\mu  (1-L)^2-\left(L^2+1\right)\right)+\theta  \left(\mu  \left(1+L+3 L^2-L^3\right)+\left(1+L+L^2+L^3\right)+2 \mu ^2 (L-1)\right)\ge 0.
		\end{aligned}
		\end{equation*}
		\item[(vi)] Multiply both sides of~\eqref{eq:thm112b3} by $L+\mu$ and reorganize to get
		\[
		(L+1)(\mu+1)-\theta(L+\mu)\geq 
		\tfrac{\left(\mu^2-1\right) (1-L^2)(1+L^2+2\mu L)}{2 \mu ^2-\mu +\mu  L^3-L^3-3 \mu  L^2-L^2-2 \mu ^2 L-\mu  L-L-1}.
		\]
		Since $\mu>1$ by (i),  $L<1$ by \eqref{eq:thm112b1}, and the denominator is positive by (ii), we have
		\[
		(L+1)(\mu+1)-\theta(L+\mu)>0.
		\]
	\end{itemize}

	(iv) shows $\lambda^A_\mu$ is nonnegative.
	(i) and (iv) show $\lambda^B_L$ is nonnegative.
	(v) shows the numerator of $K_1$ is nonnegative.
	(v) and (vi) show the numerator of $K_2$ is nonnegative. 
	(v) shows the denominator of $K_2$ is positive.
	(v) shows the denominators of $m_2$ and $m_3$ are positive.

\paragraph{Case (c)}
Assume $(\mu,L,\theta)\in R_\mathrm{(c)}$. Then:
% Assume $0<\mu$, $0<L$, and $0<\theta<2$.
% First, we show that if $(\mu,L,\theta)\in R_\mathrm{(c)}$, then
% \begin{equation}
% \begin{aligned}
% \theta(L^2+1)-2\mu(\theta+L^2-1)\geq 0,\quad (2-\theta)(1-L^2)-2\mu(\theta+L^2-1)> 0.\label{eq:thm112b5}
% \end{aligned}
% \end{equation}
\begin{itemize}
\item[(i)] We show $L<1$.
Since $(\mu,L,\theta)\notin R_\mathrm{(a)}$, we have
	\[
	0<\sqrt{L^2+1}<
	\mu\tfrac{-\left(2 (\theta -1) \mu +\theta-2\right)+L^2\left(\theta -2(1+ \mu)\right)}{\sqrt{ (2 (\theta -1) \mu +\theta -2)^2+L^2 (\theta -2 (\mu +1))^2}}.
	\] 
	This implies the numerator of the right-most-side is positive:
	\begin{equation}
	0<-\left(2 (\theta -1) \mu +\theta-2\right)+L^2\underbrace{\left(\theta -2(1+ \mu)\right)}_{<0\text{ by }\mu>0,\,\theta<2}.
	\label{eq:i-main}
	\end{equation}
	Therefore
	\[
	L^2<\frac{2 (\theta -1) \mu+\theta-2}{\theta -2(1+ \mu)}=
	1+
	\underbrace{\frac{2\theta\mu}{\theta -2(1+ \mu)}}_{<0\text{ by }\mu>0,\,\theta<2}<1,
	\]
	and we conclude $L<1$.

\item[(ii)]
We have $(2-\theta ) \left(1-L^2\right)-2 \mu  \left(\theta +L^2-1\right)>0$, since it is a simple reformulation of \eqref{eq:i-main}.
\item[(iii)]
We have $\theta(L^2+1)-2\mu(\theta+L^2-1)>0$, because
either $\theta+L^2-1< 0$ and
	\[
	\theta(L^2+1)-2\mu(\theta+L^2-1)>0
	\]
or $\theta+L^2-1\ge  0$ and
	\[
	\theta(L^2+1)-2\mu(\theta+L^2-1)>\theta(L^2+1)-(2-\theta)(1-L^2)=2 \left(\theta +L^2-1\right)\geq 0,
	\]
where the first inequality follows from (ii).
\item[(iv)] We prove $\theta  \left(1+2 \mu +L^2\right)-2 (\mu +1) \left(L^2+1\right)<0$.
Note that \eqref{eq:i-main}, the main inequality from (i), can be reformulated as
\begin{equation*}
    \begin{aligned}
    0&<(2-\theta ) \left(1-L^2\right)-2 \mu  \left(\theta +L^2-1\right)\\
    &=\theta \underbrace{\left(-2 \mu +L^2-1\right)}_{<0\text{ by }(i)}-2 (\mu +1) \left(L^2-1\right).
    \end{aligned}
\end{equation*}
Therefore,
$\theta <\tfrac{2 (\mu +1) \left(1-L^2\right)}{2 \mu -L^2+1}$
and plugging this into $\theta$ gives us
\[\theta  \left(1+2 \mu +L^2\right)-2 (\mu +1) \left(L^2+1\right)<-\tfrac{8 \mu  (\mu +1) L^2}{2 \mu+1-L^2}<0.\]
% \[0<\theta  \left(-2 \mu +L^2-1\right)-2 (\mu +1) \left(L^2-1\right)\Leftrightarrow \theta <\tfrac{2 (\mu +1) \left(1-L^2\right)}{2 \mu -L^2+1}.\]

\item[(v)] Finally, we show $\mu>1$.  (We use this later in Section~\ref{ss:lowerbound-Lipschitz-monotone}.)
Since $(\mu,L,\theta)\notin R_\mathrm{(a)}$ we have
	\[
	1<\sqrt{L^2+1}<
	\mu\tfrac{-\left(2 (\theta -1) \mu +\theta-2\right)+L^2\left(\theta -2(1+ \mu)\right)}{\sqrt{ (2 (\theta -1) \mu +\theta -2)^2+L^2 (\theta -2 (\mu +1))^2}},
	\] 
	which implies
    \[ \sqrt{(2 (\theta -1) \mu +(\theta -2))^2+L^2 (\theta -2 (\mu +1))^2}<\mu\left(-\left(2 (\theta -1) \mu +\theta-2\right)+L^2\left(\theta -2(1+ \mu)\right)\right),\]
    which, in turn, implies
    \begin{equation*}
        \begin{aligned}
        \sqrt{(2 (\theta -1) \mu +(\theta -2))^2}<&\sqrt{(2 (\theta -1) \mu +(\theta -2))^2+L^2 (\theta -2 (\mu +1))^2}\\
        &<\mu\left(-\left(2 (\theta -1) \mu +\theta-2\right)+L^2\left(\theta -2(1+ \mu)\right)\right)<\mu\left(-\left(2 (\theta -1) \mu +\theta-2\right)\right),
        \end{aligned}
    \end{equation*}
    where we used $\theta<2<2(1+\mu)$ for the last inequality. Finally, using the implied $\mu\left(-\left(2 (\theta -1) \mu +\theta-2\right)\right)>0$, we get
    \begin{equation*}
        \begin{aligned}
        0<\left(-\left(2 (\theta -1) \mu +\theta-2\right)\right)<\mu\left(-\left(2 (\theta -1) \mu +\theta-2\right)\right)\Rightarrow1<\mu.
        \end{aligned}
    \end{equation*}
\end{itemize}

(iii) and (iv) shows the numerator of $\rho^2$ is nonpositive.
(ii) shows the denominator of $\rho^2$ is negative.
(iii) shows that $\lambda^A_\mu$ is nonnegative.
(iii) shows that the numerator of  $\lambda^B_L$ is nonnegative.
(ii) shows that the denominator of $\lambda^B_L$ is positive.
(ii) shows that the denominator of $K_1$ is positive.

% 	\[L<1,\text{ and either } \mu <\frac{(\theta -2) \left(L^2-1\right)}{2 \left(\theta +L^2-1\right)} \text{ or } \theta +L^2\leq 1 \]

\subsubsection{Lower bounds}
\label{ss:lowerbound-Lipschitz-monotone}
It remains to show that $(\mu,L,\theta)\in R_\mathrm{(c)}\Rightarrow 0\le K\le 1$
Note that
\[
(\mu,L,\theta)\in R_\mathrm{(c)}
\quad\Rightarrow\quad
(\mu,L,\theta)\notin R_\mathrm{(a)}\text{ and }
(\mu,L,\theta)\notin R_\mathrm{(b)}.
\]

Write 
\[
K=\tfrac{L^2+1}{2L}\tfrac{K_\mathrm{num}}{K_\mathrm{den}}.
\]
In part I, we show $K_\mathrm{den}>0$ and $K_\mathrm{num}\ge 0$ using $(\mu,L,\theta)\notin R_\mathrm{(a)}$.
In part I, we show $K\le 1$ using $(\mu,L,\theta)\notin R_\mathrm{(a)}$ and $(\mu,L,\theta)\notin R_\mathrm{(b)}$.

\paragraph{Part I}
We now show $K_\mathrm{den}>0$ and $K_\mathrm{num}\ge 0$.

First, we quickly recall 
$(\mu,L,\theta)\notin R_\mathrm{(a)}\Rightarrow \mu>1$ and $(\mu,L,\theta)\notin R_\mathrm{(a)}\Rightarrow L<1$ by respectively (v) and (i) of case (c) of Section~\ref{ss:lip-mon-upper}.

Next we move on to the main proof. Note
\begin{equation*}
\begin{aligned}
&(\mu,L,\theta)\notin R_\mathrm{(a)}\\
&\quad\Rightarrow\quad
\tfrac{\mu  \left(-2 (\theta -1) \mu -\theta +L^2 (\theta -2 (\mu +1))+2\right)}{\sqrt{(2 (\theta -1) \mu +\theta -2)^2+L^2 (\theta -2 (\mu +1))^2}}> \sqrt{L^2+1}\\
&\quad\Rightarrow\quad\mu  \left(-2 (\theta -1) \mu -\theta +L^2 (\theta -2 (\mu +1))+2\right)>\sqrt{L^2+1}\sqrt{(2 (\theta -1) \mu +\theta -2)^2+L^2 (\theta -2 (\mu +1))^2}\\
&\quad\Rightarrow\quad\left(\mu  \left(-2 (\theta -1) \mu -\theta +L^2 (\theta -2 (\mu +1))+2\right)\right)^2> (L^2+1)((2 (\theta -1) \mu +\theta -2)^2+L^2 (\theta -2 (\mu +1))^2).
\end{aligned}
\end{equation*}
Reorganizing the last inequality gives us
\[
(\mu+1)K_\mathrm{num}>0,
\]
which proves $K_\mathrm{num}>0$.

Let us view $K_\mathrm{den}$ as a quadratic function of $\theta$.
From (ii) of case (c) of Section~\ref{ss:lip-mon-upper}, 
we have
%From $(\mu,L,\theta)\in R_\mathrm{(c)}$ (see point (ii) in the upper bound part) we had
$(2-\theta ) \left(1-L^2\right)-2 \mu  \left(\theta +L^2-1\right)>0$, which we can reformulate as $\theta<\tfrac{2 (\mu +1) \left(1-L^2\right)}{1 -L^2+2 \mu}$ using to $L<1$.
Basic computation gives us
\begin{equation*}
\begin{aligned}
& K_\mathrm{den}\Big|_{\theta=0}=4 (\mu +1) \left(\mu ^2 \left(L^2-1\right)^2+\left(L^2+1\right)^2\right)>0,\\
&K_\mathrm{den}\Big|_{\theta=\tfrac{2 (\mu +1) \left(1-L^2\right)}{1 -L^2+2 \mu}}=\tfrac{16 \mu ^2 (\mu +1) L^2 \left(L^2+1\right)^2}{\left(1-L^2+2 \mu\right)^2}>0,\\
&\tfrac{dK_\mathrm{den}}{d\theta}\Big|_{\theta=\tfrac{2 (\mu +1) \left(1-L^2\right)}{1 -L^2+2 \mu}}=-\tfrac{4 \mu  \left(L^2+1\right)^2 \left(2 \mu +L^2+1\right)}{1 -L^2+2 \mu}<0.
\end{aligned}
\end{equation*}
If the quadratic term of $K_\mathrm{den}$ is negative, $K_\mathrm{den}$ is positive when $\theta$ is between the roots, and the interval $\theta\in \left(0,\tfrac{2 (\mu +1) \left(1-L^2\right)}{1 -L^2+2 \mu}\right)$ lies between the roots.
If the quadratic term of $K_\mathrm{den}$ is nonnegative, then the roots (if they exist) would lie in $(\tfrac{2 (\mu +1) \left(1-L^2\right)}{1 -L^2+2 \mu},\infty)$ because of the sign of the derivative.
Therefore we conclude $K_\mathrm{den}>0$ in both cases.

\paragraph{Part II}
We now show $K=\tfrac{L^2+1}{2L}\tfrac{K_\mathrm{num}}{K_\mathrm{den}}\leq 1$.
Since $K_\mathrm{den}>0$, we can equivalently show
\begin{equation}\label{eq:Lips_tightness_nonpositivity}
    \begin{aligned}
    \left(L^2+1\right)K_\mathrm{num}-2LK_\mathrm{den}=a_1a_2\le 0
    %&=\left(L^2+1\right) \left((\mu -1) \left((\theta -2) \left(L^2-1\right)-2 \mu  \left(\theta +L^2-1\right)\right)^2-4 (\theta -2)^2 (\mu +1) L^2\right)\\ &\quad -2 L \left(4 \mu ^2 \left(\theta +L^2-1\right) \left((\theta -1) \mu +L^2 (-\theta +\mu +1)-1\right)+(\theta -2)^2 (\mu +1) \left(L^2+1\right)^2\right)
    \end{aligned}
\end{equation}
with
\begin{equation*}
    \begin{aligned}
a_1&=2 \mu  (1-L) \left(\theta +L^2-1\right)-(2-\theta) (L+1) \left(L^2+1\right)\\
a_2&=2 \mu ^2 (1-L) \left(\theta +L^2-1\right)+(2-\theta ) (L+1) \left(L^2+1\right)+\theta  \mu  \left((L-1)^3-4 L\right)+4 \mu  L (L+1).
    \end{aligned}
\end{equation*}

Remember that in Section~\ref{ss:lip-mon-upper} case (c) part (ii) we had shown
 $2 \mu  \left(\theta +L^2-1\right)<(2-\theta ) \left(1-L^2\right)$.
%when $(\mu,L,\theta)\notin R_\mathrm{(a)}$.
Plugging this inequality into the first term of $a_1$ gives us
\begin{equation*}
    \begin{aligned}
    a_1&\leq (2-\theta ) (1-L) \left(1-L^2\right)-(2-\theta ) (L+1) \left(L^2+1\right)=-2 (2-\theta) L (L+1)<0.
    \end{aligned}
\end{equation*}

Finally, we show $a_2\ge 0$.
Remember that 
\[
(\mu,L,\theta)\in R_\mathrm{(b)}\quad\Leftrightarrow\quad\text{\eqref{eq:thm112b1}}\text{ and }\text{\eqref{eq:thm112b2}}\text{ and }\text{\eqref{eq:thm112b3}}
\]
so we divide $(\mu,L,\theta)\notin R_\mathrm{(b)}$ into the following three cases:
\begin{enumerate}
    \item [(i)] Case $\neg\text{\eqref{eq:thm112b1}}$. This case corresponds to $L\ge 1$, but this cannot happen as $(\mu,L,\theta)\notin R_\mathrm{(a)}$ implies $L<1$.
    \item[(ii)] Case $\text{\eqref{eq:thm112b1}}\text{ and }\neg\text{\eqref{eq:thm112b2}}$.
    This case corresponds to $\mu \leq  \frac{L^2+1}{(L-1)^2}$.
    Plug $ L^2+1\geq \mu (L-1)^2$ into the second term of $a_2$ to get
\begin{equation*}
    \begin{aligned}
    a_2&\geq 2 \mu  \left(\mu  (1-L) \left(\theta +L^2-1\right)+\left(L^2+1\right) (1+L-\theta )\right)\\
    &\geq \mu  (1-L) \left(\theta +L^2-1\right)+\left(L^2+1\right) (1+L-\theta ),
    \end{aligned}
\end{equation*}
which we further split in two cases: either $ 1+L-\theta \geq 0$ and we use $ L^2+1\geq \mu (L-1)^2$ again to get
\[a_2\geq  \theta  \mu  L(1-L) >0,\]
or $1+L-\theta<0$ and we use $0<\theta-L-1\leq \theta+L^2-1$ and hence from the previous inequality on $a_2$ and $\mu>1$:
\begin{equation*}
    \begin{aligned}
    a_2
    %\mu  (1-L) \left(\theta +L^2-1\right)+\left(L^2+1\right) (1+L-\theta )+(1+L)(\theta-2)(L^2(\mu-1)+\mu-1\\
    %&\geq  \mu  (1-L) \left(\theta +L^2-1\right)+\left(L^2+1\right) (1+L-\theta )\\
    &\geq   (1-L) \left(\theta +L^2-1\right)+\left(L^2+1\right) (1+L-\theta )\\
    &=  (2-\theta) L (L+1)\\
    &> 0,
    \end{aligned}
\end{equation*}
where the last inequality follows from $\theta<2$.
\item[(iii)]
Case $\text{\eqref{eq:thm112b1}}\text{ and }\text{\eqref{eq:thm112b2}}\text{ and }\neg\text{\eqref{eq:thm112b3}}$.
Remember that in Section~\ref{ss:lip-mon-upper} case (b) part (ii), we have shown that \eqref{eq:thm112b1} and \eqref{eq:thm112b2} implies
\[
2 \mu ^2-\mu +\mu  L^3-L^3-3 \mu  L^2-L^2-2 \mu ^2 L-\mu  L-L-1>0.
\]
So there is no division by $0$ in $\neg\text{\eqref{eq:thm112b3}}$ and we have
    \[
    \theta (2 \mu ^2-\mu +\mu  L^3-L^3-3 \mu  L^2-L^2-2 \mu ^2 L-\mu  L-L-1)>2 (\mu +1) (L+1) \left(\mu +\mu  L^2-L^2-2 \mu  L-1\right).
    \]
    We conclude by noting
    \begin{align*}
    a_2&=\theta (2 \mu ^2-\mu +\mu  L^3-L^3-3 \mu  L^2-L^2-2 \mu ^2 L-\mu  L-L-1)-2 (\mu +1) (L+1) \left(\mu +\mu  L^2-L^2-2 \mu  L-1\right)\\
    &\geq0.
    \end{align*}

\section{Visualization}
\label{s:visualize}
In this section, we visualize the cases for Theorems~\ref{thm:DRS_coco_strmonotone} and \ref{thm:DRS_Lipschitz_strmonotone} and the contraction factors for Corollaries~\ref{cor:DRS_coco_strmonotone_no_relax} and \ref{cor:DRS_Lipschitz_strmonotone_no_relax}.

\begin{figure}
\begin{center}
\begin{subfigure}{.4\textwidth}
\includegraphics[width=1\textwidth]{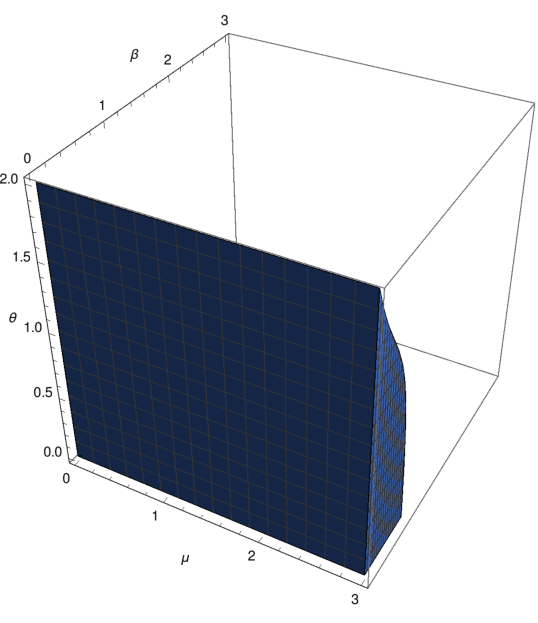}
\caption{Case (a)}
\end{subfigure}
\begin{subfigure}{.4\textwidth}
\includegraphics[width=1\textwidth]{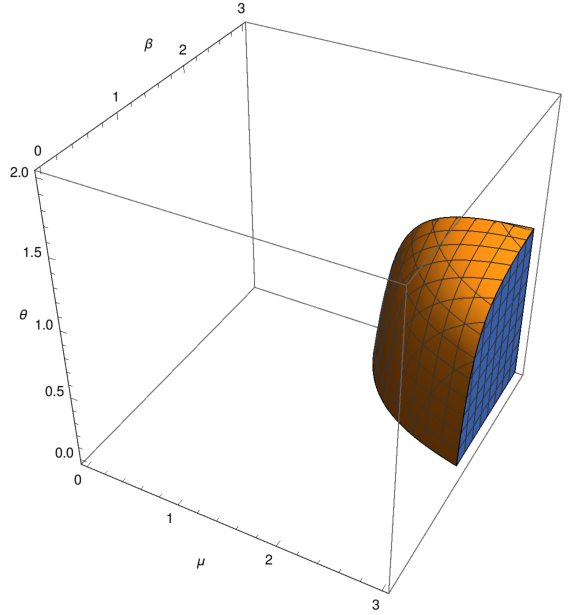}
\caption{Case (b)}
\end{subfigure}\\
\begin{subfigure}{.4\textwidth}
\includegraphics[width=1\textwidth]{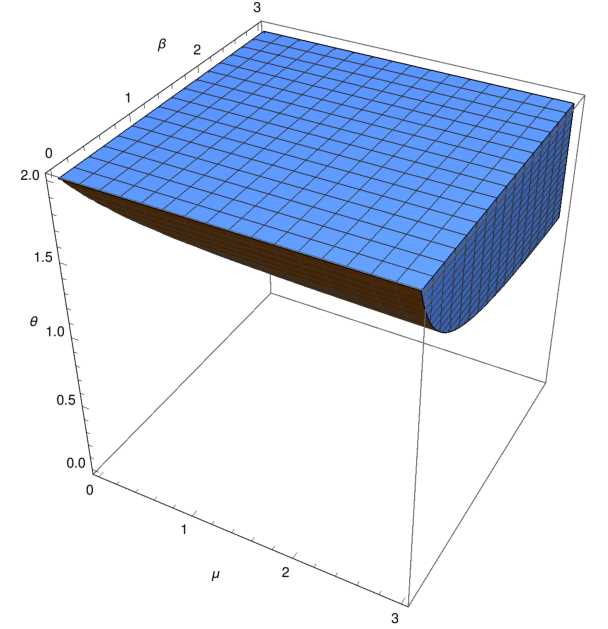}
\caption{Case (c)}
\end{subfigure}
\begin{subfigure}{.4\textwidth}
\includegraphics[width=1\textwidth]{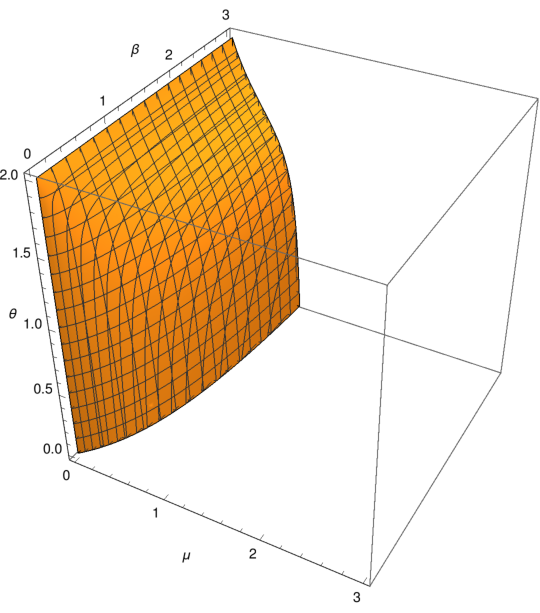}
\caption{Case (d)}
\end{subfigure}\\
\begin{subfigure}{.4\textwidth}
\includegraphics[width=1\textwidth]{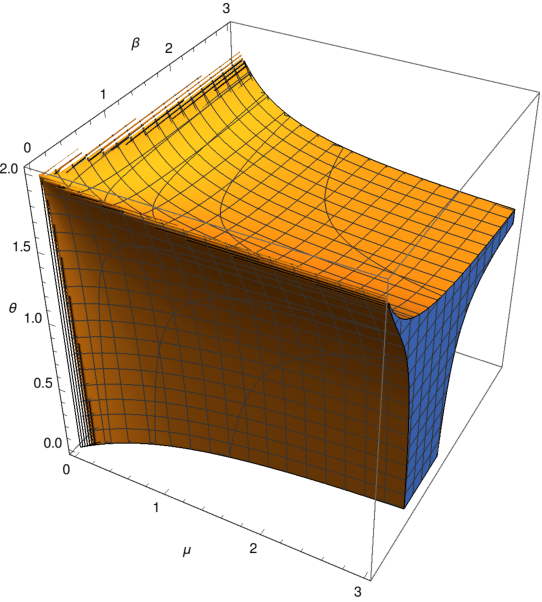}
\caption{Case (e)}
\end{subfigure}
\end{center}
\caption{Parameter regions for the 5 cases of Theorem~\ref{thm:DRS_coco_strmonotone} in the $\mu$-$\beta$-$\theta$ plane.}
\end{figure}

\begin{figure}
\begin{center}
\begin{subfigure}{.48\textwidth}
\includegraphics[width=1\textwidth]{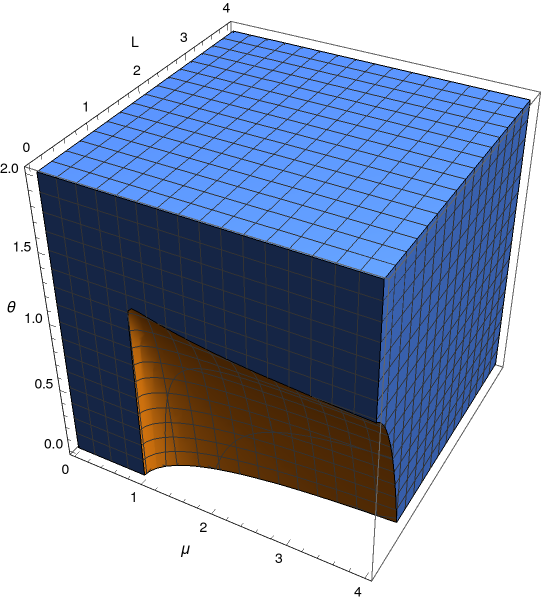}
\caption{Case (a)}
\end{subfigure}
\begin{subfigure}{.48\textwidth}
\includegraphics[width=1\textwidth]{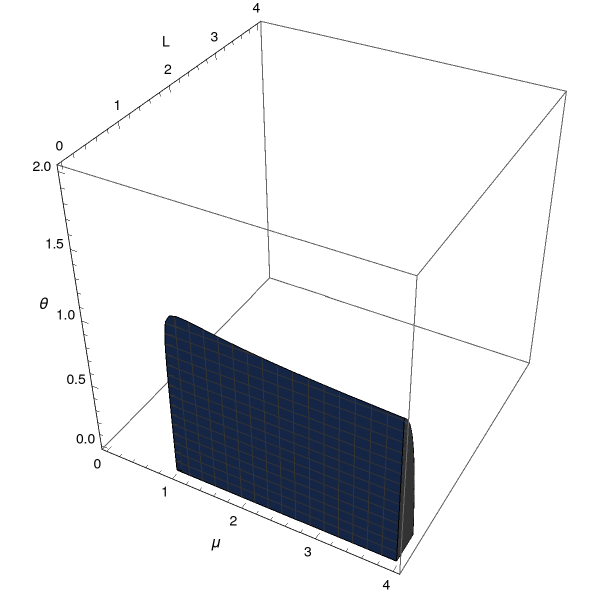}
\caption{Case (b)}
\end{subfigure}\\
\begin{subfigure}{.48\textwidth}
\includegraphics[width=1\textwidth]{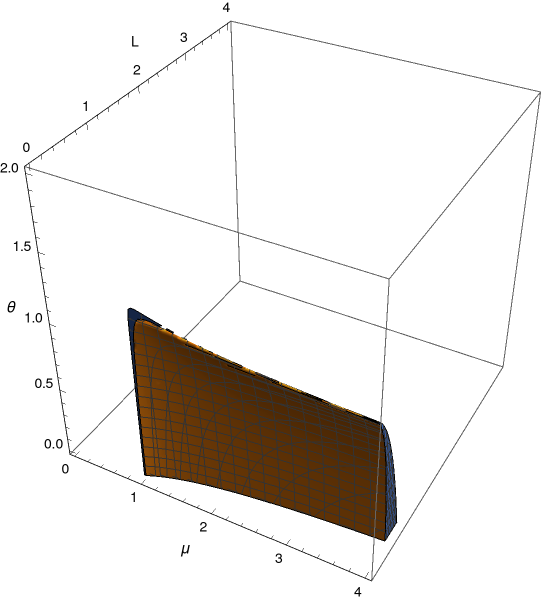}
\caption{Case (c)}
\end{subfigure}
\end{center}
\caption{Parameter regions for the 3 cases of Theorem~\ref{thm:DRS_Lipschitz_strmonotone} in the $\mu$-$L$-$\theta$ plane.}
\end{figure}

\begin{figure}
\begin{center}
\includegraphics[width=0.7\textwidth]{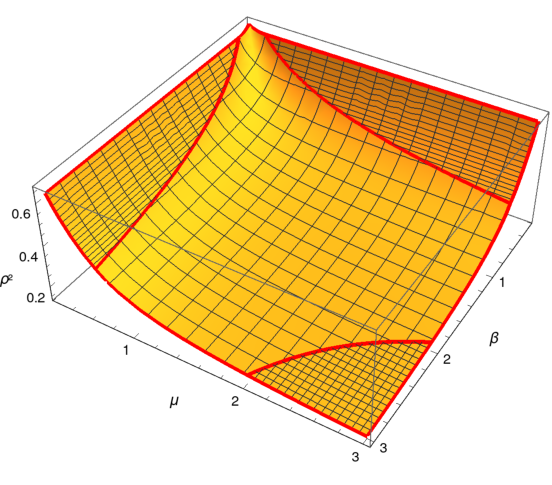}
\end{center}
\caption{Contraction factor for Corollary~\ref{cor:DRS_coco_strmonotone_no_relax} (when $\theta=1$)
in the $\mu$-$\beta$ plane.}
\end{figure}

\begin{figure}
\begin{center}
\includegraphics[width=0.7\textwidth]{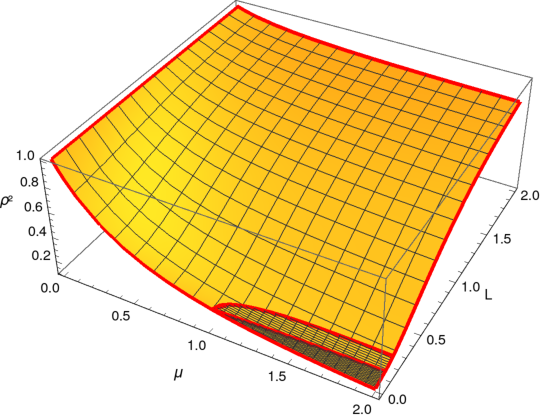}
\end{center}
\caption{Contraction factor for Corollary~\ref{cor:DRS_Lipschitz_strmonotone_no_relax} (when $\theta=1$)
in the $\mu$-$L$ plane.}
\end{figure}

\end{enumerate}

\end{document}